\documentclass[twoside,11pt]{article}

\usepackage{jmlr2e}

\usepackage{graphicx,color}
\usepackage{amsfonts}
\usepackage{bbm}
\usepackage{amsmath,euscript,mathrsfs}
\usepackage{tikz}
\usepackage{hyperref}
\newtheorem{appxthm}{Theorem}[section]
\newtheorem{appxlem}[appxthm]{Lemma}
\newtheorem{appxrem}[appxthm]{Remark}
\newtheorem{appxpro}[appxthm]{Proposition}

\newcommand{\R}{\mathbb{R}}

\newcommand{\E}{\mathbb{E}}

\newcommand{\Hyp}{\mathcal{H}}

\newcommand{\bb}{\mathbb}
\newcommand{\eu}{\EuScript}

\newcommand{\Cal}{\mathcal}

\newcommand{\intx}{\int_{\mathcal{X}}}
\newcommand{\intr}{\int_{\mathbb{R}^d}}
\ShortHeadings{Minimax Estimation of Kernel Mean Embeddings}{Tolstikhin, Sriperumbudur, and Muandet}
\firstpageno{1}

\begin{document}

\title{Minimax Estimation of Kernel Mean Embeddings}

\author{\name Ilya Tolstikhin \email ilya@tuebingen.mpg.de \\
       \addr $^\dagger$Department of Empirical Inference\\
       Max Planck Institute for Intelligent Systems\\
       Spemanstra{\ss}e 38, T{\"u}bingen 72076, Germany
       \AND
       \name Bharath K.\,Sriperumbudur \email bks18@psu.edu \\
       \addr Department of Statistics\\
       Pennsylvania State University\\
       University Park, PA 16802, USA
       \AND
       \name Krikamol Muandet$^\dagger$ \email krikamol@tuebingen.mpg.de \\
       \addr Department of Mathematics\\
       Faculty of Science, Mahidol University\\
       272 Rama VI Rd.~Rajchathevi, Bangkok 10400, Thailand
       }
       
\editor{}

\maketitle

\begin{abstract}
In this paper, we study the minimax estimation of the Bochner integral $$\mu_k(P):=\int_{\Cal{X}} k(\cdot,x)\,dP(x),$$ also called as the \emph{kernel mean embedding}, based on random samples drawn
i.i.d.~from $P$, where $k:\Cal{X}\times\Cal{X}\rightarrow\bb{R}$ is a positive definite kernel. Various estimators (including the empirical estimator), $\hat{\theta}_n$ of $\mu_k(P)$ 
are studied in the literature wherein all of them satisfy $\Vert \hat{\theta}_n-\mu_k(P)\Vert_{\Cal{H}_k}=O_P(n^{-1/2})$
with $\Cal{H}_k$ being the reproducing kernel Hilbert space induced by $k$. The main contribution of the paper is in showing that the above mentioned rate of $n^{-1/2}$ is minimax 
in $\Vert\cdot\Vert_{\Cal{H}_k}$ and $\Vert\cdot\Vert_{L^2(\bb{R}^d)}$-norms over the class of discrete measures and the class of measures that has an infinitely differentiable density, with
$k$ being a continuous translation-invariant kernel on $\bb{R}^d$. The interesting aspect of this result is that the minimax rate is independent of the smoothness of the kernel
and the density of $P$ (if it exists). 
\end{abstract}


\setlength{\parskip}{4pt}

\section{Introduction}\label{Sec:Introduction}
Over the last few years, kernel embedding of distributions \citep{Smola07Hilbert}, \citep{SGF+10} has gained a lot of attention in the machine learning community due to the wide variety of 
applications it has been employed in. Some of these applications include kernel two-sample testing~\citep{Gretton-06,Gretton12:KTT}, kernel independence and conditional independence 
tests~\citep{Gretton-08a,Fukumizu-08a}, covariate-shift~\citep{Smola07Hilbert}, density estimation~\citep{Sriperumbudur-11b}, feature selection~\citep{Song-12}, 
causal inference \citep{LMBT15}, kernel Bayes' rule~\citep{FSG13} and distribution regression~\citep{Szabo-15}.

Formally, let $\Cal{H}_k$ be a separable reproducing kernel Hilbert space (RKHS) 
\citep{Aronszajn-50} with a continuous reproducing kernel $k:\Cal{X}\times\Cal{X}\rightarrow\bb{R}$ defined on a separable topological space $\Cal{X}$.
Given a Borel probability measure $P$ defined over $\Cal{X}$ such that $\int_\Cal{X} \sqrt{k(x,x)}\,dP(x)<\infty$, the kernel mean or the mean element is defined as 
the Bochner integral 
\begin{equation}\label{Eq:kme}
\mu_P:=\intx k(\cdot,x)\,dP(x)\in\Cal{H}_k.
\end{equation}
We refer the reader to \citet[Chapter 2]{Diestel-77} and \citet[Chapter~1]{Dinculeanu:2000} for the definition of a Bochner integral. The mean element in 
(\ref{Eq:kme}) can be viewed as an embedding of $P$ in $\Cal{H}_k$, $$\mu_k:M^1_+(\Cal{X})\rightarrow \Cal{H}_k,\qquad\mu_k(P)=\mu_P,$$ where $M^1_+(\Cal{X})$ denotes the
set of all Borel probability measures on $\Cal{X}$. Hence, we also refer to $\mu_k$ as the \emph{kernel mean embedding} (KME). The mean embedding can be seen as a generalization of 
the classical kernel feature map that embeds points of an input space $\Cal{X}$ as elements in $\Cal{H}_k$. The mean embedding $\mu_k$
 can also be seen as a generalization of the classical notions
of characteristic function, moment generation function (if it exists), and Weierstrass transform of $P$ (all defined on $\bb{R}^d$) to an arbitrary topological space $\Cal{X}$ as the choice 
of $k(\cdot,x)$ as $(2\pi)^{-d/2}e^{-\sqrt{-1}\langle \cdot,x\rangle}$, 
$e^{\langle\cdot,x\rangle}$, and $(4\pi)^{-d/2}e^{-\Vert \cdot-x\Vert^2_2},\,x\in\bb{R}^d$ respectively reduces $\mu_k$ to these notions. The mean embedding $\mu_k$ is closely related to the
\emph{maximum mean discrepancy} (MMD) \citep{Gretton-06}, which is the RKHS distance between the mean embeddings of two probability measures. 
We refer the reader to \citep{SGF+10,S15} for more details on the properties of $\mu_k$ and the corresponding MMD.

In all the above mentioned statistical and machine learning applications, since the underlying distribution $P$ is 
known only through random samples $X_1,\ldots,X_n$ drawn i.i.d.~from it, an estimator of $\mu_P$ is employed. 
The goal of this paper is to study the minimax optimal estimation of $\mu_P$. 
In the literature, various estimators of $\mu_P$ have been proposed. The simplest and most popular is the empirical estimator $\mu_{P_n}$, which is constructed by replacing $P$ by its empirical
counterpart, $P_n:=\frac{1}{n}\sum^n_{i=1}\delta_{X_i}$, where $\delta_x$ denotes a Dirac measure at $x\in\Cal{X}$. 
In fact, all the above mentioned applications deal with the empirical estimator of $\mu_P$ because of its simplicity. Using Bernstein's inequality in separable 
Hilbert spaces \cite[Theorem 3.3.4]{Yurinksy-95}, it follows that for bounded continuous kernels, $\Vert \mu_{P_n}-\mu_P\Vert_{\Cal{H}_k}=O_P(n^{-1/2})$ for any $P$, i.e., the empirical
estimator is a $\sqrt{n}$-consistent estimator of $\mu_P$ in $\Cal{H}_k$-norm. This result is also proved in \citet[Theorem 2]{Smola07Hilbert}, 
\citet{Gretton12:KTT}, and \citet{LMBT15} using McDiarmid's inequality, which we 
improve in 
Proposition~\ref{thm:general-upper-bound} (also see Remark~\ref{rem:rkhs-bound} in Appendix~\ref{app:conc}) by providing a better constant.
Assuming $\Cal{X}=\bb{R}^d$ and $P$ to have a density $p$, \citet[Theorem 4.1]{S15} proposed to estimate
$\mu_P=\int_{\bb{R}^d} k(\cdot,x) p(x)\,dx$
by replacing $p$ with a kernel density estimator, which is then shown to be $\sqrt{n}$-consistent in $\Cal{H}_k$-norm if $k$ is a bounded continuous \emph{translation invariant kernel}---see 
Section~\ref{Sec:notation} for its definition---on $\bb{R}^d$. Recently, \citet[Section 2.4, Theorem 7]{MSF+15} proposed a non-parametric shrinkage estimator of $\mu_P$ and established its $\sqrt{n}$-consistency 
in $\Cal{H}_k$-norm for bounded continuous kernels on $\Cal{X}$. 
\citet[Section 3, Theorem 10]{MSF+15} also proposed a penalized M-estimator for $\mu_P$ where the penalization parameter is computed in a completely data-driven manner using leave-one-out cross validation and
showed that it is also $\sqrt{n}$-consistent in $\Cal{H}_k$-norm. In fact, the $\sqrt{n}$-consistency of all these estimators is established by showing that they are all within a 
$\Vert\cdot\Vert_{\Cal{H}_k}$-ball of size $o_P(n^{-1/2})$ around the empirical
estimator $\mu_{P_n}$.

In the above discussion, it is important to note that the convergence rate of $\mu_{P_n}$ (and also other estimators) to $\mu_P$ in 
$\Cal{H}_k$-norm does not depend on the smoothness of $k$ or the density, $p$ (if it exists). Under some mild conditions on the kernel (defined on $\bb{R}^d$), it can be shown (see Section~\ref{Sec:estimation-l2}) that
$\Cal{H}_k$ is continuously included in $L^2(\bb{R}^d)$ and $\Vert f\Vert_{L^2(\bb{R}^d)}\le c_k\Vert f\Vert_{\Cal{H}_k}$ for all $f\in\Cal{H}_k$, where $c_k$ is a constant that 
depends only on the kernel. This means, $\Vert \cdot\Vert_{L^2(\bb{R}^d)}$
is a weaker norm than $\Vert\cdot\Vert_{\Cal{H}_k}$ and therefore it could be possible that $\mu_{P_n}$ converges to $\mu_P$ in $L^2(\bb{R}^d)$ at a rate faster than $n^{-1/2}$ 
(depending on the smoothness of $k$). In Proposition~\ref{thm:general-upper-bound} (also see Remark~\ref{rem:l2-condition} in Appendix~\ref{app:conc}) we show that 
$\Vert \mu_{P_n}-\mu_P\Vert_{L^2(\bb{R}^d)}=O_P(n^{-1/2})$.
Now given these results, it is of interest to understand whether these rates are optimal in a minimax sense, i.e., whether the above mentioned estimators are minimax rate optimal or can they be improved upon? Therefore the goal of this work is to obtain 
minimax rates for the estimation of $\mu_P$ in $\Vert\cdot\Vert_{\Cal{H}_k}$ and~$\Vert\cdot\Vert_{L^2(\bb{R}^d)}$.

Formally,
we would like to find the minimax rate $r_{n,k}(\Cal{F},\Cal{P})$ and a positive constant $c_{k}(\Cal{F},\Cal{P})$ (independent of $n$) such that
\begin{equation}
\inf_{\hat{\theta}_n}
\sup_{P \in \mathcal{P}}
P^n
\left\{r^{-1}_{n,k}(\Cal{F},\Cal{P})\Vert \hat{\theta}_n-\mu_P\Vert_\Cal{F}\ge  c_{k}(\Cal{F},\Cal{P})\right\}>0,\label{Eq:minmax}\end{equation}
where $\Cal{F}$ is either $\Cal{H}_k$ or $L^2(\bb{R}^d)$, $\Cal{P}$ is a suitable subset of Borel probability measures on $\Cal{X}$,
and the infimum is taken over all estimators $\hat{\theta}_n$ mapping the i.i.d.~sample $X_1,\dots,X_n$ to~$\Cal{F}$. 
Suppose $k(x,y)=\langle x,y\rangle,\,x,y\in\R^d$.
Norms $\Vert\cdot\Vert_{\Cal{H}_k}$ and $\Vert \cdot\Vert_{L^2(\R^d)}$ match for this choice of $k$ and the corresponding RKHS is finite dimensional, i.e., $\Cal{H}_k=\R^d$. 
For a distribution $P$ on $\R^d$ satisfying $\int_{\R^d}\Vert x\Vert_2\,dP(x)<\infty$, this choice of 
kernel yields $\mu_P=\int x\,dP(x)$ as the mean embedding of $P$ which simply is the mean of $P$. It is well-known \citep[Chapter~5, Example 1.14]{LC08} that the minimax rate of 
estimating $\mu_P\in\R^d$ based on $(X_i)^n_{i=1}$ is $r_{n,k}(\Cal{F},\Cal{P})=n^{-1/2}$ for the class $\Cal{P}$ of Gaussian distributions on $\R^d$. In fact, this rate is attained by the empirical estimator $\mu_{P_n}=\frac{1}{n}\sum^n_{i=1}X_i$, which
is the sample mean. Based on this observation, while one can intuitively argue that the minimax rate of estimating $\mu_P$ is $n^{-1/2}$ even if $\Cal{H}_k$ is an infinite dimensional RKHS, it is difficult to extend the finite dimensional argument in a rigorous manner to the estimation of the infinite dimensional object, $\mu_P$. 
In this paper, through a key inequality---see \eqref{Eq:inequality-lbd}--- we rigorously show that it is indeed the case.

The main result of the paper is that if $k$ is translation invariant on $\Cal{X}=\bb{R}^d$ 
(see Theorems~\ref{thm:translation invariant-lower-bound-discrete} and \ref{thm:shift-invariant-lower-bound-discrete-l2} for precise conditions on the kernel) and $\Cal{P}$ is the 
set of all Borel discrete probability measures on $\bb{R}^d$, then the minimax rate $r_{n,k}(\Cal{F},\Cal{P})$ is $n^{-1/2}$ for both $\Cal{F}=\Cal{H}_k$ and $\Cal{F}=L^2(\bb{R}^d)$. 
Next, we show in Theorems~\ref{thm:translation invariant-lower-bound} and \ref{thm:shift-invariant-lower-bound-l2}
that the minimax rate for the estimation of $\mu_P$ in both 
$\Vert\cdot\Vert_{\Cal{H}_k}$ and $\Vert\cdot\Vert_{L^2(\bb{R}^d)}$ still remains $n^{-1/2}$ even when $\Cal{P}$ is restricted to the class of Borel probability measures 
which have densities, $p$ that are continuously 
infinitely differentiable. The reason for considering such a class of distributions with smooth densities is that $\mu_P$, which is the convolution of 
$k$ and $p$, is smoother than $k$. Therefore one might wonder if it could be possible to estimate $\mu_P$ at a rate faster than $n^{-1/2}$ that depends on the smoothness of 
$k$ and $p$. Our result establishes that even for the class of distributions with very smooth densities, the minimax rate is independent of the smoothness of $k$ and the density of 
$P$. The key ingredient in the proofs of Theorems~\ref{thm:translation invariant-lower-bound} and \ref{thm:shift-invariant-lower-bound-l2} is the non-trivial inequality 
(see Proposition~\ref{thm:rkhs-distance-lower-bound}) 
\begin{equation}\Vert \mu_{G_0}-\mu_{G_1}\Vert_\Cal{F}\ge c^\prime_{k,\sigma^2}\Vert \tau_0-\tau_1\Vert_2,\label{Eq:inequality-lbd}\end{equation}
which relates the $\Cal{F}$-distance between the mean embeddings of the Gaussian distributions, $G_0=N(\tau_0,\sigma^2 I)$ and $G_1=N(\tau_1,\sigma^2 I)$ to the Euclidean distance 
between the means of these Gaussians, where $c^\prime_{k,\sigma^2}$ is a constant that depends only on $\sigma^2$ and the translation invariant characteristic kernel $k$. 
Combining (\ref{Eq:inequality-lbd}) with Le Cam's method (see Appendix~\ref{app:sec:lecam}) 
implies that the estimation of an infinite dimensional object $\mu_P$ is 
as hard as the estimation of finite dimensional mean of a Gaussian distribution,
thereby establishing the minimax rate to be $n^{-1/2}$.
These results show that the empirical estimator---and other estimators we discussed above---of $\mu_P$ is minimax rate optimal.

\citet[Corollary 1]{RRP+15} derived a special case of \eqref{Eq:inequality-lbd} for the Gaussian kernel~$k$ by ignoring small terms in the Taylor series expansion of $\Vert \mu_{G_0}-\mu_{G_1}\Vert_{\Hyp_k}$ (refer to Remark~\ref{rem:gaussian-special}). 
They used this result to show that the MMD between $G_0$ and $G_1$ decreases to zero exponentially/polynomially fast in $d$ even when the Kullback-Leibler divergence between the two is kept constant, which in turn sheds some light on the decaying power of MMD-based hypothesis tests in high dimensions.
Proposition~\ref{thm:rkhs-distance-lower-bound} is more general, as it holds for \emph{any} translation-invariant kernel $k$ and does not require a truncation of small reminder terms.

The paper is organized as follows. Various notations used throughout the paper and definitions 
are collected in Section~\ref{Sec:notation}. The main results
on minimax estimation of $\mu_P$ in $\Vert\cdot\Vert_{\Cal{H}_k}$ and $\Vert\cdot\Vert_{L^2(\bb{R}^d)}$ for translation invariant kernels (and also \emph{radial} kernels) on $\bb{R}^d$ are presented in 
Sections~\ref{section:rkhs} and \ref{Sec:estimation-l2} respectively. The proofs of the results are provided in Section~\ref{Sec:proofs} while some supplementary results 
needed in the proofs are collected in appendices.
\section{Definitions \& Notation}\label{Sec:notation}
Define $\Vert
a\Vert_2:=\sqrt{\sum^d_{i=1}a^2_i}$ and $\langle a,b\rangle:=\sum^d_{i=1}a_ib_i$, where $a:=(a_1,\ldots,a_d)\in\bb{R}^d$ and $b:=(b_1,\ldots,b_d)\in\bb{R}^d$. 
$C(\R^d)$ (\emph{resp.} $C_b(\R^d)$) denotes the space of
all continuous (\emph{resp.}\:bounded continuous) functions on $\R^d$.
$f\in C(\R^d)$ is said to
\emph{vanish at infinity} if for every $\epsilon > 0$ the set $\{x :
|f(x)|\ge\epsilon\}$ is compact. The class of all continuous $f$ on $\R^d$
which vanish at infinity is denoted as $C_0(\R^d)$. For $f\in
C_b(\R^d)$, $\Vert f\Vert_\infty:=\sup_{x\in\R^d}|f(x)|$
denotes the supremum norm of $f$. $M_b(\R^d)$ (\emph{resp.} $M^b_+(\R^d)$) denotes the set of all finite (\emph{resp.} finite non-negative)
Borel measures on $\R^d$. $\text{supp}(\mu)$ denotes the support of $\mu\in M_b(\R^d)$ which is defined as 
$\text{supp}(\mu)=\{x\in\R^d\,\,\vert\,\, \text{for any open set}\,\,U\,\,\text{such that}\,\,x\in U,\,|\mu|(U)\ne 0\}$,
where $|\mu|$ is the total-variation of $\mu$. $M^1_+(\R^d)$ denotes the set of Borel probability measures on $\R^d$. For $\mu\in M^b_+(\R^d)$, $L^r(\R^d,\mu)$
denotes the Banach space of $r$-power ($r\ge
1$) $\mu$-integrable functions and we will use
$L^r(\R^d)$ for $L^r(\R^d,\mu)$ if $\mu$ is a Lebesgue measure on
$\R^d$. For $f\in L^r(\R^d,\mu)$, $\Vert
f\Vert_{L^r(\R^d,\mu)}:=\left(\int_{\R^d}|f|^r\,d\mu\right)^{1/r}$ denotes
the $L^r$-norm of $f$ for $1\le r<\infty$ and we denote it as
$\Vert\cdot\Vert_{L^r(\R^d)}$ if $\mu$ is the
Lebesgue measure.
The convolution $f\ast g$
of two measurable functions $f$ and $g$ on $\bb{R}^d$ is defined as $$(f\ast
g)(x):=\int_{\bb{R}^d}f(y)g(x-y)\,dy,$$ provided the integral exists for all
$x\in\bb{R}^d$. The Fourier transforms of $f\in
L^1(\bb{R}^d)$ and $\mu\in M_b(\bb{R}^d)$ are defined
as $$f^\wedge(y):=\mathscr{F}[f](y)=(2\pi)^{-d/2}\int_{\bb{R}^d}f(x)\,e^{-i\langle
y,x\rangle}\,dx,\,\qquad\,y\in\bb{R}^d$$
and $$\mu^\wedge(y):=\mathscr{F}[\mu](y)=(2\pi)^{-d/2}\int_{\bb{R}^d}e^{-i\langle
y,x\rangle}\,d\mu(x),\,\qquad\,y\in\bb{R}^d$$ respectively, where $i$ denotes the imaginary unit $\sqrt{-1}$.

A kernel $k\colon \R^d\times\R^d \to \R$ is called \emph{translation invariant} if there exists a symmetric positive definite function, $\psi$ such that
$k(x,y) = \psi(x - y)$ for all $x,y\in\R^d$. Bochner's theorem (see \citealp[Theorem 6.6]{W05}) provides a complete characterization for a positive definite function $\psi$: 
A continuous function $\psi\colon \R^d \to \R$ is positive definite if and only if it is the Fourier transform of $\Lambda_{\psi}\in M^b_+(\R^d)$, i.e., 
\begin{equation}
\psi(x) = \frac{1}{(2\pi)^{d/2}} \int_{\R^d} e^{-i \langle x, w\rangle} d \Lambda_{\psi}(w),\quad x\in\R^d.\label{Eq:Bochner}
\end{equation}
A kernel $k$ is called \emph{radial} if there exists $\phi:\bb{R}_+\rightarrow\bb{R}$ such that $k(x,y) = \phi(\|x-y\|^2_2)$ for all $x,y\in\R^d$.
From Sch\"onberg's representation (\citealp{S38}; \citealp[Theorems 7.13 \& 7.14]{W05}) it is known that a kernel $k$ is radial on every $\bb{R}^d$ if and only if there exists
$\nu\in M^b_+([0,\infty))$
such that the following holds for all $x,y\in\R^d$:
\begin{equation}
\label{eq:radial-equivalent}
k(x,y) = \phi(\|x-y\|^2) = \int_0^\infty e^{-t\|x-y\|^2}d\nu(t).
\end{equation}
Some examples of reproducing kernels on $\bb{R}^d$ (in fact all these are radial) that appear throughout the paper are: \vspace{-1mm}
\begin{enumerate}
\setlength\itemsep{.5mm}
\item
\emph{Gaussian:} 
$
k(x,y) = 
\exp\left(-\frac{\|x - y\|_2^2}{2\eta^2}\right),\,\eta>0;
$\vspace{-1mm}
\item
\emph{Mixture of Gaussians:} $k(x,y) = \sum_{i=1}^M \beta_i \exp\left( - \frac{\|x - y\|^2_2}{2\eta_i^2}\right)$, where $M\ge 2$, $\eta_1^2\geq \eta_2^2 \geq \dots \geq \eta^2_M > 0$,  and positive constants $\beta_1,\dots, \beta_M$ such that $\sum^M_{i=1} \beta_i =C_M<\infty$;
\vspace{-1mm}
\item
\emph{Inverse Multiquadrics:} $k(x,y) = (c^2 + \|x - y\|^2_2)^{-\gamma}$, $c,\gamma>0$; 
\vspace{-1mm}
\item
\emph{Mat\'ern:} $k(x,y) = 
\frac{c^{2\tau - d}}{\Gamma(\tau - \frac{d}{2})2^{\tau - 1 - d/2}}
\left(\frac{\|x - y\|_2}{c}\right)^{\tau - \frac{d}{2}}\eu{K}_{\frac{d}{2} - \tau}(c\|x-y\|_2),$ $\tau > d/2$, $c>0$, 
where $\eu{K}_{\alpha}$ is the \emph{modified Bessel function of the third kind} of order $\alpha$ and $\Gamma$ is the Gamma function. 
\end{enumerate}
A kernel $k$ is said to be \emph{characteristic} if the mean embedding, $\mu_k:P\rightarrow\mu_P$ is injective, 
where $\mu_P$ is defined in (\ref{Eq:kme}). 
Various characterizations for the injectivity of $\mu_k$ (or $k$ being 
characteristic) are known in literature (for details, see \citealp{SFL11} and references therein). If $k$ is a bounded continuous translation invariant positive definite kernel on 
$\bb{R}^d$, a simple characterization can be obtained for it to be characteristic \citep[Theorem 9]{SGF+10}: $k$ is characteristic if and only if $\text{supp}(\Lambda_\psi)=\bb{R}^d$ 
where $\Lambda_\psi$ is defined in (\ref{Eq:Bochner}). This characterization implies that the above mentioned examples are characteristic kernels.
Examples of non-characteristic kernels of translation invariant type include $k(x,y)=\frac{\sin(x-y)}{x-y},\,\,x,y\in\bb{R}$ and $k(	x,y)=\cos(x-y),\,\,x,y\in\bb{R}$. More generally, 
polynomial kernels of any finite order are non-characteristic.
\section{Minimax Estimation of $\mu_P$ in the RKHS Norm}
\label{section:rkhs}
In this section, we present our main results related to the minimax estimation of kernel mean embeddings (KMEs) in the RKHS norm. As discussed in Section~\ref{Sec:Introduction}, 
various estimators of $\mu_k(P)$ are known in literature and all these have a convergence rate of $n^{-1/2}$ if the kernel is bounded. 
The main goal of this section is to show that the rate $n^{-1/2}$ is actually minimax optimal for different choices of $\Cal{P}$ (see (\ref{Eq:minmax})) 
under some mild conditions on $k$. 

First, choosing $\Cal{P}$ to be the set of all discrete probability measures on $\bb{R}^d$, in Section~\ref{subsection:worst-case-rkhs} (see Theorem~\ref{thm:translation invariant-lower-bound-discrete} and Corollary~\ref{thm:radial-lower-bound-discrete}),
we present the minimax lower bounds of order $\Omega(n^{-1/2})$ with constant factors depending only on the properties of the kernel for translation invariant and radial kernels respectively.
Next we will show in Section~\ref{subsection:smooth-rkhs} that the rate $n^{-1/2}$ remains minimax optimal for translation invariant and radial kernels even if we choose
 the class 
$\mathcal{P}$ to contain only probability distributions with infinitely continuously differentiable densities. For translation invariant kernels the result 
(see Theorem~\ref{thm:translation invariant-lower-bound}) is based on a key inequality, which relates the RKHS distance between embeddings of Gaussian distributions to the 
Euclidean distance between the mean vectors of these distributions (see Proposition~\ref{thm:rkhs-distance-lower-bound}).
The minimax lower bound for radial kernels (see Theorem~\ref{thm:radial-lower-bound}) is derived using a slightly different argument.
Instead of applying the bound of Theorem~\ref{thm:translation invariant-lower-bound} to the particular case of radial kernels, we will present a direct analysis based on the special properties of radial kernels.
This will lead us to the lower bound with almost optimal constant factors, depending only on the shape of Borel measure $\nu$ corresponding to the kernel.

Our analysis is based on the following simple idea:
if a kernel $k$ is characteristic, there is a one-to-one correspondence between any given set of Borel probability measures $\mathcal{P}$ defined over $\R^d$ and a set $\mu_k(\mathcal{P})$ of their embeddings into the RKHS $\Hyp_k$.
This means that distributions in $\mathcal{P}$ are indexed by their embeddings $\Theta:=\mu_k(\mathcal{P})$ and so (\ref{Eq:minmax}) can be equivalently written as
\begin{equation}\inf_{\hat{\theta}_n}\sup_{\theta\in \Theta}\,\bb{P}_\theta\left\{r^{-1}_{n,k}(\Hyp_k, \mathcal{P})\Vert \hat{\theta}_n-\theta\Vert_{\Cal{H}_k}\ge c_k(\Hyp_k, \mathcal{P})\right\}>0,\label{Eq:equiv}\end{equation}
where the goal is to find the minimax rate $r_{n,k}(\Hyp_k, \mathcal{P})$ and a positive constant $c_k(\Hyp_k, \mathcal{P})$ (independent of $n$) such that (\ref{Eq:equiv}) holds and $\bb{P}_\theta=P^n$ when $\theta=\mu_k(P)$.
Using this equivalence, we obtain the minimax rates by employing Le Cam's method 
\citep{T08}---see Theorems \ref{thm:Tsybakov-two} and \ref{thm:Tsybakov} for a reference.
\subsection{Lower Bounds for Discrete Probability Measures}
\label{subsection:worst-case-rkhs}
The following result (proved in Section \ref{proof:rkhs-distance-lower-bound-discrete}) presents a minimax rate of $n^{-1/2}$ for estimating $\mu_k(P)$, where $k$ is assumed to be translation invariant on $\bb{R}^d$. 
%
\begin{theorem}[Translation invariant kernels]
\label{thm:translation invariant-lower-bound-discrete}
Let $\Cal{P}$ be the set of all Borel discrete probability measures on $\bb{R}^d$.
Suppose $k(x,y) = \psi(x - y)$, where $\psi\in C_b(\mathbb{R}^d)$ is positive definite and $k$ is characteristic.
Assume there exists $z\in\R^d$ and $\beta > 0$, such that  $\psi(0) - \psi(z) \geq \beta$.
Then the following holds:
\[
\inf_{\hat{\theta}_n}
\sup_{P\in\mathcal{P}}
P^n
\left\{
\| \hat{\theta}_n - \mu_k(P) \|_{\Hyp_{k}} \geq 
\frac{1}{6}\sqrt{\frac{2\beta}{n}}
\right\}
\geq
\frac{1}{4}.
\]
\end{theorem}
The result is based on Le Cam's method involving two hypotheses (see Theorem \ref{thm:Tsybakov-two}), where we choose them 
 to be KMEs of discrete measures, both supported on the same pair of points separated by $z$ in $\R^d$.\vspace{2mm}\\
\noindent \textbf{Remark (Choosing $z$ and $\beta$)}\hspace{2mm} \textit{As discussed in \citet[Section 3.4]{SGF+10}, if $k$ is translation invariant and characteristic on $\R^d$, then it is also strictly positive definite.
This means that $\psi(0) > 0$.
Moreover, the following hold: (a) Since $\psi$ is positive definite, we have $|\psi(x)| \leq \psi(0)$ for all $x\in \R^d$ and (b) since $\psi$ is characteristic, it cannot be a constant function.
Together these facts show that there always exist $z\in\R^d$ and $\beta > 0$ satisfying the assumptions of Theorem~\ref{thm:translation invariant-lower-bound-discrete}. 
For instance, a Gaussian kernel $k(x,v)=\exp\bigl(-\|x-v\|^2_2/(2\eta^2)\bigr)$ satisfies $\psi(0) - \psi(z) \geq \|z\|_2^2/(4\eta^2)$ if $\|z\|_2^2 \leq 2\eta^2$, where we used a simple fact that $1-e^{-x}\geq x/2$ for $0\leq x \leq 1$.}
\vspace{2mm}\\
While Theorem~\ref{thm:translation invariant-lower-bound-discrete} dealt with general translation invariant kernels, the following result (proved in Section \ref{proof:radial-lower-bound-discrete}) specializes it to radial kernels, i.e., kernels of the 
form in (\ref{eq:radial-equivalent}),
by providing a simple condition on $\nu$ under which Theorem~\ref{thm:translation invariant-lower-bound-discrete} holds. 
\begin{corollary}[Radial kernels]
\label{thm:radial-lower-bound-discrete}
Let $\Cal{P}$ be the set of all Borel discrete probability measures on $\bb{R}^d$
and $k$ be radial on $\bb{R}^d$, i.e., 
$k(x,y) = \psi_{\nu}(x-y) := \int_0^\infty e^{-t\|x - y\|_2^2} d\nu(t),$
where $\nu\in M^b_+([0,\infty))$ such that $\mathrm{supp}(\nu) \neq \{0\}$.
Assume there exist $0<t_1< \infty$ and $\alpha > 0$ satisfying 
$\nu([t_1,\infty)) \geq \alpha$. 
Then the following holds:
\[
\inf_{\hat{\theta}_n}
\sup_{P \in \mathcal{P}}
P^n
\left\{
\| \hat{\theta}_n - \mu_k(P) \|_{\Hyp_{k}} \geq 
\frac{1}{6}\sqrt{\frac{\alpha}{n}}
\right\}
\geq
\frac{1}{4}.
\]
\end{corollary}
\noindent \textbf{Remark (Choosing $t_1$ and $\alpha$)}\hspace{2mm} \textit{Since $\mathrm{supp}(\nu) \neq \{0\}$ the assumption of $\nu[t_1,\infty) \geq \alpha$ is always satisfied.
For instance, if $\nu$ is a probability measure with positive median $\eta$ then we can set $t_1 = \eta$ and $\alpha = \frac{1}{2}$. Based on this, it is easy to verify (see Appendix~\ref{alpha:cor-radial1}) that $\alpha=1$ for
Gaussian, $\alpha=C_M$ for mixture of Gaussian kernels, $\alpha=\frac{c^{-2\gamma}}{2}$ for inverse multiquadrics and $\alpha=\frac{1}{2}$ for Mat\'{e}rn kernels.}

\subsection{Lower Bounds for Probability Measures with Smooth Densities}
\label{subsection:smooth-rkhs}
So far, we have shown that the rate $n^{-1/2}$ is minimax optimal for the problem of KME estimation (both for translation invariant and radial kernels). As discussed in Section~\ref{Sec:Introduction},
since this rate is independent of the smoothness of the estimand (which is determined by the smoothness of the kernel), one might wonder whether 
the minimax rate can be improved by restricting $\Cal{P}$ to distributions with smooth densities. 
We show in this section
(see Theorems~\ref{thm:translation invariant-lower-bound} and \ref{thm:radial-lower-bound}) that this is not the case by restricting $\mathcal{P}$ to contain 
only distributions with infinitely continuously differentiable densities 
and proving the minimax lower bound of order $n^{-1/2}$.
%
%

We will start the analysis with translation invariant kernels and present a corresponding lower bound in Theorem \ref{thm:translation invariant-lower-bound}.
The proof of this result is again based on an application of Le Cam's method involving two hypotheses (see Theorem \ref{thm:Tsybakov-two}), where this time these hypotheses are 
chosen to be embeddings of the \hbox{$d$-dimensional} Gaussian distributions.
One of the main steps, when applying Theorem \ref{thm:Tsybakov-two}, is to lower bound the distance between these embeddings. 
This is done in the following result (proved in Section \ref{proof:rkhs-distance-lower-bound}), which essentially shows that if we take two Gaussian distributions $G(\mu_0,\sigma^2 I)$ and $G(\mu_1,\sigma^2 I)$ with the mean vectors $\mu_0,\mu_1\in\R^d$ which are close enough to each other, then the RKHS distance between the corresponding embeddings can be lower bounded by the Euclidean distance $\|\mu_0 - \mu_1\|_2$.
\vspace{-6mm}
\begin{proposition}
\label{thm:rkhs-distance-lower-bound}
Let $\sigma > 0$.
Suppose $k(x,y) = \psi(x - y)$, where $\psi\in C_b(\mathbb{R}^d)$ is positive definite and $k$ is characteristic. 
Then there exist constants ${\epsilon_{\psi,\sigma^2}, c_{\psi,\sigma^2} >0}$ depending only on $\psi$ and $\sigma^2$, 
such that the following condition holds for any $a\in\R^d$ with $\|a\|^2_2 \leq \epsilon_{\psi,\sigma^2}$:
\begin{equation}
\label{eq:psi-condition}
c_{\psi, \sigma^2}
\leq
\min_{e_z\in S^{d-1}}
\frac{2}{(2\pi)^{d/2}}
\int_{\R^d}
e^{ - \sigma^2\|w\|_2^2}
\langle e_z, w\rangle^2
\cos\left(
\langle a, w\rangle
\right)
d\Lambda_{\psi}(w)
<\infty,
\end{equation}
where $S^{d-1}$ is a unit sphere in $\R^d$
and $\Lambda_\psi\in M^b_+(\R^d)$ is defined in \eqref{Eq:Bochner}. 
Moreover, 
for all vectors $\mu_0,\mu_1\in\R^d$ satisfying $\Vert \mu_0 - \mu_1\Vert^2_2 \le \epsilon_{\psi,\sigma^2}$, the following holds:
\begin{equation}
\label{eq:rkhs-euclidean-rel}
\| \theta_0 - \theta_1\|_{\Hyp_{k}}
\geq
\sqrt{\frac{c_{\psi,\sigma^2}}{2}} \|\mu_0 - \mu_1\|_2,
\end{equation}
where $\theta_0$ and $\theta_1$ are KMEs of Gaussian measures $G(\mu_0,\sigma^2 I)$ and $G(\mu_1,\sigma^2 I)$ respectively.
\end{proposition}
\vspace{-5mm}
\begin{remark}[KME expands small distances]\label{rem:gaussian-special}
For a Gaussian kernel, it is possible to show (\citealp[Example 3]{Sriperumbudur-12}; \citealp[Proposition 1]{RRP+15}) that $\|\theta_0 - \theta_1\|^2_{\Hyp_k} = C_1 \bigl(1 - \exp(-C_2\|\mu_0 - \mu_1\|^2_2)\bigr)$, where $C_1$ and $C_2$ are positive constants that depend only on $\sigma^2$ and $\eta^2$.
This shows that~\eqref{eq:rkhs-euclidean-rel} holds for $\|\mu_0 - \mu_1\|_2\in [0, D]$, where $D$ satisfies $C_1 \bigl(1 - \exp(-C_2D^2)\bigr) = \frac{1}{2}D^2c_{\psi,\sigma^2}$.
In other words, Proposition~\ref{thm:rkhs-distance-lower-bound} states that the mapping $f_{\sigma^2}\colon \R^d \to \Hyp_k$ defined by $f_{\sigma^2}(x) := \mu_k\bigl(G(x,\sigma^2 I)\bigr)$ expands small distances.
\end{remark}
\vspace{-5mm}
\begin{remark}[Computing $c_{\psi, \sigma^2}$ and $\epsilon_{\psi, \sigma^2}$]\label{rem:compute-constant}
Generally it may be very hard to compute (or bound) the constants $c_{\psi, \sigma^2}$ and $\epsilon_{\psi, \sigma^2}$ appearing in the statement of Proposition~\ref{thm:rkhs-distance-lower-bound}. 
However, in some cases this may be still possible.
In Appendix \ref{appendix:alternative} we will provide an extensive analysis for the case of radial kernels.
\end{remark}
\vspace{-2mm}
Based on Proposition~\ref{thm:rkhs-distance-lower-bound}, the following result shows that the rate of $n^{-1/2}$ remains minimax optimal for the problem of KME estimation with translation invariant kernels, even if we restrict the class of distributions $\mathcal{P}$ to contain only measures with smooth densities.
\begin{theorem}[Translation invariant kernels]
\label{thm:translation invariant-lower-bound}
Let $\mathcal{P}$ be the set of distributions over $\R^d$ whose densities are continuously infinitely differentiable.
Suppose $k(x,y) = \psi(x - y)$, where $\psi\in C_b(\mathbb{R}^d)$ is positive definite and $k$ is characteristic. Define $c_\psi:=c_{\psi,1}$ and $\epsilon_\psi:=\epsilon_{\psi,1}$
where $c_{\psi,1}$ and $\epsilon_{\psi,1}$ are positive constants that satisfy \eqref{eq:psi-condition} in Proposition~\ref{thm:rkhs-distance-lower-bound}. 
Then for any $n \geq \frac{1}{\epsilon_{\psi}}$, the following holds:
\[
\inf_{\hat{\theta}_n}
\sup_{P\in\mathcal{P}}
P^n
\left\{
\| \hat{\theta}_n - \mu_k(P) \|_{\Hyp_{k}} \geq 
\frac{1}{2}\sqrt{\frac{ c_{\psi}}{2n}}
\right\}
\geq
\frac{1}{4}.
\]
\end{theorem}
\begin{proof}
The proof will be based on Theorem~\ref{thm:Tsybakov-two}.
For this we need to find two probability measures $P_0$ and $P_1$ on $\R^d$ and corresponding KMEs $\theta_0$ and $\theta_1$, such that $\|\theta_0 - \theta_1\|_{\Hyp_k}$ is 
of the order $\Omega(n^{-1/2})$, while $\mathrm{KL}(P_0^n\| P_1^n)$ is upper bounded by a constant independent of $n$. Here $\mathrm{KL}(P_0\|P_1)$ denotes the Kullback-Leibler divergence
between $P_0$ and $P_1$, which is defined as $\mathrm{KL}(P_0\|P_1)=\int \log\frac{dP_0}{dP_1}\,dP_0$ where $P_0$ is absolutely continuous w.r.t.~$P_1$.

%

Pick two Gaussian distributions $G_0:=G(\mu_0, \sigma^2 I)$ and $G_1:=G(\mu_1, \sigma^2 I)$ for $\mu_0,\mu_1\in\R^d$, and $\sigma^2 > 0$.
It is known that \cite[Section 2.4]{T08}
\begin{equation}
\label{eq:KL-distance}
\mathrm{KL}(G_0^n \| G_1^n) = n \cdot \frac{\|\mu_0 - \mu_1\|_2^2}{ 2\sigma^2},
\end{equation}
where $G_0^n$ and $G_1^n$ are $n$-fold product distributions. 
Choose $\mu_0$ and $\mu_1$ such that
$$
\|\mu_0 - \mu_1\|_2^2 = \frac{1}{n}.
$$
Denote KMEs of $G_0$ and $G_1$ using $\theta_0$ and $\theta_1$ respectively.
Next we will take $\sigma^2 = 1$ and apply Proposition~\ref{thm:rkhs-distance-lower-bound}. Since $c_\psi$ and $\epsilon_\psi$ satisfy \eqref{eq:psi-condition} in 
Proposition~\ref{thm:rkhs-distance-lower-bound}, it follows from Proposition~\ref{thm:rkhs-distance-lower-bound} that 
for ${1}/{n} \leq \epsilon_{\psi}$, 
\[
\|\theta_0 - \theta_1\|_{\Hyp_k}^2
\geq
\frac{c_{\psi}}{2} \|\mu_0 - \mu_1\|_2^2
=
\frac{ c_{\psi}}{2n}.
\]
This shows that the first condition of Theorem \ref{thm:Tsybakov-two} is satisfied for $\theta_0$ and $\theta_1$ with
$
s:= \frac{1}{2}\sqrt{ c_{\psi}/(2n)}.
$
Moreover, using \eqref{eq:KL-distance} we can show that the second condition of Theorem~\ref{thm:Tsybakov-two} is satisfied with $\alpha = \frac{1}{2}$.
We conclude the proof with an application of Theorem \ref{thm:Tsybakov-two}.\vspace{-6mm}
\end{proof}
\begin{remark}[Lower bound on the sample size $n$]
\label{remark:rkhs-lower-bound-n}
Note that Theorem~\ref{thm:translation invariant-lower-bound} holds only for large enough sample size $n$ (i.e., $n\ge 1/\epsilon_{\psi}$). 
This assumption on $n$ can be dropped if we set $\|\mu_0 - \mu_1\|_2^2 = \epsilon_{\psi} / n$ in the proof.
In this case, the lower bound $\frac{1}{2}\sqrt{{ c_{\psi}}/({2n})}$ will be replaced with $\frac{1}{2}\sqrt{{ c_{\psi} \epsilon_{\psi}}/({2n})}$, while the lower bound on the minimax probability $1/4$ will be replaced with 
$$
\max\left(
\frac{1}{4}e^{-\frac{\epsilon_{\psi}}{2}},
\frac{1 - \sqrt{\epsilon_{\psi} / 4}}{2}
\right).
$$
The latter is generally undesirable, especially if $\epsilon_{\psi}$ grows with $d\to\infty$, since we want the minimax probability to be lower bounded by some universal non-zero constant that
does not depend on the properties of the problem at hand.
\vspace{-2mm}
\end{remark}
Since radial kernels are particular instances of translation invariant kernels, Theorem~\ref{thm:translation invariant-lower-bound} can be specialized by explicitly 
computing the constants 
$c_{\psi}$ and $\epsilon_\psi$ to derive a minimax lower bound of order $\Omega(n^{-1/2})$.  
Unfortunately, the resulting lower bound will depend on the dimensionality~$d$ in a rather bad way and, as a consequence, is suboptimal in some situations.
For instance, if we consider a Gaussian kernel $k(x,y) = \exp\bigl(-\frac{1}{2\eta^2}\|x-y\|^2_2\bigr)$, then 
a straightforward computation of $c_{\psi}$ shows that
 the lower bound in Theorem~\ref{thm:translation invariant-lower-bound} has the form $\sqrt{(1 + 2/\eta^2)^{-d/2}/n}$ 
which shrinks to zero as $d\rightarrow\infty$,
while 
Proposition~\ref{thm:general-upper-bound} (also see Remark~\ref{rem:rkhs-bound}) provides a dimension independent upper bound of the order $O_p(n^{-1/2})$. 
Therefore, instead of specializing Theorem~\ref{thm:translation invariant-lower-bound} to radial kernels, we obtain the following result for radial kernels by using a refined analysis which yields
a minimax rate of $\Omega(n^{-1/2})$ that matches the upper bound of Proposition~\ref{thm:general-upper-bound} up to constant factors that depend only on the shape of Borel measure~$\nu$. In particular, 
when specialized to the Gaussian kernel, the result matches the upper bound up to a constant factor independent of $d$.
\begin{theorem}[Radial kernels]
\label{thm:radial-lower-bound}
Let $k$ be radial on $\bb{R}^d$, i.e.,
$
k(x,y) = 
\int_0^\infty e^{-t\|x - y\|_2^2}\, d\nu(t),
$ 
where $\nu\in M^b_+([0,\infty))$ 
and $\mathcal{P}$ be the set of distributions over $\R^d$ whose densities are continuously infinitely differentiable.
Assume that $\mathrm{supp}(\nu) \neq \{0\}$ and there exist $0 < t_0 \leq t_1 < \infty$, $0<\beta < \infty$ such that $\nu([t_0, t_1]) \geq \beta$. 
Then the following holds:
\[
\inf_{\hat{\theta}_n}
\sup_{P \in \mathcal{P}}
P^n
\left\{
\| \hat{\theta}_n - \mu_k(P) \|_{\Hyp_{k}} \geq 
\frac{1}{50}\sqrt{\frac{1}{n}\cdot \frac{\beta t_0}{t_1e}\left(1 - \frac{2}{2 + d}\right)}
\right\}
\geq
\frac{1}{5}.
\]
\end{theorem}
\begin{proof}
The proof, which is presented in Section \ref{proof:radial-lower-bound}, is based on an application of 
Le Cam's method involving multiple hypotheses (see Theorem~\ref{thm:Tsybakov}), where we use exponential (in~$d$) number of Gaussian distributions with variances decaying as $\frac{1}{d}$.
\end{proof}
\noindent \textbf{Remark (Non-trivial lower bound as $d\rightarrow\infty$)}\hspace{2mm} \textit{The proof of Theorem~\ref{thm:radial-lower-bound} is based on Gaussian distributions with variances decaying as $1/d$. As $d\rightarrow\infty$, it is obvious
the densities of these distributions do not have uniformly bounded Lipschitz constants, i.e., they are arbitrarily ``peaky''. Hence, if we choose $\Cal{P}$ to be class of
distributions with infinitely differentiable densities that have uniformly bounded Lipschitz constants, then as $d\rightarrow \infty$, the densities considered in the proof of 
Theorem~\ref{thm:radial-lower-bound} do not belong to $\Cal{P}$. On the other hand, the densities considered in the proof of Theorem~\ref{thm:translation invariant-lower-bound} still
belong to $\Cal{P}$ but yielding an uninteresting result since $c_\psi\rightarrow 0$ when $d\rightarrow\infty$. Therefore, it is an open question whether a non-trivial lower bound 
can be obtained for radial kernels (or any other translation invariant kernels) if we choose $\mathcal{P}$ to contain only distributions with densities having uniformly 
bounded Lipschitz constants.}
\vspace{2mm}\\
\noindent \textbf{Remark (Alternative Proof)}\hspace{2mm} \textit{For completeness, we also present an alternative proof of Theorem~\ref{thm:radial-lower-bound} in Appendix \ref{appendix:alternative}.
It is based on Proposition~\ref{thm:rkhs-distance-lower-bound}, which holds for any translation invariant kernel.
As a result, this proof leads to slightly worse constants compared to Theorem \ref{thm:radial-lower-bound} (where we used an analysis specific to radial kernels), as well 
as a superfluous condition on the minimal sample size $n$.}
\vspace{2mm}\\
In Appendix~\ref{beta:thm-radial1}, we compute the positive constant $B_k:=\frac{\beta t_0}{t_1}$ that appears in the lower bound in Theorem~\ref{thm:radial-lower-bound} 
in a closed form for Gaussian, mixture of Gaussian, inverse multiquadric and Mat\'{e}rn kernels.

\section{Minimax Estimation of $\mu_P$ in the $L^2(\R^d)$ Norm}
\label{Sec:estimation-l2}
So far, we have discussed the minimax estimation of the kernel mean embedding (KME) in the RKHS norm. In this section, we investigate the minimax estimation of KME in $L^2(\bb{R}^d)$ norm.
The reason for this investigation is as follows. Let $k(x,y)=\psi(x-y),\,x,y\in\bb{R}^d$, where $\psi\in L^1(\bb{R}^d)\cap C(\bb{R}^d)$ is strictly positive definite. 
The corresponding RKHS is given by (see \citealp[Theorem 10.12]{W05})
\begin{equation}\Cal{H}_k=\left\{f\in L^2(\bb{R}^d)\cap C(\bb{R}^d)\,:\,\int_{\bb{R}^d}\frac{\left|f^\wedge(\omega)\right|^2}{\psi^\wedge(\omega)}\,d\omega<\infty\right\},
\label{Eq:rkhs-rd}
\end{equation}
which is endowed with the inner product $\langle f,g\rangle_{\Cal{H}_k}=\int_{\bb{R}^d}\frac{f^\wedge(\omega)\overline{g^\wedge(\omega)}}{\psi^\wedge(\omega)}\,d\omega$ 
with $f^\wedge$ being the Fourier transform of $f$ in the $L^2$-sense. 
It follows from (\ref{Eq:rkhs-rd}) that for any $f\in\Cal{H}_k$,
\begin{equation}
\Vert f\Vert^2_{L^2(\bb{R}^d)}\stackrel{(\star)}{=}\Vert f^\wedge \Vert^2_{L^2(\bb{R}^d)}=\intr \left|f^\wedge(\omega)\right|^2\,d\omega=\intr \frac{\left|f^\wedge(\omega)\right|^2}{\psi^\wedge(\omega)} \psi^\wedge(\omega)\,d\omega
\stackrel{(\dagger)}{\le} \Vert \psi^\wedge\Vert_\infty \Vert f\Vert^2_{\Cal{H}_k}\stackrel{(\ddagger)}{<}\infty,
 \label{Eq:bnd}
\end{equation}
where $(\star)$ follows from Plancherel theorem \citep[Corollary 5.25]{W05}, $\Vert f\Vert_{\Cal{H}_k}$ is defined in~(\ref{Eq:rkhs-rd}), $(\dagger)$ follows from H\"{o}lder's inequality, and $(\ddagger)$ holds since $\psi^\wedge\in C_0(\bb{R}^d)$ (by Riemann-Lebesgue lemma, \citealp[Theorem 8.22]{F99}). Note that
$\psi^\wedge$ is non-negative \citep[Theorem~6.11]{W05} and so the inequality in~$(\dagger)$ is valid. It therefore follows from (\ref{Eq:bnd}) that
$\Cal{H}_k$ is continuously included in $L^2(\bb{R}^d)$ and $\Vert \cdot\Vert_{L^2(\bb{R}^d)}$
is a weaker norm than $\Vert\cdot\Vert_{\Cal{H}_k}$.\footnote{The continuous inclusion of $\Cal{H}_k$ in $L^2(\bb{R}^d)$ is known for Gaussian kernels on $\bb{R}^d$ (e.g., see \citealp[Lemma 11]{Vert-06}). 
Similar result is classical for Sobolev spaces in general (e.g., see \citealp[Section 9.3, p.\,302]{F99}) and particularly for those induced by Mat\'{e}rn kernels. \citet[Theorem 4.26]{SC08}
provides a general result for continuous inclusion of $\Cal{H}_k$ in $L^2(\mu)$ assuming $\intx \sqrt{k(x,x)}\,d\mu(x)<\infty$ where $\mu$ is a $\sigma$-finite measure. However, the result does not hold for translation invariant
kernels on $\bb{R}^d$ as the integrability condition is violated.} This means it is possible that the minimax rate of estimating $\mu_P$ in $\Vert \cdot\Vert_{L^2(\bb{R}^d)}$ 
could be faster than its RKHS counterpart with the rate possibly depending on the smoothness of $k$.
Hence, it is of interest to analyze the minimax rates of estimating $\mu_P$ in $\Vert\cdot\Vert_{L^2(\bb{R}^d)}$.
Interestingly, we show in this section that the minimax rate in the~$L^2$ setting is still $n^{-1/2}$.

The analysis in the $L^2$ setting follows ideas similar to those of the RKHS setting wherein, first, in Section~\ref{subsection:worst-case-l2}, we consider the minimax rate of 
estimating $\mu_P$ for translation invariant and radial kernels when $\Cal{P}$ is the set of all Borel discrete 
probability measures on $\bb{R}^d$ (see Theorem~\ref{thm:shift-invariant-lower-bound-discrete-l2} and Corollary~\ref{thm:radial-l2-lower-bound-discrete}). Next, in Section~\ref{subsection:smooth-l2},
we choose $\Cal{P}$ to be the set of all probability distributions that have
infinitely continuously differentiable densities and study the question of minimax rates for translation invariant (see Theorem \ref{thm:shift-invariant-lower-bound-l2}) and 
radial kernels (see Theorem~\ref{thm:radial-l2-lower-bound}). For both these choices of~$\Cal{P}$, we show that the rate is $n^{-1/2}$ irrespective of the smoothness of $k$. Exploiting
the injectivity of mean embedding for characteristic kernels (see the paragraph below and the paragraph around \eqref{Eq:equiv}), these results are derived using Le Cam's method (see Theorems~\ref{thm:Tsybakov-two} and \ref{thm:Tsybakov}).
Combined with Proposition~\ref{thm:general-upper-bound} (also see Remark~\ref{rem:l2-condition}), these results show that the empirical estimator, $\mu_{P_n}$ is minimax optimal.
Finally, in Section \ref{subsection:KDE} we discuss the relation between our results and some classical results of nonparametric density estimation, particularly, those of the kernel density estimator.

Before we proceed to the main results of this section, we briefly discuss the difference between estimation in RKHS and $L^2(\R^d)$ norms. Suppose $k(x,y)=\psi(x-y),\,x,y\in\bb{R}^d$
where $\psi\in L^1(\bb{R}^d)\cap L^2(\bb{R}^d)\cap C(\bb{R}^d)$ is positive definite and characteristic. It is easy to verify that $\mu_P\in L^1(\bb{R}^d)\cap L^2(\bb{R}^d)$. Since $\mu_P=\psi\ast P$, (\ref{Eq:rkhs-rd}) implies
\begin{equation}
\Vert \mu_P\Vert^2_{\Cal{H}_k}=\intr \frac{|(\psi\ast P)^\wedge|^2}{\psi^\wedge(\omega)}\,d\omega=\intr |\phi_P(\omega)|^2\psi^\wedge(\omega)\,d\omega=\Vert \phi_P\Vert^2_{L^2(\bb{R}^d,\psi^\wedge)}
\label{Eq:rkhs-charac} 
\end{equation}
whereas
\begin{equation}\Vert \mu_P\Vert^2_{L^2(\bb{R}^d)}\stackrel{(\star)}{=}\intr |\mu^\wedge_P(\omega)|^2\,d\omega=\intr |\phi_P(\omega)|^2(\psi^\wedge)^2(\omega)\,d\omega=\Vert \phi_P\Vert^2_{L^2(\bb{R}^d,(\psi^\wedge)^2)},\label{Eq:l2-charac}\end{equation}
where $\phi_P(\omega):=\int e^{-i\omega^Tx}\,dP(x)$ is the characteristic function of $P$ and $(\star)$ follows from Plancherel's theorem. It follows from (\ref{Eq:rkhs-charac}) and (\ref{Eq:l2-charac})
that the RKHS norm emphasizes the high frequencies of~$\phi_P$ compared to that of the $L^2$-norm. Since $\psi$ is characteristic, i.e., $P\mapsto \mu_k(P)\in \Cal{H}_k$ is injective, which is guaranteed if and only if $\text{supp}(\psi^\wedge)=\bb{R}^d$
\citep[Theorem 9]{SGF+10}, it follows from (\ref{Eq:l2-charac}) that $P\mapsto \mu_k(P)\in L^1(\bb{R}^d)\cap L^2(\bb{R}^d)$ is injective.
Therefore (\ref{Eq:minmax}) can be equivalently
written as (\ref{Eq:equiv}) by replacing $\Vert\cdot\Vert_{\Cal{H}_k}$ with $\Vert\cdot\Vert_{L^2(\bb{R}^d)}$ (see the discussion around \eqref{Eq:equiv}) and we obtain minimax rates by employing Le Cam's method as we did in the 
previous section.
\subsection{Lower Bounds for Discrete Probability Measures}
\label{subsection:worst-case-l2}
The following result (proved in Section \ref{proof:shift-invariant-lower-bound-discrete-l2}) for translation invariant kernels is based on an application of Le Cam's method involving two hypotheses 
(see Theorem \ref{thm:Tsybakov-two}), where we choose them to be KMEs of discrete measures, both supported on the same pair of points separated by a vector $z$ in $\R^d$. 
\begin{theorem}[Translation invariant kernels]
\label{thm:shift-invariant-lower-bound-discrete-l2}
Let $\mathcal{P}$ be the set of all Borel discrete probability measures on $\bb{R}^d$. Suppose $k(x,y)=\psi(x-y),\,x,y\in\bb{R}^d$ where $\psi\in L^2(\bb{R}^d)\cap C(\bb{R}^d)$ is positive definite 
and $k$ is characteristic.
Define
\begin{equation}
\label{eq:l2-shiftinvariant-expression-lower}
C^{\psi}_z := 
2\left(
\|\psi\|_{L^2(\bb{R}^d)}^2
-
\int_{\R^d}
\psi(y)\psi(y + z)
dy\right)
\end{equation}
for some $z\in\R^d\setminus \{0\}$. 
Then $C^\psi_z>0$ and 
\[
\inf_{\hat{\theta}_n}
\sup_{P \in \Cal{P}}
P^n
\left\{
\| \hat{\theta}_n - \mu_k(P) \|_{L^2(\bb{R}^d)} \geq 
\frac{1}{6}\sqrt{\frac{C_z^{\psi}}{n}}
\right\}
\geq
\frac{1}{4}.
\]
\end{theorem}
Using Cauchy-Schwartz inequality, the constant $C^\psi_z$ in Theorem~\ref{thm:shift-invariant-lower-bound-discrete-l2} can be shown (see the proof of Lemma~\ref{lemma:l2-discrete-closed} in Section~\ref{proof:shift-invariant-lower-bound-discrete-l2}) 
to be positive for every $z\in\bb{R}^d\backslash\{0\}$ if $k$ is characteristic, i.e., $\text{supp}(\Lambda_\psi)=\bb{R}^d$ (see \eqref{Eq:Bochner} for $\Lambda_\psi$).
The following result (proved in Section~\ref{proof:radial-l2-lower-bound-discrete}) specializes Theorem~\ref{thm:shift-invariant-lower-bound-discrete-l2} to radial kernels.
\begin{corollary}[Radial kernels]
\label{thm:radial-l2-lower-bound-discrete}
Let $\mathcal{P}$ be the set of all Borel discrete probability measures on $\bb{R}^d$ and $k$ be radial on $\bb{R}^d$, i.e.,
$
k(x,y) = \psi_{\nu}(x-y) := 
\int_0^\infty e^{-t\|x - y\|_2^2} d\nu(t),
$ 
where $\nu\in M^b_+([0,\infty))$ such that
$\emph{supp}(\nu) \neq \{0\}$ and 
\begin{equation}
\label{equation:worst-case-condition}
\int_0^{\infty}
t^{-d/2}
d\nu(t)<\infty.
\end{equation}
Assume that there exist $0<\delta_0\leq \delta_1 < \infty$ and $\beta > 0$ such that 
$
\nu([\delta_0,\delta_1]) \geq \beta.
$
Then the following holds:
\[
\inf_{\hat{\theta}_n}
\sup_{P \in \Cal{P}}
P^n
\left\{
\| \hat{\theta}_n - \mu_k(P) \|_{L^2(\bb{R}^d)} \geq 
\frac{\beta}{6}\sqrt{\frac{1}{n}
\left(\frac{\pi}{2\delta_1}\right)^{d/2}}
\right\}
\geq
\frac{1}{4}.
\]
\end{corollary}
In Corollary~\ref{thm:radial-l2-lower-bound-discrete}, since $\text{supp}(\nu)\ne\{0\}$, the assumption of $\nu([\delta_0,\delta_1])\ge \beta$ is always satisfied. In addition, the condition 
(\ref{equation:worst-case-condition}) on $\nu$ is satisfied by Gaussian, mixture of Gaussians, inverse multiquadric (while \eqref{equation:worst-case-condition} is satisfied for $\gamma>d/2$, the result in Corollary~\ref{thm:radial-l2-lower-bound-discrete} holds for $\gamma>d/4$)
and Mat\'{e}rn kernels---refer to Remark~\ref{rem:l2-condition} for more details.
Also, for these examples of kernels, the positive constant $A_k:=\beta^2\delta^{-d/2}_1$ in the lower bound in Corollary~\ref{thm:radial-l2-lower-bound-discrete} can be computed 
in a closed form (see Appendix~\ref{Ak:cor-radial} for details).

%
%

\subsection{Lower Bounds for Probability Measures with Smooth Densities}
\label{subsection:smooth-l2}
Next, as we did in Section~\ref{subsection:smooth-rkhs}, we choose $\mathcal{P}$ to be the set of all probability measures that have infinitely continuously differentiable densities and show that the minimax rate of estimating 
$\mu_P$ in $L^2$-norm for translation invariant (see Theorem~\ref{thm:shift-invariant-lower-bound-l2}) and radial kernels (see Theorem~\ref{thm:radial-l2-lower-bound}) is $n^{-1/2}$. The proof 
of these results are again based on an application of Le Cam's method involving two (see Theorem \ref{thm:Tsybakov-two}) and multiple hypotheses (see Theorem \ref{thm:Tsybakov}), 
where these hypotheses are 
chosen to be embeddings of the \hbox{$d$-dimensional} Gaussian distributions. As in Section~\ref{subsection:smooth-rkhs}, the results of this section are based on the following 
result (proved in Section \ref{proof:l2-distance-lower-bound}), which is conceptually similar to that of Proposition~\ref{thm:rkhs-distance-lower-bound}. 
\begin{proposition}
\label{thm:l2-distance-lower-bound}
Let $\sigma > 0$. Suppose $k(x,y) = \psi(x - y)$, where $\psi\in L^1(\bb{R}^d)\cap C_b(\mathbb{R}^d)$ is positive definite 
and $k$ is characteristic. Then there exist constants ${\epsilon_{\psi,\sigma^2}, c_{\psi,\sigma^2} >0}$ depending only on $\psi$ and $\sigma^2$, 
such that the following condition holds for any $a\in\R^d$ with $\|a\|^2_2 \leq \epsilon_{\psi,\sigma^2}$:
\begin{equation}
\label{eq:psi-condition-l2}
c_{\psi, \sigma^2} \leq \min_{e_z\in S^{d-1}}
2
\int_{\R^d}
e^{ - \sigma^2\|w\|_2^2}
\langle e_z, w\rangle^2
\cos\left(
\langle a, w\rangle
\right)
\bigl(\psi^\wedge(w)\bigr)^2
dw
< \infty,
\end{equation}
where $S^{d-1}$ is a unit sphere in $\R^d$.
Moreover, for all vectors $\mu_0,\mu_1\in\R^d$ satisfying $\Vert \mu_0-\mu_1\Vert^2_2\le\epsilon_{\psi,\sigma^2}$, the following holds:
\[
\| \theta_0 - \theta_1\|_{L^2(\bb{R}^d)}
\geq
\sqrt{\frac{c_{\psi,\sigma^2}}{2}} \|\mu_0 - \mu_1\|_2,
\]
where $\theta_0$ and $\theta_1$ are KMEs of the Gaussian measures $G(\mu_0,\sigma^2 I)$ and $G(\mu_1,\sigma^2 I)$ respectively.
\end{proposition}
%
The following result for translation invariant kernels is established using the above result wherein the proof is exactly the same as that of Theorem~\ref{thm:translation invariant-lower-bound} 
except for an application of Proposition~\ref{thm:l2-distance-lower-bound} in place of Proposition~\ref{thm:rkhs-distance-lower-bound}.
\begin{theorem}[Translation invariant kernels]
\label{thm:shift-invariant-lower-bound-l2}
Let $\mathcal{P}$ be the set of distributions over $\R^d$ whose densities are continuously infinitely differentiable.
Suppose $k(x,y) = \psi(x - y)$, where $\psi\in L^1(\bb{R}^d)\cap C_b(\mathbb{R}^d)$ is positive definite and $k$ is characteristic.
Define $c_\psi:=c_{\psi,1}$ and $\epsilon_\psi:=\epsilon_{\psi,1}$
where $c_{\psi,1}$ and $\epsilon_{\psi,1}$ are positive constants that satisfy \eqref{eq:psi-condition-l2} in Proposition~\ref{thm:l2-distance-lower-bound}. 
Then for any $n \geq \frac{1}{\epsilon_{\psi}}$, the following holds:
\[\inf_{\hat{\theta}_n}
\sup_{P\in\mathcal{P}}
P^n
\left\{
\| \hat{\theta}_n - \mu_k(P) \|_{L^2(\bb{R}^d)} \geq 
\frac{1}{2}\sqrt{\frac{ c_{\psi}}{2n}}
\right\}
\geq
\frac{1}{4}.
\]
\end{theorem}
As discussed in Remark~\ref{remark:rkhs-lower-bound-n}, it is possible to remove the requirement of minimal sample size in Theorem~\ref{thm:shift-invariant-lower-bound-l2}. Also,
as discussed in Remark~\ref{rem:compute-constant} and in the paragraph following Remark~\ref{remark:rkhs-lower-bound-n}, the constants $c_{\psi}$ and $\epsilon_{\psi}$ appearing in the bound
in Theorem~\ref{thm:shift-invariant-lower-bound-l2} are not only difficult to compute but also may depend on the dimensionality $d$ in a sup-optimal manner, particularly as ${d\rightarrow\infty}$.
Therefore, similar to what was done in Section \ref{subsection:smooth-rkhs}, we will not specialize Theorem \ref{thm:shift-invariant-lower-bound-l2} to radial kernels but instead 
present the following result (proved in Section \ref{proof:radial-l2-lower-bound} and the proof closely follows that of Theorem~\ref{thm:radial-lower-bound}), which is based on a direct analysis involving the properties of radial kernels. For the particular case of a Gaussian kernel,
this lower bound matches the upper bound of Proposition~\ref{thm:general-upper-bound} (also see Remark~\ref{rem:l2-condition}) up to a constant factor independent of $d$.
\begin{theorem}[Radial kernels]
\label{thm:radial-l2-lower-bound}
Let $k$ be radial on $\bb{R}^d$, i.e.,
$
k(x,y) = 
\int_0^\infty e^{-t\|x - y\|_2^2}\, d\nu(t),
$ 
where $\nu\in M^b_+([0,\infty))$ 
and $\mathcal{P}$ be the set of distributions over $\R^d$ whose densities are continuously infinitely differentiable. 
Assume that (\ref{equation:worst-case-condition}) holds, $\mathrm{supp}(\nu) \neq \{0\}$ and
there exist $0 < \delta_0 \leq \delta_1 < \infty$, $0<\beta < \infty$ such that $\nu([\delta_0, \delta_1]) \geq \beta$. 
Then the following holds:
\[
\inf_{\hat{\theta}_n}
\sup_{P \in \mathcal{P}}
P^n
\left\{
\| \hat{\theta}_n - \mu_k(P) \|_{L^2(\bb{R}^d)} \geq 
\frac{1}{50}\sqrt{\frac{1}{n} \left(\frac{\pi}{2\delta_1}\right)^{d/2} \frac{\beta^2 \delta_0}{\delta_1e}\left(1 - \frac{2}{2 + d}\right)}
\right\}
\geq
\frac{1}{5}.
\]
%
\end{theorem}
The constant $B_k:=\beta^2\delta_0\delta^{-\frac{d+2}{2}}_1$ in the lower bound in the above result can be computed 
in a closed form for Gaussian, mixture of Gaussian, inverse multiquadric, and Mat\'{e}rn kernels (see Appendix~\ref{Bk:thm-radial2} for details). 
The factor $(\pi/2)^{d/4}$ can be eliminated from the lower bound by considering a rescaled kernel $(\pi/2)^{-d/4} \psi(x-y)$.
Nevertheless, the bound will still depend on $d$ exponentially as captured by the constant $B_k$. This can be further overcome by using the normalized kernel 
$k(x,y) / \|\psi\|_{L^2(\R^d)}$. In the particular case of normalized Gaussian kernels $(\pi \eta^2)^{-d/2}\exp\bigl( -\frac{1}{2\eta^2} \|x-y\|^2_2\bigr)$ 
this will lead to dimension-free lower bounds.

\subsection{Relation to Kernel Density Estimation}
\label{subsection:KDE}
In this section, we discuss the relation between the estimation of $\mu_P$ and density estimation. The problem of density estimation deals with estimating an unknown density, $p$
based on random samples $(X_i)^n_{i=1}$ drawn i.i.d.~from it. One of the popular non-parametric methods for density estimation is kernel density estimation (KDE), where the estimator is of the form~\citep[Section 1.2]{T08}
\[
\hat{p}_n(x_1,\ldots,x_d) = \frac{1}{n\prod^d_{i=1}h_i}\sum_{i=1}^n K\left(\frac{X_{i,1} - x_1}{h_1}, \cdots, \frac{X_{i,d} - x_d}{h_d}\right).
\]
Here $K\colon \R^d \to \R$ is the \emph{smoothing kernel} (this kernel should not be confused with the reproducing kernel $k$ which we used throughout the paper), $h_1,\dots,h_d>0$ 
are bandwidths, and $X_{i,j}$ is the $j$-th coordinate of the $i$-th sample point.
Assuming $p\in L^2(\bb{R}^d)$, the consistency of $\hat{p}_n$ is usually studied in the sense of 
\emph{mean integrated squared error} (MISE) $\E\| \hat{p}_n - p\|^2_{L^2(\R^d)}$, which can be decomposed into variance and bias terms as:
\begin{equation}\E\| \hat{p}_n - p\|^2_{L^2(\R^d)}=\E\bigl\|\hat{p}_n-\E[\hat{p}_n]\bigr\|^2_{L^2(\R^d)} +\bigl\|p-\E[\hat{p}_n]\bigr\|^2_{L^2(\R^d)}.
\label{Eq:bias-variance}
\end{equation}
Assume $K$ to be bounded and $h_1=\cdots=h_d=h$. Define $K_h:=h^{-d}K(\cdot/h)$. Then for any fixed $x\in\R^d$,
$$\hat{p}_n(x)=\frac{1}{nh^d}\sum^n_{i=1}K\left(\frac{X_i - x}{h}\right)=\frac{1}{n}\sum^n_{i=1}K_h(X_i-x)=\intr K_h(z-x)\,dP_n(z)$$ and
\begin{align*}
\E[\hat{p}_n(x)] 
&=
\frac{1}{h^d} \int_{\R^d}  K\left(\frac{z - x}{h}\right) p(z) dz=(K_h\ast p)(x).
\end{align*}
This shows that $\hat{p}_n=\mu_{K_h}(P_n)$ and $\E[\hat{p}_n]=\mu_{K_h}(P)$ where $P$ is the distribution with $p$ as its density w.r.t.~the Lebesgue measure and $P_n$ is the empirical
measure constructed based on samples $(X_i)^n_{i=1}$ drawn from $p$. Therefore the results of Section~\ref{Sec:estimation-l2} (and more generally of this paper) are about the minimax 
rates for $\E[\hat{p}_n]$. 
However, note that $K_h$ need not be positive definite (and therefore need not be the reproducing kernel of some RKHS). On the other hand, $K$ has to be positive, i.e., 
$K(x)\ge 0,\,\forall\,x\in\bb{R}^d$ and normalized, i.e., $\int_{\bb{R}^d}K(x)\,dx=1$ to yield an estimator that is a valid density, unlike in kernel mean estimation where $k$ need not be positive nor normalized.   
The minimax rate of $n^{-1/2}$ for estimating $\E[\hat{p}_n]$ is achieved
by the kernel density estimator $\hat{p}_n$ (which is nothing but the empirical estimator of $\mu_{K_h}(P)$) as it is known (based on a straightforward generalization of \citealp[Proposition 1.4]{T08} for multiple dimensions) that 
\begin{equation}
\E
\bigl\|
\hat{p}_n
-
\E[\hat{p}_n]
\bigr\|^2_{L^2(\R^d)}
\leq
\frac{\|K\|_{L^2(\R^d)}^2}{n h^d},\nonumber
\end{equation}
where we assume $K\in L^2(\R^d)$. The bandwidth parameter $h$ is immaterial in the estimation of $\mu_{K_h}(P)$ and can be treated as a constant (independent of $n$) unlike in the 
problem of estimating $p$ where $h$ should decay to zero 
at an appropriate rate for the bias $\bigl\|p-\E[\hat{p}_n]\bigr\|_{L^2(\R^d)}$ to converge to zero as $n\rightarrow \infty$. In particular, if $p$ lies in a Sobolev space of
smoothness index $s$, then the bias-squared term in \eqref{Eq:bias-variance} behaves as $h^{2s}$, which combined with the above bound on the variance yields a rate of 
%
$n^{-\frac{2s}{2s + 1}}$ for $h=n^{-\frac{1}{2s+1}}$. 
This rate is 
known to be minimax optimal for the problem of estimating $p$ while our rates are minimax optimal for the problem of smoothed density estimation where the smoothing is carried out
by the kernel.

\section{Proofs}\label{Sec:proofs}
In this section we present all the missing proofs of results of Sections~\ref{section:rkhs} and \ref{Sec:estimation-l2}. 

\subsection{Proof of Theorem~\ref{thm:translation invariant-lower-bound-discrete}}
\label{proof:rkhs-distance-lower-bound-discrete}
Pick two discrete distributions
$P_0=p_0 \delta_x + (1-p_0) \delta_v$ and $P_1=p_1 \delta_x + (1-p_1) \delta_v$, where $x,v\in\R^d$, $0 < p_0 < 1$, $0 < p_1 < 1$ and $\delta_x$ denotes a Dirac measure supported at $x$.
Define $\theta_0=\mu_k(P_0)$ and $\theta_1=\mu_k(P_1)$. Since $\Vert\theta_0\Vert^2_{\Cal{H}_k}=\int\int k(x,y)\, dP_0(x)\,dP_0(y)$, which follows from the reproducing property of $k$,
it is easy to verify that
\begin{align*}
\|\theta_0 - \theta_1\|^2_{\Hyp_{k}}
=
\E[k(\xi,\xi')]
+
\E[k(\eta,\eta')]
-
2\,\E[k(\xi,\eta)],
\end{align*}
where $\xi$ and $\eta$ are random variables distributed according to $P_0$ and $P_1$ respectively, and $\xi'$ and $\eta'$ are independent copies of $\xi$ and $\eta$.
Since $k$ is translation invariant, we have $k(v,v) = k(x,x) = \psi(0)$ and $k(x,v) = k(v,x) = \psi(x - v)$, which imply
\begin{align}
\label{eq:proof-worst-rkhs}
\|\theta_0 - \theta_1\|^2_{\Hyp_{k}}
=
2(p_0 - p_1)^2\bigl(\psi(0) - \psi(x - v)\bigr).
\end{align}

Also note that 
\begin{align*}
\mathrm{KL}(P_0 \| P_1)
&=
p_0\log\frac{p_0}{p_1}
+
(1-p_0)\log\frac{1-p_0}{1-p_1}\\
&=
p_0\log\left(1 + \frac{p_0-p_1}{p_1}\right)
+
(1-p_0)\log\left(1 + \frac{p_1-p_0}{1-p_1}\right)\\
&\stackrel{(\ast)}{\leq}
\log\left\{ p_0\left(1 + \frac{p_0-p_1}{p_1}\right) + (1-p_0)\left(1 + \frac{p_1-p_0}{1-p_1}\right) \right\}\\
&=\log\left(1 + (p_0 - p_1)\left(\frac{p_0}{p_1} - \frac{1-p_0}{1-p_1}\right)  \right),
\end{align*}
where we used Jensen's inequality in $(\ast)$ for the logarithmic function, which is concave.
Next, using a simple inequality $\log(1 + x) \leq x$, which holds for all $x>-1$, we get
\begin{align*}
\mathrm{KL}(P_0 \| P_1)
&\leq
(p_0 - p_1)\left(\frac{p_0}{p_1} - \frac{1-p_0}{1-p_1}\right)
=
\frac{(p_0-p_1)^2}{p_1(1-p_1)}.
\end{align*}
Note that a maximal value of denominator is achieved when $p_1 = \frac{1}{2}$. Setting $p_1 = \frac{1}{2}$ we get the following upper bound:
$
\mathrm{KL}(P_0 \| P_1) \leq 4  \left(p_0-\frac{1}{2}\right)^2,
$
which when used in the chain rule of KL-divergence yields
\begin{equation}
\mathrm{KL}(P_0^n \| P_1^n) \leq 4n\left(p_0-\frac{1}{2}\right)^2.\nonumber
\end{equation}
Choosing $p_0$ such that $(p_0-\frac{1}{2})^2 = \frac{1}{9n}$ yields
$
\mathrm{KL}(P_0^n \| P_1^n) \leq \frac{4}{9}$
and
$
\|\theta_0 - \theta_1\|^2_{\Hyp_{k}}
=
\frac{2}{9n}\bigl(\psi(0) - \psi(x - v)\bigr).
$
Choose $x$ and $v$ in such a way that $x - v = z$, where $z\in\R^d$ is a point for which $\psi(0) - \psi(z) \geq \beta$ and $\beta > 0$. 
This yields
\[
\|\theta_0 - \theta_1\|^2_{\Hyp_{k}}
\geq
\frac{2\beta}{9n},
\]
which shows that the assumptions of Theorem \ref{thm:Tsybakov-two} are satisfied with
$
s:= \frac{1}{6}\sqrt{\frac{2\beta}{n}}
$
and ${\alpha:=\frac49}$.
The result follows from an application of Theorem \ref{thm:Tsybakov-two} by noticing that $\frac{1 - \sqrt{\alpha / 2}}{2} > 1/4$.\vspace{1mm}

\noindent \textbf{Remark (Measures with bounded support)}\hspace{2mm} \textit{It is evident from the above proof that exactly the same lower bound holds if we restrict $\mathcal{P}$ to contain only probability measures with bounded support.
We can proceed further and assume that for each $P\in\mathcal{P}$ the radius of $\mathrm{supp}(P)$ is upper bounded by some positive constant $R$. 
In this case the same reasoning will work as long as $\psi$ is not ``flat'' on the ball of radius $R$ centered around origin.} 

\subsection{Proof of Corollary \ref{thm:radial-lower-bound-discrete}}
\label{proof:radial-lower-bound-discrete}
The proof is based on application of Theorem \ref{thm:translation invariant-lower-bound-discrete}. Since $\text{supp}(\nu)\ne \{0\}$, it follows from \cite[Proposition 5]{SFL11} that $k$ is 
characteristic. We now show that there exist $z\in\R^d$ and $\beta > 0$, such that  $\psi_{\nu}(0) - \psi_{\nu}(z) \geq \beta$.
Note that for any $x\in\R^d$
\begin{align*}
\psi_{\nu}(0) - \psi_{\nu}(x )
&=
\int_{0}^{\infty}
\left( 1 - e^{-t \|x \|^2_2} \right)
d\nu(t)
\geq
\int_{t_1}^{\infty}
\left( 1 - e^{-t \|x \|^2_2} \right)
d\nu(t)\\
&\geq
\int_{t_1}^{\infty}
\left( 1 - e^{-t_1 \|x \|^2_2} \right)
d\nu(t)
=
\nu([t_1,\infty))\left( 1 - e^{-t_1 \|x \|^2_2} \right)\\
&\geq
\alpha\left( 1 - e^{-t_1 \|x \|^2_2} \right)
\geq
\frac{\alpha t_1}{2} \|x \|^2_2,
\end{align*}
where the last inequality holds whenever $\|x \|^2_2 \leq \frac{1}{t_1}$.
Choosing $z$ such that $\|z\|^2_2 = \frac{1}{t_1}$ yields
$
\psi_{\nu}(0) - \psi_{\nu}(z) \geq \frac{\alpha}{2}.
$
The result therefore follows from Theorem \ref{thm:translation invariant-lower-bound-discrete} by choosing $\beta = \frac{\alpha}{2}$.

\subsection{Proof of Proposition~\ref{thm:rkhs-distance-lower-bound}}
\label{proof:rkhs-distance-lower-bound}
Before we prove Proposition~\ref{thm:rkhs-distance-lower-bound}, first we will derive a closed form expression for the RKHS distance between KMEs of two $d$-dimensional Gaussian distributions with the kernel being translation invariant, i.e., $k(x,y) = \psi(x-y)$.
Throughout this section $\Lambda_{\psi}$ will denote a finite non-negative Borel measure corresponding to the positive-definite function $\psi$ from \eqref{Eq:Bochner}. 
\begin{lemma}
\label{lemma:rkhs-distance}
Let $\theta_0$ and $\theta_1$ be KME of Gaussian measures $G(\mu_0,\sigma^2 I)$ and $G(\mu_1,\sigma^2 I)$ for $\mu_0,\mu_1\in\R^d$ and $\sigma^2 > 0$.
Suppose $k(x,y) = \psi(x - y)$, where $\psi\in C_b(\mathbb{R}^d)$ is positive definite. Then
\begin{equation}
\| \theta_0 - \theta_1\|^2_{\Hyp_{k}}=
\frac{2}{(2\pi)^{d/2}}
\int_{\R^d}
e^{ - \sigma^2\|w\|_2^2}
\left(
1 - 
\cos\left(
\langle \mu_0 - \mu_1, w\rangle
\right)
\right)
d\Lambda_{\psi}(w).\label{Eq:pro-bnd}
\end{equation}
\end{lemma}

\begin{proof}
Note that
\begin{equation}
\label{eq:distance}
\| \theta_0 -  \theta_1\|_{\Hyp_{k}}^2 = \| \theta_0\|_{\Hyp_{k}}^2 + \| \theta_1\|_{\Hyp_{k}}^2 - 2\langle \theta_0,\theta_1\rangle_{\Hyp_{k}},
\end{equation}
where
$
\langle \theta_0, \theta_1 \rangle_{\Hyp_{k}} = \E_X\E_Y[ k(X,Y)]
$
with $X\sim G(\mu_0,\sigma^2 I)$ and $Y\sim G(\mu_1,\sigma^2 I)$. 
We will now derive the closed form for the inner product:
\begin{eqnarray}
\langle \theta_0, \theta_1\rangle_{\Hyp_{k}}
&{}={}&
\int_{\R^d}
\int_{\R^d}
\psi(x - y)
\frac{1}{(2\pi \sigma^2)^{d}}
e^{-\frac{1}{2\sigma^2} \|x - \mu_0\|_2^2-\frac{1}{2\sigma^2} \|y - \mu_1\|_2^2}
dx
dy\nonumber
\end{eqnarray}
\begin{eqnarray}
&{}={}&
\int_{\R^d}
\int_{\R^d}
\frac{1}{(2\pi)^{d/2}}
\int_{\R^d}
e^{-i\langle x - y, w\rangle}
d\Lambda_{\psi}(w)
\frac{1}{(2\pi \sigma^2)^{d}}
e^{-\frac{1}{2\sigma^2} \|x - \mu_0\|_2^2-\frac{1}{2\sigma^2} \|y - \mu_1\|_2^2}
dx\,
dy,\nonumber
\end{eqnarray}
where we used \eqref{Eq:Bochner}. The function appearing under the integral is absolutely integrable and so by Tonelli-Fubini theorem \cite[Theorem 4.4.5]{D02} we obtain
\begin{align*}
\langle \theta_0, \theta_1\rangle_{\Hyp_{k}}
&=
\int_{\R^d}
\int_{\R^d}
\frac{e^{i\langle y, w\rangle}e^{-\frac{1}{2\sigma^2} \|y - \mu_1\|_2^2}}{(2\pi \sigma^2)^{d/2}(2\pi)^{d/2}}
\left\{
\int_{\R^d}
\frac{e^{-i\langle x, w\rangle}}{(2\pi \sigma^2)^{d/2}}
e^{-\frac{1}{2\sigma^2} \|x - \mu_0\|_2^2}
dx\right\}
dy\,
d\Lambda_{\psi}(w)\\
&=
\int_{\R^d}
\int_{\R^d}
\frac{1}{(2\pi \sigma^2)^{d/2}}
\frac{1}{(2\pi)^{d/2}}
e^{i\langle y, w\rangle}e^{-\frac{1}{2\sigma^2} \|y - \mu_1\|_2^2}
e^{-i\langle\mu_0,w\rangle - \frac{\sigma^2\|w\|_2^2}{2}}
dy\,
d\Lambda_{\psi}(w)\\
&=
\frac{1}{(2\pi)^{d/2}}
\int_{\R^d}
\left\{
\int_{\R^d}
\frac{1}{(2\pi \sigma^2)^{d/2}}
e^{i\langle y, w\rangle}
e^{-\frac{1}{2\sigma^2} \|y - \mu_1\|_2^2}
dy\right\}
e^{-i\langle\mu_0,w\rangle - \frac{\sigma^2\|w\|_2^2}{2}}
d\Lambda_{\psi}(w)\\
&=
\frac{1}{(2\pi)^{d/2}}
\int_{\R^d}
e^{i\langle \mu_1, w\rangle - \frac{\sigma^2\|w\|_2^2}{2}}
e^{-i\langle\mu_0,w\rangle - \frac{\sigma^2\|w\|_2^2}{2}}
d\Lambda_{\psi}(w),
\end{align*}
where we used Lemma~\ref{lemma:Fourier-Gaussian} to compute the Fourier transform for a Gaussian density.
Using Euler's formula and the fact that $\Lambda_{\psi}$ is symmetric according to Lemma~\ref{lemma:psi-symmetric}, while $\sin(x)$ is an odd function, we get
\begin{align}
\langle \theta_0, \theta_1\rangle_{\Hyp_{k}}
&=
\frac{1}{(2\pi)^{d/2}}
\int_{\R^d}
\cos\left(
\langle \mu_0 - \mu_1, w\rangle
\right)
e^{ - \sigma^2\|w\|_2^2}\,
d\Lambda_{\psi}(w).\label{Eq:eee}
\end{align}
The result in (\ref{Eq:pro-bnd}) follows by using (\ref{Eq:eee}) in (\ref{eq:distance}).
\end{proof}
\noindent
\emph{Proof of Proposition~\ref{thm:rkhs-distance-lower-bound}:} 
Define $a := \mu_0 - \mu_1$ and $$G(a):=\frac{2}{(2\pi)^{d/2}}
\int_{\R^d}
e^{ - \sigma^2\|w\|_2^2}
\left(
1 - 
\cos
\langle a, w\rangle
\right)
d\Lambda_{\psi}(w).$$
Note that $G(0) = 0$. Next, since for any $i=1,\dots,d$
\[
\left|
\frac{\partial}{\partial a_i}
e^{ - \sigma^2\|w\|_2^2}
(
1 - 
\cos
\langle a, w\rangle
)\right|
=
\left|e^{ - \sigma^2\|w\|_2^2}
w_i\sin
\langle a, w\rangle
\right|
\leq
\left|e^{ - \sigma^2\|w\|_2^2}
w_i
\right|\in L_1(\Lambda_{\psi}),
\]
we can differentiate $G$ under the integral sign \cite[Theorem 2.27]{F99} and get $\nabla G(0) = 0$.

If a function $f\colon \R^d \to \R$ is \emph{strongly convex} with parameter $m>0$ on some set $A \subseteq \R^d$, then for all $x,y\in A$:
\[
f(x) \geq f(y) + \langle \nabla f(y), x - y \rangle + \frac{m}{2}\|x - y\|_2^2.
\]
If we can show that $G$ is strongly convex on $B_{\epsilon}:= \{b\in\R^d\colon \|b\|^2_2 \leq \epsilon\}$ for some $\epsilon > 0$, 
then we can apply previous inequality with $y = 0$ and $x=a$ to obtain 
\[
G(a) \geq \frac{m}{2}\|a\|_2^2,\,\,\forall a\in B_{\epsilon}.
\]
It is known that a twice continuously differentiable function $f$ is strongly convex on $A\subseteq \R^d$ with parameter $m>0$ if the matrix $\nabla^2 f(x) - m\cdot I$ is positive definite for all $x\in A$, where $\nabla^2 f$ is the Hessian and $I\in\R^{d\times d}$ is an identity matrix.
Next we compute the Hessian of $G$ by once again employing differentiation under the integral sign (justified in the similar way as above) to obtain
\[
\frac{\partial^2 G(a)}{\partial a_i \partial a_j}
=
\frac{2}{(2\pi)^{d/2}}
\int_{\R^d}
e^{ - \sigma^2\|w\|_2^2}
w_i w_j
\cos\left(
\langle a, w\rangle
\right)
d\Lambda_{\psi}(w),
\quad
0\leq i, j \leq d.
\]
Thus
\begin{align*}
\nabla^2 G(a)
&=
\frac{2}{(2\pi)^{d/2}}
\int_{\R^d}
e^{ - \sigma^2\|w\|_2^2}
w w^T
\cos\left(
\langle a, w\rangle
\right)
d\Lambda_{\psi}(w).	
\end{align*}

In order to prove that $G$ is strongly convex on $B_{\epsilon} \subseteq \R^d$ we need to show that $\nabla^2 G(a) - m\cdot I $ is positive definite for each $a\in B_{\epsilon}$ and some $m>0$.
In other words, we need to show that there is $m>0$ such that for each $z\in\R^d \setminus\{0\}$ and $a\in B_{\epsilon}$ the following holds:
\[
\langle z, \nabla^2 G(a) z\rangle \geq m\|z\|^2_2,
\]
or, equivalently,
\begin{equation}
\label{eq:Hessian-check}
\frac{2}{(2\pi)^{d/2}}
\int_{\R^d}
e^{ - \sigma^2\|w\|_2^2}
\langle e_z, w\rangle^2
e^{-i
\langle a, w\rangle
}
d\Lambda_{\psi}(w)
\geq
m,
\end{equation}
where $e_z:= z/\|z\|_2 \in \R^d$ is a vector of unit length pointed in the direction of $z$.
Note that l.h.s.~of \eqref{eq:Hessian-check} is the Fourier transform of a measure $\mathcal{T}_z$ on $\bb{R}^d$, which is absolutely continuous with respect to $\Lambda_{\psi}$ with Radon-Nikodym derivative $2e^{ - \sigma^2\|w\|_2^2}\langle e_z, w\rangle^2$.

Fix any $z\in\R^d$.
 We will first show that we can apply Bochner's Theorem (see \eqref{Eq:Bochner}) for the measure~$\mathcal{T}_z$.
For this we need to check that it is (a) non-negative and (b) finite.
Part (a) is apparent from the facts that $\Lambda_{\psi}$ is non-negative and $\mathcal{T}_z$ has a non-negative density with respect to $\Lambda_{\psi}$.
To check (b) we write
\[
\int_{\R^d} d \mathcal{T}_z(x)
=
\frac{2}{(2\pi)^{d/2}}
\int_{\R^d}
e^{ - \sigma^2\|w\|_2^2}
\langle e_z, w\rangle^2
d\Lambda_{\psi}( w ) < \infty,
\]
as $e^{ - \sigma^2\|w\|_2^2}\langle e_z, w\rangle^2$ is positive and bounded for any $z\in\R^d$, while $\Lambda_{\psi}$ is finite.
We conclude from Bochner's Theorem, that function $\tilde{\psi}_z(x)$ defined in the following way:
\[
\tilde{\psi}_z(x) =
\int_{\R^d}
e^{-i \langle x, w\rangle}
d\mathcal{T}_z(w),
\]
is positive-definite.
Moreover it is well known \citep[Theorem 9.4.4]{D02} that $\tilde{\psi}_z\in C_b(\bb{R}^d)$ as $\tilde{\psi}_z$ is the characteristic function of $\Cal{T}_z$.
Finally, it follows from the discussion in \cite[Section 3.3]{SFL11}, that if $\mathrm{supp}(\mathcal{T}_z) = \R^d$, then a bounded and continuous function $\tilde{\psi}_z(x)$ is \emph{strictly} positive definite.
To check the condition $\mathrm{supp}(\mathcal{T}_z) = \R^d$ we note that $\mathrm{supp}(\Lambda_{\psi}) = \R^d$ since $\psi$ is characteristic \cite[Theorem 9]{SGF+10}, and no open sets of $\R^d$ are contained in the region where $e^{ - \sigma^2\|w\|_2^2}\langle e_z, w\rangle^2=0$.

Summarizing, we have established that the l.h.s.~of \eqref{eq:Hessian-check} is equal to $\tilde{\psi}_z(a)$, where $\tilde{\psi}_z\colon \R^d \to \R$ is a bounded, continuous and strictly positive definite function for each $z\in\R^d\setminus \{0\}$.
In particular, we have $\tilde{\psi}_z(0) > 0$ for all $z\in\R^d$.
Note that $\tilde{\psi}_z(0)$ depends on $z$ only through its direction.
Next we want to show that
\begin{equation}
\label{eq:infimum-not-zero}
\inf_{z\in S_d} \tilde{\psi}_z(0) > 0,
\end{equation}
where the infimum is over the unit sphere $S_d:=\{b\in \R^d\colon \|b\|^2_2 = 1\}$.
Note that the function $F\colon z\to \tilde{\psi}_z(0)$ defined on $S_d$ is continuous.
Since $S_d$ is closed and bounded, we know that $F$ attains its minimum on it.
In other words, there is $z^*\in S_d$, such that
\[
\inf_{z\in S_d} \tilde{\psi}_z(0) = \tilde{\psi}_{z^*}(0).
\]
Thus, if $\inf_{z\in S_d} \tilde{\psi}_z(0) = 0$, we will also get $\tilde{\psi}_{z^*}(0) = 0$, which will contradict the fact that $\tilde{\psi}_z(0) > 0$ for each $z\in\R^d\backslash\{0\}$.
This proves \eqref{eq:infimum-not-zero}.
Using Lemma \ref{lemma:continuous-compact-inf} 
we also conclude that $\inf_{z\in S_d} \tilde{\psi}_z \colon \R^d \to \R$ is a continuous function.
Now we may finally conclude that there are constants $c_{\psi,\sigma^2}, \epsilon_{\psi,\sigma^2} > 0$ such that
\[
\inf_{z\in S_d} \tilde{\psi}_z(a) \geq c_{\psi,\sigma^2}
\]
for all $a\in B_{\epsilon_{\psi,\sigma^2}}$. Finally, we take $m = c_{\psi,\sigma^2}$ and this concludes the proof.

\subsection{Proof of Theorem \ref{thm:radial-lower-bound}}
\label{proof:radial-lower-bound}
The proof is based on application of Theorem \ref{thm:Tsybakov} where we choose $\theta_0,\dots,\theta_M$ to be KMEs of $d$-dimensional Gaussian measures with variances decaying to zero as $d\to\infty$.

Let $G(\mu_0, \sigma^2 I)$ and $G(\mu_1, \sigma^2 I)$ be two $d$-dimensional Gaussian distributions with mean vectors $\mu_0,\mu_1\in\R^d$ and variance $\sigma^2 > 0$. Define $\theta_0$ and $\theta_1$ to be the embeddings of $G(\mu_0, \sigma^2 I)$ and $G(\mu_1, \sigma^2 I)$ respectively.

\bigskip
\noindent
\emph{(A) Deriving a closed form expression for $\| \theta_0 - \theta_1\|_{\Hyp_k}^2$.}

\bigskip

Using Lemma \ref{lemma:rkhs-distance}, presented in Section \ref{proof:rkhs-distance-lower-bound}, we have
\begin{equation}
\label{eq:proof-continuous-rkhs-distance}
\| \theta_0 - \theta_1\|^2_{\Hyp_{k}}=
\frac{2}{(2\pi)^{d/2}}
\int_{\R^d}
e^{ - \sigma^2\|w\|_2^2}
\left(
1 - 
\cos\left(
\langle \mu_0 - \mu_1, w\rangle
\right)
\right)
d\Lambda_{\psi}(w),
\end{equation}
where $\Lambda_{\psi}$ is a finite non-negative Borel measure from the Bochner's Theorem corresponding to the kernel $k$. 
We now show that $\Lambda_{\psi}$ is absolutely continuous with respect to the Lebesgue measure on $\R^d$ and has the following density: 
\[
\lambda_{\psi}(w)
=
\int_0^{\infty}
\frac{1}{(2 t)^{d/2}}
e^{-\frac{\|w\|^2_2}{4t}}
d\nu(t),
\quad w\in\R^d.
\]
{Indeed, by noticing that
\[
\int_0^{\infty}
\left(\int_{\R^d}
\left|
e^{-i\langle w, x\rangle}
\frac{1}{(2 t)^{d/2}}
e^{-\frac{\|w\|^2_2}{4t}}
\right|
dw\right)
d\nu(t) < \infty
\]
we may apply Tonelli-Fubini theorem \cite[Theorem 4.4.5]{D02} to interchange the order of integration and get
\begin{align*}
\int_{\R^d}
\frac{e^{-i\langle w, x\rangle}}{(2\pi)^{d/2}}
\left(\int_0^{\infty}
\frac{1}{(2 t)^{d/2}}
e^{-\frac{\|w\|^2_2}{4t}}
d\nu(t)\right)
dw
&=
\int_0^{\infty}
\left(\int_{\R^d}
\frac{e^{-i\langle w, x\rangle}}{(2\pi)^{d/2}}
\frac{1}{(2 t)^{d/2}}
e^{-\frac{\|w\|^2_2}{4t}}
dw\right)
d\nu(t)\\
&=\int_0^{\infty}
e^{-t \|x\|^2_2}
d\nu(t).
\end{align*}
Substituting the form of $\lambda_{\psi}$ into \eqref{eq:proof-continuous-rkhs-distance} we can write
\[
\| \theta_0 - \theta_1\|^2_{\Hyp_{k}}=
\frac{2}{(2\pi)^{d/2}}
\int_{\R^d}
\int_0^{\infty}
e^{ - \sigma^2\|w\|_2^2}
\left(
1 - 
\cos\left(
\langle \mu_0 - \mu_1, w\rangle
\right)
\right)
\frac{1}{(2 t)^{d/2}}
e^{-\frac{\|w\|^2_2}{4t}}\,
d\nu(t)\,
dw.
\]
Applying Tonelli-Fubini theorem once again and using Lemma \ref{lemma:Fourier-Gaussian} together with Euler's formula we obtain
\begin{align}
\notag
\| \theta_0 - \theta_1\|^2_{\Hyp_{k}}
&=
\frac{2}{(2\pi)^{d/2}}
\int_0^{\infty}
\frac{1}{(2 t)^{d/2}}
\int_{\R^d}
e^{ - \frac{\|w\|_2^2}{2}(2\sigma^2 + \frac{1}{2t})}
\left(
1 - 
\cos\left(
\langle \mu_0 - \mu_1, w\rangle
\right)
\right)\,
d w\,
d\nu(t)\\
\notag
&=
\int_{0}^{\infty} \frac{2}{(4\sigma^2 t + 1)^{d/2}}d\nu(t)
-
\int_{0}^{\infty} \frac{2}{(4\sigma^2 t + 1)^{d/2}}
e^{-\frac{t \|\mu_0 - \mu_1\|^2_2}{4\sigma^2 t + 1}}d\nu(t)\\
&=
\label{eq:rkhs-distance-radial-closed}
\int_0^\infty
2\left(\frac{1 }{1 + 4t\sigma^2}\right)^{d/2}
\left(
1
-
\exp\left(
-\frac{t\|\mu_0-\mu_1\|^2_2}{1 + 4t\sigma^2}
\right)
\right)
d\nu(t).
\end{align}}

\bigskip
\noindent
\emph{(B) Lower bounding $\| \theta_0 - \theta_1\|_{\Hyp_k}^2$ in terms of $\|\mu_0 - \mu_1\|^2_2$.}

\bigskip
It follows from (\ref{eq:rkhs-distance-radial-closed}) that
$$\| \theta_0 - \theta_1\|^2_{\Hyp_{k}}\geq
\int_{t_0}^{t_1}
2\left(\frac{1 }{1 + 4t\sigma^2}\right)^{d/2}
\left(
1
-
\exp\left(
-\frac{t\|\mu_0-\mu_1\|^2_2}{1 + 4t\sigma^2}
\right)
\right)
d\nu(t),$$
where $0 < t_0 \leq t_1 < \infty$. Note that $1 - e^{-x} \geq \frac{x}{2}$ for $0\leq x \leq 1$. Using this we get
\[
1
-
\exp\left(
-\frac{t\|\mu_0-\mu_1\|_2^2}{1 + 4t\sigma^2}
\right)
\geq
\frac{t\|\mu_0-\mu_1\|^2_2}{2(1 + 4t\sigma^2)},\quad \forall t\in[t_0, t_1]
\]
as long as
\begin{equation}
\label{eq:proof-assumption-2}
{t_1\|\mu_0-\mu_1\|^2_2}\leq{1 + 4t_1\sigma^2}.
\end{equation}
Thus, as long as \eqref{eq:proof-assumption-2} holds, we can lower bound the RKHS distance as:
\begin{align}
\notag
\left\|
\theta_0
-
\theta_1
\right\|_{\Hyp_{k}}^2
\notag
&\geq
\int_{t_0}^{t_1}
\left(\frac{1 }{1 + 4t\sigma^2}\right)^{d/2}
\frac{t\|\mu_0-\mu_1\|^2_2}{1 + 4t\sigma^2}\,
d\nu(t)\\
\label{eq:rkhs-distance-lower}
&=
\|\mu_0-\mu_1\|^2_2
\int_{t_0}^{t_1}
\frac{t}{(1 + 4t\sigma^2)^{(d+2)/2}}\,
d\nu(t).
\end{align}
Note that the function $
t\mapsto\frac{t}{(1 + 4t\sigma^2)^{(d+2)/2}}$ monotonically increases on $[0, \frac{1}{2d\sigma^2}]$, reaches its global maximum at $t = \frac{1}{2d\sigma^2}$ and then decreases on $[\frac{1}{2d\sigma^2}, \infty)$.
Thus we have
\begin{align*}
\int_{t_0}^{t_1}
\frac{t}{(1 + 4t\sigma^2)^{(d+2)/2}}\,
d\nu(t)
&\geq
\beta
\min\left\{
\frac{t_0}{(1 + 4t_0\sigma^2)^{(d+2)/2}},
\frac{t_1}{(1 + 4t_1\sigma^2)^{(d+2)/2}}
\right\}.
\end{align*}
Setting \begin{equation}\sigma^2 = \frac{1}{2t_1d}\label{Eq:signa}\end{equation} yields that $t = t_1$ is the global maximum of the function 
$t\mapsto \frac{t}{\left(1 + \frac{2t}{t_1 d}\right)^{(d+2)/2}}$, in which case
\[
\frac{t_0}{(1 + 4t_0\sigma^2)^{(d+2)/2}}
\leq
\frac{t_1}{(1 + 4t_1\sigma^2)^{(d+2)/2}}.
\]
With this choice of $\sigma^2$ we get
\begin{eqnarray}
\lefteqn{
\int_{t_0}^{t_1}
\frac{t}{(1 + 4t\sigma^2)^{(d+2)/2}}\,
d\nu(t)
\geq
\frac{\beta t_0}{\left(1 + \frac{2t_0}{t_1d}\right)^{(d+2)/2}}
\geq
\frac{\beta t_0}{\left(1 + \frac{2}{d}\right)^{(d+2)/2}}}\nonumber\\
&&\qquad\qquad\qquad\qquad=
\beta t_0
\left(
1 - \frac{2}{2 + d}
\right)^{(d+2)/2}
\geq
\frac{\beta t_0}{e}
\left(
1 - \frac{2}{2 + d}
\right),
\label{eq:integral-d-lower-bound-final}
\end{eqnarray}
where we used the fact that $(1 - \frac{1}{x})^{x-1}$ monotonically decreases to $\frac{1}{e}$.
Using (\ref{eq:integral-d-lower-bound-final}) in (\ref{eq:rkhs-distance-lower}), we obtain
\begin{equation}
\label{eq:proof-rkhs-radial-step-B}
\left\|
\theta_0
-
\theta_1
\right\|_{\Hyp_{k}}^2
\geq
\frac{\beta t_0}{e}
\left(
1 - \frac{2}{2 + d}
\right)\|\mu_0-\mu_1\|^2_2.
\end{equation}

\bigskip
\noindent\emph{(C.1) Application of Theorem \ref{thm:Tsybakov}: Choosing $\theta_0,\dots,\theta_M$.}

\bigskip

Now we are going to apply Theorem \ref{thm:Tsybakov}. First of all, we need to choose $M+1$ embeddings.
Recall that Theorem \ref{thm:Tsybakov} requires these embeddings to be sufficiently distant from each other, while the corresponding distributions should be close.
We will choose the embeddings $\{\theta_0,\dots, \theta_M\}$ to be KMEs of Gaussian distributions $G(\mu_i, \sigma^2 I)$ for specific choice of $\sigma^2>0$ and ${\mu_i\in\R^d}$, $i=0,\dots,M$.
Mean vectors $\{\mu_i\}_{i=0}^M$ will be constrained to live in the ball $B(c_{\nu},n) := \left\{ x\in\R^d\colon \|x\|_2^2 \leq c_{\nu}/n                                                                                                                                                                                                                                                                                                                                                                                                                                                                                                                                                                                                                                                                                                                                                                                                                                                                                                                                                                                                                                                                                    \right\}$, where $c_{\nu}$ is a positive constant to be specified later.
This guarantees that KL-divergences between the Gaussian distributions will remain small.
At the same time, it was shown in \eqref{eq:proof-rkhs-radial-step-B} that the RKHS distance between embeddings $\theta_i$ and $\theta_j$ is lower bounded by the Euclidean distance between $\mu_i$ and $\mu_j$.
In other words, in order for the embeddings $\theta_0,\dots, \theta_M$ to be sufficiently separated we need to make sure that the mean vectors $\mu_0,\dots,\mu_M$ are not too close to each other.
Summarizing, we face the problem of choosing a finite collection of pairwise distant points in the Euclidean ball.
This question is closely related to the concepts of \emph{packing} and \emph{covering numbers}.

For any set $A\in \R^d$ its $\epsilon$-\emph{packing number} $M(A, \epsilon)$ is the largest number of points in $A$ separated from each other by at least a distance of $\epsilon$.
An $\epsilon$-\emph{covering number} $N(A, \epsilon)$ of $A$ is the minimal number of balls of radius $\epsilon$ needed to cover $A$.
Packing numbers are lower bounded by the covering numbers \cite[Theorem 1.2.1]{D99}:
\[
N(A, \epsilon) \leq M(A, \epsilon).
\]
Also, it is well known that the $\epsilon$-covering number of a unit $d$-dimensional Euclidean ball is lower bounded by $\lfloor\epsilon^{-d}\rfloor$.
Together, these facts state that we can find at least $\lfloor\epsilon^{-d}\rfloor$ points in the $d$-dimensional unit ball, which are at least $\epsilon$ away from each other. Similarly (just by a simple scaling) we can find at least $\lfloor\epsilon^{-d}\rfloor$ points in the $d$-dimensional ball of radius $R>0$, which are at least $R\cdot\epsilon$ away from each other.
Applying this fact to $B(c_{\nu},n) $ we can finally argue that there are at least $N^d$ points in $B(c_{\nu},n) $ which are at least $N^{-1}\sqrt{c_{\nu}/n}$ away from each other, where $N\geq 3$ is an integer to be specified later.
Now, take $M = N^d  -1 \geq 2$ (which explains the lower bound $N\geq 3$) and fix $\mu_0,\dots,\mu_M$ to be these $M+1$ points.

\bigskip
\noindent\emph{(C.2) Application of Theorem \ref{thm:Tsybakov}: Lower bounding $\|\theta_i - \theta_j\|_{\Hyp_k}$.}

\bigskip

With this choice of parameters $\mu_0,\dots,\mu_M$, for any $0\leq i < j \leq M$,  we have
\begin{align*}
&\left\|
\theta_i
-
\theta_j
\right\|_{\Hyp_{k}}^2
\geq
\frac{c_{\nu}}{N^2n}
\frac{\beta t_0}{e}
\left(
1 - \frac{2}{2 + d}
\right),
\end{align*}
where we used \eqref{eq:proof-rkhs-radial-step-B} and the lower bound on $\|\mu_i - \mu_j\|_2$.
Setting \begin{equation}c_{\nu} = \frac{C}{\beta t_0}\label{Eq:cnu}\end{equation} for some $C>0$ we obtain
\[
\left\|
\theta_i
-
\theta_j
\right\|_{\Hyp_{k}}^2\\
\geq
\frac{C}{N^2en}
\left(
1 - \frac{2}{2 + d}
\right).
\]
This satisfies the first assumption of Theorem \ref{thm:Tsybakov} with 
$
s:= \frac{1}{2N}\sqrt{\frac{C}{en}
\left(
1 - \frac{2}{2 + d}
\right)}.$

\bigskip
\noindent\emph{(C.3) Application of Theorem \ref{thm:Tsybakov}: Upper bounding $\mathrm{KL}(P_{\theta_i}\| P_{\theta_j})$.}

\bigskip

Note that for any $0\leq i < j \leq M$ we have
\begin{align*}
\mathrm{KL}\bigl(G^n(\mu_i,\sigma^2 I) \| G^n(\mu_j, \sigma^2 I )\bigr)
=
n\cdot\frac{\|\mu_i - \mu_j\|_2^2}{2\sigma^2}
\leq\frac{2c_\nu}{\sigma^2}=
4 C \frac{t_1 d}{\beta t_0},
\end{align*}
where the inequality holds since $\mu_i \in B(c_{\nu},n)$ and the equality follows from (\ref{Eq:signa}) and (\ref{Eq:cnu}).
Here we used the fact that for any points $x$ and $y$ contained in a ball of radius $R$ we obviously have $\|x - y\|\leq 2R$.
Also note that we chose $M = N^d  - 1 \geq 2$ and thus 
\[
\log(M) 
=
d \log(N) + \log(1 - N^{-d})
\geq
d\log(N) + 1 - \frac{N^d}{N^d -1} 
\geq
d \log(N) - \frac{1}{N -1} 
\]
for $d\geq 1$, where we used the inequality $\log(x) \geq 1 - \frac{1}{x}$ which holds for $x\geq 0$.
Taking 
\begin{equation}
\label{Eq:Cc}
C = \frac{\beta t_0}{32t_1} \left(\log N - \frac{1}{N-1}\right)
\end{equation} we get
\[
4 C \frac{t_1 d}{\beta t_0} = \frac{1}{8}\left(d\log N - \frac{d}{N-1}\right) \leq \frac{1}{8}\left(d\log N - \frac{1}{N-1}\right) \leq \frac{1}{8} \log(M)
\]
for any $d\geq 1$.
Concluding, we get
\begin{align*}
\mathrm{KL}\bigl(G^n(\mu_i,\sigma^2 I) \| G^n(\mu_j\| \sigma^2 I )\bigr)
\leq
\frac{1}{8}\log(M)
\end{align*}
and thus the second assumption of Theorem \ref{thm:Tsybakov} is satisfied with $\alpha = \frac{1}{8}$.
Finally, it is easy to check that if we take $N = 5$ then condition \eqref{eq:proof-assumption-2} will be satisfied.
Indeed,
\begin{align*}
t_1\|\mu_0-\mu_1\|_2^2 \leq \frac{4t_1c_{\nu}}{n} = \frac{4t_1}{\beta t_0 n}\frac{\beta t_0}{32t_1} \left(\log N - \frac{1}{N-1}\right)
=
\frac{1}{ 8 n}\left(\log N - \frac{1}{N-1}\right),
\end{align*}
while
\[
1 + 4t_1\sigma^2 \geq 1.
\]
Thus \eqref{eq:proof-assumption-2} holds whenever
\[
\log N - \frac{1}{N-1} \leq 8n,
\]
which obviously holds for $N=5$ and any $n \geq 1$.

To conclude the proof we insert \eqref{Eq:Cc} into $s:= \frac{1}{2N}\sqrt{\frac{C}{en}
\left(
1 - \frac{2}{2 + d}
\right)}$, lower bound this value using $\frac{1}{8}\left(\log N - \frac{1}{N-1}\right) \geq \frac{4}{25}$,
and notice that
$$
\frac{\sqrt{M}}{1 + \sqrt{M}}
\left(
1 - 2\alpha - \sqrt{\frac{2\alpha}{\log(M)}}
\right) \geq \frac{1}{5}.
$$

\subsection{Proof of Theorem \ref{thm:shift-invariant-lower-bound-discrete-l2}}
\label{proof:shift-invariant-lower-bound-discrete-l2}
The proof is based on the following result, which gives a closed form expression for the $L^2(\R^d)$ distance between embeddings of two discrete distributions supported on the same pair of points in $\R^d$.
\begin{lemma}
\label{lemma:l2-discrete-closed}
Suppose $k(x,y)= \psi(x - y)$, where $\psi \in L^2(\R^d)\cap C_b(\R^d)$ is positive definite and $k$ is characteristic.
Define
$P_0=p_0 \delta_x + (1-p_0) \delta_v$ and $P_1=p_1 \delta_x + (1-p_1) \delta_v$, where $0 < p_0 < 1$, $0 < p_1 < 1$,  $x,v\in\R^d$, and $x\neq v$.
Then
\[
\left\|
\mu_k(P_0)
-
\mu_k(P_1)
\right\|_{L^2(\bb{R}^d)}^2
=
C^{\psi}_{x,v} 
(p_0 - p_1)^2,
\]
where
\[
C^{\psi}_{x,v} := 
2\left(
\|\psi\|_{L^2(\bb{R}^d)}^2
-
\int_{\R^d}
\psi(y)\psi(y +x - v)
dy\right) > 0.
\]
\end{lemma}
\begin{proof}
We have
\begin{align*}
\left\|
\mu_k(P_0)
-
\mu_k(P_1)
\right\|_{L^2(\bb{R}^d)}^2
&=
(p_0 - p_1)^2
\left\|
k(x,\cdot) - k(v,\cdot)
\right\|_{L^2(\bb{R}^d)}^2\\
&=
(p_0 - p_1)^2\int_{\R^d}
\bigl(
\psi(x - y) - \psi(v - y)
\bigr)^2
dy\\
&\stackrel{(\star)}{=}
(p_0 - p_1)^2\int_{\R^d}
\bigl(
\psi(y) - \psi(y + x - v)
\bigr)^2
dy\\
&=
2(p_0 - p_1)^2
\int_{\R^d}
\psi(y)\bigl(
\psi(y) - \psi(y + x - v)
\bigr)
dy\\
&=
2(p_0 - p_1)^2
\left(
\|\psi\|_{L^2(\bb{R}^d)}^2
-
\int_{\R^d}
\psi(y)\psi(y + x - v)
dy\right),
\end{align*}
where we used the symmetry of $\psi$ in $(\star)$.
Cauchy-Schwartz inequality states that
\[
\int_{\R^d}
\psi(y)\psi(y + x - v)
dy
\mathop{\leq}^{(\star)}
\sqrt{\int_{\R^d}
\psi^2(y)
dy}\,
\sqrt{\int_{\R^d}
\psi^2(y + x - v)
dy}
=
\|\psi\|_{L^2(\bb{R}^d)}^2,
\]
where the equality in $(\star)$ holds if and only if $\psi(y) = \lambda\, \psi(y + x - v)$ for some constant $\lambda$ and all $y\in\R^d$.
If we take $y=v - x$ the above condition implies 
$
\psi(v - x) = \lambda\,  \psi(0)
$
and with $y = 0$ we get
$
\psi(0) = \lambda \, \psi(x - v).
$
Together these identities show that $\psi(0) = \lambda^2 \, \psi(0)$.
Since the kernel is characteristic and translation invariant, $\psi$ is strictly positive definite, which means $\psi(0) > 0$.
We conclude that $\lambda = \pm 1$.
Assume that $\lambda = -1$.
In this case $\psi(x - v) = - \psi(0) < 0$.
Repeating the argument we can show that $\psi\bigl(2(x - v)\bigr) = \psi(0)$ and generally $\psi\bigl(m(x - v)\bigr) = (-1)^m\psi(0)$ for all $m\in \mathbb{N}$.
Since $\psi \in L^2(\R^d)$ we need $\psi^2$ to be integrable on $\R^d$.
Summarizing, we showed that a non-negative, integrable, and continuous function takes the same strictly positive value $\psi^2(0)>0$ infinitely many times, leading to a contradiction. Arguing similarly for $\lambda=1$
will result in a contradiction. This means the equality in $(\star)$ is never attained which concludes the proof.
\end{proof}
The proof of Theorem \ref{thm:shift-invariant-lower-bound-discrete-l2} is carried out by simply repeating the proof of Theorem \ref{thm:translation invariant-lower-bound-discrete} but replacing \eqref{eq:proof-worst-rkhs} with the result in Lemma \ref{lemma:l2-discrete-closed} and using $x-v := z$.

\subsection{Proof of Corollary~\ref{thm:radial-l2-lower-bound-discrete}}
\label{proof:radial-l2-lower-bound-discrete}
The proof will be based on Theorem \ref{thm:shift-invariant-lower-bound-discrete-l2}.
The moment condition \eqref{equation:worst-case-condition} on $\nu$ is sufficient for $\psi_{\nu} \in L^2(\R^d)$ to hold (see Remark \ref{rem:l2-condition}).
Thus we only need to compute the expression $
\|\psi_\nu\|_{L^2(\bb{R}^d)}^2
-
\int_{\R^d}
\psi_\nu(y)\psi_\nu(y + z)
dy$ appearing in \eqref{eq:l2-shiftinvariant-expression-lower}.
Note that 
\begin{align}
\label{eq:radial-l2-worst-expression}
\int_{\R^d}
\psi_\nu(y)\psi_\nu(y + z)
dy
&=
\int_{\R^d}
\int_0^\infty
\int_0^\infty
e^{-t_1 \|y\|^2_2 - t_2 \|y+z\|^2_2}
d\nu(t_1) 
d\nu(t_2) 
dy.
\end{align}
Since
\begin{align*}
\int_0^\infty\!\!\!
\int_0^\infty\!\!\!
\int_{\R^d}
e^{-t_1 \|y\|^2_2 - t_2 \|y+z\|^2_2}
dy\,
d\nu(t_1) 
d\nu(t_2) 
&=
\int_0^\infty\!\!\!
\int_0^\infty
\left(\frac{\pi}{t_1 + t_2}\right)^{d/2}
e^{
-\frac{t_1t_2\|z\|^2_2}{t_1 + t_2}
}
d\nu(t_1) 
d\nu(t_2)\\
&\leq
\nu([0,\infty))\int_0^\infty
\left(\frac{\pi}{t_1}\right)^{d/2}
d\nu(t_1) < \infty,
\end{align*}
we may apply Tonelli-Fubini theorem to switch the order of integration in \eqref{eq:radial-l2-worst-expression} and get
\begin{align*}
\int_{\R^d}
\psi_\nu(y)\psi_\nu(y + z)
dy
=
\int_0^\infty\!\!\!
\int_0^\infty
\left(\frac{\pi}{t_1 + t_2}\right)^{d/2}
e^{
-\frac{t_1t_2\|z\|^2_2}{t_1 + t_2}
}
d\nu(t_1) 
d\nu(t_2).
\end{align*}
Using this we get
\begin{align*}
\|\psi_\nu\|_{L^2(\bb{R}^d)}^2
-
\int_{\R^d}
\psi_\nu(y)\psi_\nu(y + z)
dy
&=
\int_0^\infty\!\!\!
\int_0^\infty
\left(\frac{\pi}{t_1 + t_2}\right)^{d/2}
\left(1 - e^{
-\frac{t_1t_2\|z\|^2_2}{t_1 + t_2}
}\right)
d\nu(t_1) 
d\nu(t_2)\\
&\geq
\left(\frac{\pi}{2\delta_1}\right)^{d/2}
\int_{\delta_0}^{\delta_1}\!\!\!
\int_{\delta_0}^{\delta_1}
\left(1 - e^{
-\frac{t_1t_2\|z\|^2_2}{t_1 + t_2}
}\right)
d\nu(t_1) 
d\nu(t_2).
\end{align*}
Since $t_1$ and $t_2$ are bounded below by $\delta_0 > 0$ we may take $\|z\|_2$ large enough so that the following will hold:
$$
\|\psi_\nu\|_{L^2(\bb{R}^d)}^2
-
\int_{\R^d}
\psi_\nu(y)\psi_\nu(y + z)
dy
\geq
\frac{\beta^2}{2}
\left(\frac{\pi}{2\delta_1}\right)^{d/2}.
$$

\subsection{Proof of Proposition~\ref{thm:l2-distance-lower-bound}}
\label{proof:l2-distance-lower-bound}
The proof is based on the following result, which provides a closed form expression for the $L^2(\bb{R}^d)$ distance between the embeddings of  Gaussian measures.
\begin{lemma}
\label{lemma:l2-distance-closed-form}
Let $\theta_0$ and $\theta_1$ be KME of Gaussian measures $G(\mu_0,\sigma^2 I)$ and $G(\mu_1,\sigma^2 I)$ for $\mu_0,\mu_1\in\R^d$ and $\sigma^2 > 0$.
Suppose $k(x,y) = \psi(x - y)$, where $\psi\in L^1(\R^d)\cap C_b(\mathbb{R}^d)$ is positive definite and 
and $k$ is characteristic. 
Then
\begin{align}
\label{eq:l2-distance-closed-form}
\| \theta_0 - \theta_1\|_{L^2(\bb{R}^d)}^2
=
2\int_{\R^d}
\left(1 - e^{-i\langle w, \mu_0 - \mu_1\rangle}\right)
\bigl(\psi^\wedge(w)\bigr)^2
e^{-\sigma^2 \|w\|^2}
dw.
\end{align}
\end{lemma}
\begin{proof}
First of all, note that $\psi\in L^2(\bb{R}^d)$ since $\psi \in L^1(\R^d)$ and $\psi$ is bounded. 
This shows that $\theta_0, \theta_1 \in L^2(\R^d)$.
We will use $P_0$ and $P_1$ to denote the corresponding Gaussian distributions $G(\mu_0,\sigma^2 I)$ and $G(\mu_1,\sigma^2 I)$.
By definition we have
\begin{align*}
\langle \theta_0 , \theta_1\rangle_{L^2(\bb{R}^d)}
&=
\int_{\R^d}
\left(
\int_{\R^d}
k(x,y)\, dP_0(x)
\right)
\left(
\int_{\R^d}
k(z,y)\, dP_1(z)
\right)
dy\\
&=
\int_{\R^d}
\left(
\int_{\R^d}
\int_{\R^d}
k(x,y) 
k(z,y) \,
dP_0(x)
\,dP_1(z)
\right)
dy.
\end{align*}
Using the fact that $\psi$ is bounded \cite[Theorem 6.2]{W05} we get
\begin{align*}
\int_{\R^d}
\int_{\R^d}
\int_{\R^d}
\bigl|
k(x,y) 
k(z,y)\bigr|\,
dy\, dP_0(x)\, dP_1(z) 
&\leq
\psi(0)\int\!\!
\int\!\!
\int_{\R^d}
\bigl|
k(x,y) \bigr|\,
dy\, dP_0(x)\, dP_1(z) \\
&=
\psi(0)\int\!\!
\int\!\!
\int_{\R^d}
\bigl|
\psi(x-y) \bigr|
dy\, dP_0(x) dP_1(z)\\
&=
\psi(0)\|\psi\|_{L^1(\bb{R}^d)}
\int\!\!
\int_{\R^d} dP_0(x) dP_1(z) < \infty.
\end{align*}
This allows us to use Tonelli-Fubini theorem \cite[Theorem 4.4.5]{D02} and get
\begin{align*}
\langle \theta_0, \theta_1\rangle_{L^2(\bb{R}^d)}&=
\int_{\R^d}
\int_{\R^d}
\left(\int_{\R^d}
k(x,y) 
k(z,y) 
dy\right)
dP_0(x)
dP_1(z)
\\
&=
\int_{\R^d}
\int_{\R^d}
\left(\int_{\R^d}
\psi(y)
\psi(y + z - x)
dy\right)
dP_0(x)
dP_1(z)\\
&\stackrel{(\star)}{=}
\frac{1}{(2\pi)^{d/2}}
\int_{\R^d}
\int_{\R^d}
\left(\int_{\R^d}
\psi(y)
\int_{\R^d} e^{-i\langle y + z - x, w\rangle} d\Lambda_{\psi}(w)
dy\right)
dP_0(x)
dP_1(z)\\
&=
\frac{1}{(2\pi)^{d/2}}
\int_{\R^d}
\int_{\R^d}
\left(\int_{\R^d}
\int_{\R^d} 
\psi(y)
e^{-i\langle y + z - x, w\rangle} d\Lambda_{\psi}(w)
dy\right)
dP_0(x)
dP_1(z),
\end{align*}
where we used \eqref{Eq:Bochner} in $(\star)$.
Since $\psi \in L^1(\R^d)$ we have
\[
\int_{\R^d}
\int_{\R^d} 
\left|
e^{-i\langle y, w\rangle}\psi(y)
e^{-i\langle z - x, w\rangle} \right|
dy\,
d\Lambda_{\psi}(w)
=
\int_{\R^d}
\int_{\R^d} 
\left|
\psi(y)
 \right|
 dy\,
d\Lambda_{\psi}(w)
< \infty
\]
and thus we can use Tonelli-Fubini theorem to switch the order of integration:
\begin{align}
\notag
\langle \theta_0 , \theta_1\rangle_{L^2(\bb{R}^d)}&=
\frac{1}{(2\pi)^{d/2}}
\int_{\R^d}
\int_{\R^d}
\left(\int_{\R^d}
\int_{\R^d} 
\psi(y)
e^{-i\langle y + z - x, w\rangle} 
dy
\,d\Lambda_{\psi}(w)
\right)
dP_0(x)
dP_1(z)\\
\label{eq:proof-l2-continuous-inner-product}
&=
\int_{\R^d}
\int_{\R^d}
\left(\int_{\R^d}
\psi^\wedge(w)
e^{-i\langle  z - x, w\rangle} 
\,d\Lambda_{\psi}(w)
\right)
dP_0(x)
dP_1(z).
\end{align}
Next we are going to argue that if both $\psi$ and $\psi^\wedge$ belong to $L^1(\R^d)$ (the latter is true as it follows from 
\citealp[Corollary 6.12]{W05}) then $\psi^\wedge$ is the Radon-Nikodym derivative of $\Lambda_{\psi}$ with respect to the Lebesgue measure.
To this end, since $\psi^\wedge\in L^1(\bb{R}^d)$, Fourier inversion theorem \cite[Corollary 5.24]{W05} yields that for all $x\in \R^d$, the following holds:
\[
\psi(x) = \frac{1}{(2\pi)^{d/2}} \int_{\R^d} e^{i\langle w, x \rangle} \psi^\wedge(w) dw.
\]
On the other hand, using \eqref{Eq:Bochner} and Lemma \ref{lemma:psi-symmetric}, we also have
\[
\psi(x) = \frac{1}{(2\pi)^{d/2}} \int_{\R^d} e^{i\langle w, x \rangle} d\Lambda_{\psi}(w).
\]
These two identities show that for all $x\in\R^d$
\begin{equation}
\label{eq:chf-are-the-same}
\int_{\R^d} e^{i\langle w, x \rangle} \psi^\wedge(w) dw
=
\int_{\R^d} e^{i\langle w, x \rangle} d\Lambda_{\psi}(w).
\end{equation}
Note that since $k$ is translation invariant and characteristic, $\psi$ is a strictly positive definite function \cite[Section 3.4]{SGF+10} and 
therefore it follows from \cite[Theorem 6.11]{W05} that $\psi^\wedge$ is non-negative (and nonvanishing). 
Since $\psi^\wedge \in L^1(\R^d)$ we conclude that $\psi^\wedge$ is the Radon-Nikodym derivative of a finite non-negative measure $\mathcal{T}_{\psi}$ on $\R^d$, which is absolutely continuous with respect to the Lebesgue measure.
\eqref{eq:chf-are-the-same} (after proper normalization) shows that the characteristic functions of measures $\Lambda_{\psi}$ and $\mathcal{T}_{\psi}$ coincide.
We finally conclude from \cite[Theorem~9.5.1]{D02} that $\Lambda_{\psi} = \mathcal{T}_{\psi}$, which means that $\Lambda_{\psi}$ is absolutely continuous with respect to the Lebesgue measure and has a density $\psi^\wedge$.

Returning to \eqref{eq:proof-l2-continuous-inner-product} we can write it as
\begin{align*}
\langle \theta_0 , \theta_1\rangle_{L^2(\bb{R}^d)}&=
\int_{\R^d}
\int_{\R^d}
\left(\int_{\R^d}
\bigl(\psi^\wedge(w)\bigr)^2
e^{-i\langle  z - x, w\rangle} 
\,dw
\right)
dP_0(x)
dP_1(z).
\end{align*}
We already showed that $\psi \in L^2(\R^d)$. From Plancherel's theorem \cite[Corollary 5.25]{W05}, we have $\| \psi\|_{L^2(\bb{R}^d)} = \| \psi^\wedge\|_{L^2(\bb{R}^d)}$ and thus $\psi^\wedge \in L^2(\R^d)$.
Another application of Tonelli-Fubini theorem yields
\begin{align*}
\langle \theta_0 , \theta_1\rangle_{L^2(\R^d)}&=
\int_{\R^d}
\bigl(\psi^\wedge(w)\bigr)^2
\left(
\int_{\R^d}
\int_{\R^d}
e^{-i\langle  z - x, w\rangle} 
dP_0(x)
dP_1(z)
\right)dw\\
&=
\int_{\R^d}
\bigl(\psi^\wedge(w)\bigr)^2
e^{i\langle w, \mu_0 - \mu_1\rangle}
e^{-\sigma^2 \|w\|^2_2}
dw.
\end{align*}
Noticing that the Fourier transform of a real and even function is also even, we conclude that $\psi^\wedge$ is also even.
This finishes the proof since $\|\theta_0 - \theta_1\|^2_{L^2(\R^d)} = \|\theta_1\|^2_{L^2(\R^d)} + \|\theta_0\|^2_{L^2(\R^d)} - 2\langle \theta_0, \theta_1\rangle_{L^2(\R^d)}$.
\end{proof}
Now we turn to the proof of Theorem \ref{thm:l2-distance-lower-bound}.
We will write $\tilde{\Lambda}_{\psi}$ to denote a non-negative finite measure, absolutely continuous with respect to the Lebesgue measure with density $(2\pi)^{d/2}(\psi^\wedge)^2$.
Then \eqref{eq:l2-distance-closed-form} in Lemma \ref{lemma:l2-distance-closed-form} can be written as
\begin{eqnarray}
\left\| \theta_0 - \theta_1\right\|_{L^2(\R^d)}^2
&{}={}&
\frac{2}{(2\pi)^{d/2}}
\int_{\R^d}
\left(1 - e^{i\langle w, \mu_0 - \mu_1\rangle}\right)
e^{-\sigma^2 \|w\|^2_2}
d\tilde{\Lambda}_{\psi}(w).\nonumber\\
&{}={}&
\frac{2}{(2\pi)^{d/2}}
\int_{\R^d}
\left(1 - \cos(\langle w, \mu_0 - \mu_1\rangle)\right)
e^{-\sigma^2 \|w\|^2_2}
d\tilde{\Lambda}_{\psi}(w),\nonumber
\end{eqnarray}
which is exactly of the form in Lemma~\ref{lemma:rkhs-distance} but with $\Lambda_\psi$ replaced by $\tilde{\Lambda}_{\psi}$.  
From the proof of Lemma \ref{lemma:l2-distance-closed-form}, since $\psi^\wedge$ is even and non-vanishing, the corresponding measure 
$\tilde{\Lambda}_{\psi}$ is symmetric and $\mathrm{supp}(\tilde{\Lambda}_{\psi}) = \R^d$. The result therefore follows by carrying out the proof of Proposition~\ref{thm:rkhs-distance-lower-bound} verbatim but for replacing $\Lambda_\psi$ with $\tilde{\Lambda}_{\psi}$.


\subsection{Proof of Theorem \ref{thm:radial-l2-lower-bound}}
\label{proof:radial-l2-lower-bound}
The proof will closely follow that of Theorem \ref{thm:radial-lower-bound}. Let $G(\mu_0, \sigma^2 I)$ and $G(\mu_1, \sigma^2 I)$ be two $d$-dimensional Gaussian distributions with mean vectors $\mu_0,\mu_1\in\R^d$ and variance $\sigma^2 > 0$.
Let $\theta_0$ and $\theta_1$ denote the kernel mean embeddings of $G(\mu_0, \sigma^2 I)$ and $G(\mu_1, \sigma^2 I)$ respectively.

\bigskip
\noindent\emph{(A) Deriving a closed form expression for $\| \theta_0 - \theta_1\|_{L^2(\R^d)}^2$.}

\medskip
The condition in \eqref{equation:worst-case-condition} ensures that $\psi_\nu\in L^2(\bb{R}^d)$ (see Remark~\ref{rem:l2-condition}). In fact, using a similar argument it can be shown that $\psi_\nu\in L^1(\bb{R}^d)$. Also it is easy to verify that $\psi_\nu\in C_b(\R^d)$.
Next, under this moment condition we may apply Tonelli-Fubini theorem to compute the Fourier transform of $\psi_\nu$:
\begin{equation}
\label{eq:radial-kernel-fourier}
\psi^\wedge_\nu(w)
=
\frac{1}{(2\pi)^{d/2}}
\int_{\R^d}
e^{-i \langle w, x\rangle}
\int_0^\infty e^{-t \|x\|^2_2}
d\nu(t)
dx
=
\int_0^\infty 
\frac{1}{(2 t)^{d/2}}
e^{-\frac{\|w\|^2_2}{4t}}
d\nu(t).
\end{equation}
It is immediate to see that $\psi^\wedge_\nu \in L_1(\R^d)$.
Therefore Lemma \ref{lemma:l2-distance-closed-form} yields
\begin{equation}
\label{eq:radial-l2-proof-distance}
\| \theta_0 - \theta_1\|_{L^2(\R^d)}^2
=
2\int_{\R^d}
\left(1 - e^{-i\langle w, \mu_0 - \mu_1\rangle}\right)
\left(\psi^\wedge_\nu(w)\right)^2
e^{-{\sigma^2 \|w\|^2_2}}
dw.
\end{equation}
Denoting $G(w) := \psi^\wedge_\nu(w)e^{-\frac{\sigma^2 \|w\|^2_2}{2}}$ and using a well-known property of the Fourier transform we get
\begin{equation}
\label{eq:l2-radial-fourier-of-product}
(G^2)^\wedge(\tau)=\frac{1}{(2\pi)^{d/2}}\int_{\R^d}
e^{-i\langle w, \tau\rangle}
G^2(w)
dw
=
\frac{1}{(2\pi)^{d/2}}\int_{\R^d}
G^\wedge(x)
G^\wedge(\tau - x)
dx.
\end{equation}
Next we compute the Fourier transform of $G$ using \eqref{eq:radial-kernel-fourier}:
\begin{align*}
G^\wedge(x)
&=
\frac{1}{(2\pi)^{d/2}}
\int_{\R^d}
e^{-i \langle w, x\rangle}
\left(
\int_0^\infty 
\frac{1}{(2 t)^{d/2}}
e^{-\frac{\|w\|^2_2}{4t}}
d\nu(t)
\right)
e^{-\frac{\sigma^2 \|w\|^2_2}{2}}
dw.
\end{align*}
Using the moment condition on $\nu$ we have
\begin{align}
\notag
\int_0^\infty
\frac{1}{(4 \pi t)^{d/2}}
\int_{\R^d}
e^{-\frac{\|x\|^2_2}{4t}}
e^{-\frac{\sigma^2 \|x\|^2_2}{2}}
dx
d\nu(t)
&=
\frac{1}{(2\sigma^2)^{d/2}}
\int_0^\infty
{\left(t + \frac{1}{2\sigma^2}\right)^{-d/2}}
d\nu(t)\\
\label{eq:radial-l2-tonelli-proof}
&\leq
\frac{1}{(2\sigma^2)^{d/2}}
\int_0^\infty
t^{-d/2}
d\nu(t) < \infty.
\end{align}
This allows us to use Tonelli-Fubini theorem and write
\begin{align*}
G^\wedge(x)
&=
\frac{1}{(2\pi)^{d/2}}
\int_0^\infty 
\frac{1}{(2 t)^{d/2}}
\int_{\R^d}
e^{-i \langle w, x\rangle}
e^{-\frac{\|w\|^2_2}{4t}}
e^{-\frac{\sigma^2 \|w\|^2_2}{2}}
dw\,
d\nu(t)\\
&=
\int_0^\infty 
\frac{1}{(2 t \sigma^2 + 1)^{d/2}}
\exp\left(-\frac{t\|x\|^2_2}{2t\sigma^2 + 1}\right)
d\nu(t).
\end{align*}
Returning to \eqref{eq:l2-radial-fourier-of-product} and denoting $\Delta_1 := 2t_1 \sigma^2 + 1$, $\Delta_2 := 2t_2 \sigma^2 + 1$ we obtain
\begin{align*}
(G^2)^\wedge(\tau)
&=
\int_{\R^d}
\int_0^\infty 
\int_0^\infty 
\frac{1}{(2\pi\Delta_1\Delta_2)^{d/2}}
\exp\left(-\frac{t_1\|x\|^2_2}{\Delta_1}
-\frac{t_2\|\tau - x\|^2_2}{\Delta_2}\right)
d\nu(t_1)\,
d\nu(t_2)\,
dx.
\end{align*}
Using a simple identity
\[
a \|x\|^2_2 + b\|x - y\|^2_2 = (a + b)\left\| x - \frac{b}{a + b}y\right\|^2_2 + \frac{ab}{a + b} \|y\|^2_2,
\]
which holds for any $x,y\in\R^d$ and $a,b\in \R$ with $a + b \neq 0$, we obtain
\begin{align*}
(G^2)^\wedge(\tau)&=
\int_{\R^d}
\int_0^\infty 
\int_0^\infty 
\frac{1}{(2\pi\Delta_1\Delta_2)^{d/2}}
\exp\left(
-\left( \frac{t_1}{\Delta_1} + \frac{t_2}{\Delta_2}\right)
\left\|
x - \frac{t_2 \Delta_1\tau}{t_1 \Delta_2 + t_2 \Delta_1}
\right\|^2_2\right)
\\
&\quad\quad\quad\quad\quad\quad\quad\quad\quad\quad\quad\quad\times
\exp\left(
-
\frac{t_1t_2\|\tau\|^2_2}{t_1 \Delta_2 + t_2 \Delta_1} 
\right)
d\nu(t_1)\;
d\nu(t_2)\;
dx.
\end{align*}
Using an argument similar to \eqref{eq:radial-l2-tonelli-proof} we can show that Tonelli-Fubini theorem is applicable to the r.h.s.~of the above equation. Therefore, changing the order of integration we get
\begin{align*}
(G^2)^\wedge(\tau)&=
\int_0^\infty 
\int_0^\infty 
\left(
\frac{1}{2(t_1 \Delta_2 + t_2 \Delta_1)}
\right)^{d/2}
\exp\left(
-
\frac{t_1t_2\|\tau\|^2_2}{t_1 \Delta_2 + t_2 \Delta_1} 
\right)
d\nu(t_1)\,
d\nu(t_2).
\end{align*}
Noticing that $(G^2)^\wedge(0)=\frac{1}{(2\pi)^{d/2}} \int_{\R^d}G^2(w)\,dw$ and returning to \eqref{eq:radial-l2-proof-distance} we get
\begin{eqnarray}
\lefteqn{\| \theta_0 - \theta_1\|_{L^2(\R^d)}^2
=2(2\pi)^{d/2}\left((G^2)^\wedge(0)-(G^2)^\wedge(\mu_0-\mu_1)\right)}\nonumber\\
&&=
2
\int_0^\infty \!\!
\int_0^\infty 
\left(
\frac{\pi}{t_1 \Delta_2 + t_2 \Delta_1}
\right)^{d/2}
\left(1 - 
\exp\left(
-
\frac{t_1t_2\|\mu_0 - \mu_1\|^2_2}{t_1 \Delta_2 + t_2 \Delta_1} 
\right)\right)
d\nu(t_1)
d\nu(t_2).\nonumber
\end{eqnarray}

\bigskip
\noindent\emph{(B) Lower bounding $\| \theta_0 - \theta_1\|_{L^2(\R^d)}^2$ in terms of $\|\mu_0 - \mu_1\|^2_2$.}

\medskip
Consider
\begin{align*}
&\| \theta_0 - \theta_1\|_{L_2(\R^d)}^2\\
&\geq
2
\int_{\delta_0}^{\delta_1} \!\!
\int_{\delta_0}^{\delta_1}
\left(
\frac{\pi}{t_1 \Delta_2 + t_2 \Delta_1}
\right)^{d/2}
\left(1 - 
\exp\left(
-
\frac{t_1t_2\|\mu_0 - \mu_1\|^2_2}{t_1 \Delta_2 + t_2 \Delta_1} 
\right)\right)
d\nu(t_1)
d\nu(t_2)\\
&=
2
\int_{\delta_0}^{\delta_1} \!\!
\int_{\delta_0}^{\delta_1}
\left(
\frac{\pi}{4t_1t_2\sigma^2 + t_1 + t_2}
\right)^{d/2}
\left(1 - 
\exp\left(
-
\frac{t_1t_2\|\mu_0 - \mu_1\|^2_2}{4t_1t_2\sigma^2 + t_1 + t_2} 
\right)\right)
d\nu(t_1)
d\nu(t_2).
\end{align*}
Using the fact that $1 - e^{-x} \geq \frac{x}{2}$ for $0\leq x \leq 1$, we obtain 
\begin{align}
\label{eq:radial-l2-explowerbound}
\| \theta_0 - \theta_1\|_{L^2(\R^d)}^2
&\geq
\int_{\delta_0}^{\delta_1} \!\!
\int_{\delta_0}^{\delta_1}
\left(
\frac{\pi}{4t_1t_2\sigma^2 + t_1 + t_2}
\right)^{d/2}
\frac{t_1t_2\|\mu_0 - \mu_1\|^2_2}{4t_1t_2\sigma^2 + t_1 + t_2}\,
d\nu(t_1)\,
d\nu(t_2)
\end{align}
whenever
\[
\frac{t_1t_2\|\mu_0 - \mu_1\|^2_2}{4t_1t_2\sigma^2 + t_1 + t_2} \leq 1.
\]
Note that the expression on the left hand side of the previous inequality is increasing both in $t_1$ and $t_2$.
This means that for $t_1,t_2\in[\delta_0, \delta_1]$ we have:
\[
\frac{t_1t_2\|\mu_0 - \mu_1\|^2_2}{4t_1t_2\sigma^2 + t_1 + t_2} \leq 
\frac{\delta_1\|\mu_0 - \mu_1\|^2_2}{4\delta_1\sigma^2 + 2}
\]
and thus \eqref{eq:radial-l2-explowerbound} holds whenever 
\[
\delta_1\|\mu_0 - \mu_1\|^2_2 \leq 4\delta_1\sigma^2 + 2
\]
which will be satisfied later. \eqref{eq:radial-l2-explowerbound} can be rewritten as
\begin{align}
\notag
\| \theta_0 - \theta_1\|_{L^2(\R^d)}^2
&\geq
\int_{\delta_0}^{\delta_1} \!\!
\int_{\delta_0}^{\delta_1}
\frac{(\pi)^{d/2}t_1t_2\|\mu_0 - \mu_1\|^2_2}{(4t_1t_2\sigma^2 + t_1 + t_2)^{d/2+1}}\,
d\nu(t_1)\,
d\nu(t_2)\\
\notag
&=
\int_{\delta_0}^{\delta_1} \!\!
\left(\frac{\pi}{t_1 + t_2}\right)^{d/2}
\int_{\delta_0}^{\delta_1}
\frac{\frac{t_1t_2}{t_1 + t_2}\|\mu_0 - \mu_1\|^2_2}{\bigl(4\frac{t_1t_2}{t_1 + t_2}\sigma^2 + 1\bigr)^{d/2+1}}\,
d\nu(t_1)\,
d\nu(t_2)\\
\label{eq:radial-l2-old-proof-repeated}
&\geq
\left(\frac{\pi}{2\delta_1}\right)^{d/2}
\int_{\delta_0}^{\delta_1} \!\!
\int_{\delta_0}^{\delta_1}
\frac{S(t_1,t_2)\|\mu_0 - \mu_1\|^2_2}{\bigl(4S(t_1,t_2)\sigma^2 + 1\bigr)^{d/2+1}}\,
d\nu(t_1)\,
d\nu(t_2),
\end{align}
where $S(t_1,t_2) := \frac{t_1t_2}{t_1 + t_2}$.
Note that $S(t_1,t_2)$ takes values in $[\frac{\delta_0}{2}, \frac{\delta_1}{2}]$ as $t_1$ and $t_2$ varies in $[\delta_0, \delta_1]$.
We can now repeat part of the proof of Theorem \ref{thm:radial-lower-bound} where we showed that the function
$
t\mapsto\frac{t}{(1 + 4t\sigma^2)^{(d+2)/2}}
$
monotonically increases on $[0, \frac{1}{2d\sigma^2}]$, reaches its global maximum at $t = \frac{1}{2d\sigma^2}$, and then decreases on $[\frac{1}{2d\sigma^2}, \infty)$.
Using this fact we have
\begin{align*}
&\int_{\delta_0}^{\delta_1} \!\!\!\!
\int_{\delta_0}^{\delta_1}
\frac{S(t_1,t_2)}{\bigl(4S(t_1,t_2)\sigma^2 + 1\bigr)^{\frac{d}{2}+1}}
d\nu(t_1)
d\nu(t_2)
\geq
\frac{\beta^2}{2}
\min\left\{
\frac{\delta_0}{(1 + 2\delta_0\sigma^2)^{\frac{d}{2}+1}},
\frac{\delta_1}{(1 + 2\delta_1\sigma^2)^{\frac{d}{2}+1}}
\right\}.
\end{align*}
By setting $\sigma^2 := \frac{1}{\delta_1 d}$ we ensure that $t = \frac{\delta_1}{2}$ is the global maximum of the function 
$t\mapsto\frac{t}{\bigl(1 + \frac{4t}{\delta_1 d}\bigr)^{\frac{d}{2}+1}}$
and thus
$
\frac{\delta_0}{(1 + 2\delta_0\sigma^2)^{\frac{d}{2}+1}}
\leq
\frac{\delta_1}{(1 + 2\delta_1\sigma^2)^{\frac{d}{2}+1}}.
$ 
Combining this with \eqref{eq:radial-l2-old-proof-repeated} we have
\begin{eqnarray}
\lefteqn{
\| \theta_0 - \theta_1\|_{L^2(\R^d)}^2
\geq
\frac{\beta^2\delta_0}{2\bigl(1 + \frac{2\delta_0}{\delta_1 d}\bigr)^{\frac{d}{2}+1}}
\left(\frac{\pi}{2\delta_1}\right)^{d/2}\|\mu_0 - \mu_1\|^2_2}\nonumber\\
&&\geq
\frac{\beta^2\delta_0}{2(1 + \frac{2}{d})^{\frac{d}{2}+1}}
\left(\frac{\pi}{2\delta_1}\right)^{d/2}\|\mu_0 - \mu_1\|^2_2
\geq
\frac{\beta^2\delta_0}{2e}\left(1 - \frac{2}{2+d}\right)
\left(\frac{\pi}{2\delta_1}\right)^{d/2}\|\mu_0 - \mu_1\|^2_2,\nonumber
\end{eqnarray}
where we used an analysis similar to \eqref{eq:integral-d-lower-bound-final}.

\bigskip
\noindent\emph{(C) Application of Theorem \ref{thm:Tsybakov}.}

\medskip
We finish by repeating all the remaining steps carried out in the proof of Theorem~\ref{thm:radial-lower-bound} (steps C.1, C.2, and C.3), where we set 
\[
\sigma^2 := \frac{1}{\delta_1 d}\quad \text{and} \quad c_{\nu} := \frac{1}{16\delta_1}\left(\log N - \frac{1}{N-1}\right).
\]

\section*{Acknowledgements}
The authors thank the action editor and two anonymous reviewers for their detailed comments, which helped to improve the presentation. The authors would also like to thank 
David Lopez-Paz, Jonas Peters, Bernhard Sch\"{o}lkopf, and Carl-Johann Simon-Gabriel for useful discussions.
\appendix
\section{$\sqrt{n}$-consistency of $\mu_k(P_n)$}\label{app:conc}
In the following, we present a general result whose special cases establishes the convergence rate of $n^{-1/2}$ for $\Vert\mu_k(P_n)-\mu_k(P)\Vert_{\Cal{F}}$ when $\Cal{F}=\Cal{H}_k$
and $\Cal{F}=L^2(\bb{R}^d)$.
\begin{appxpro}
\label{thm:general-upper-bound}
Let $(X_i)^n_{i=1}$ be random samples drawn i.i.d.~from $P$ defined on a separable topological space $\Cal{X}$. 
Suppose $r:\Cal{X}\rightarrow H$ is continuous and 
\begin{equation}\sup_{x\in\Cal{X}} \|r(x)\|_{H}^2 \leq C_k<\infty,\label{Eq:uniform-general}\end{equation}
where $H$ is a separable Hilbert space of real-valued functions. 
Then for any $0 < \delta \leq 1$ with probability at least $1-\delta$ we have
\[
\left\Vert \intx r(x)\,dP_n(x) - \intx r(x)\,dP(x) \right\Vert_{H}
\leq
\sqrt{\frac{C_k}{{n}}} + \sqrt{\frac{2C_k\log(1/\delta)}{n}}.
\]
\end{appxpro}
\begin{proof}
Note that $r:\Cal{X}\rightarrow H$ is a $H$-valued measurable function as $r$ is continuous and $H$ is separable \citep[Lemma A.5.18]{SC08}.
The condition in (\ref{Eq:uniform-general}) ensures that $\int \Vert r(x)\Vert_H dQ(x)\le\sqrt{C_k}<\infty$ for any $Q\in M^1_+(\Cal{X})$ and therefore $\int r(x)\,dQ(x)$ is well defined 
as a Bochner integral for any $Q\in M^1_+(\Cal{X})$ \cite[Theorem 2, p.45]{Diestel-77}. 
By McDiarmid's inequality, it is easy to verify that with probability at least $1-\delta$, 
\begin{eqnarray}\left\Vert \intx r(x)\,dP_n(x) - \intx r(x)\,dP(x) \right\Vert_{H}&{}\le{}& \bb{E}\left\Vert \intx r(x)\,dP_n(x) - \intx r(x)\,dP(x) \right\Vert_{H}\nonumber\\
&{}{}&\qquad\qquad+\sqrt{\frac{2C_k\log(1/\delta)}{n}},\label{Eq:Mcdi}
\end{eqnarray}
 where
\begin{eqnarray}
\lefteqn{\bb{E}\left\Vert\intx r(x)\,dP_n(x) - \intx r(x)\,dP(x)\right\Vert_{H}\le \sqrt{\bb{E}\left\Vert\intx r(x)\,dP_n(x) - \intx r(x)\,dP(x)\right\Vert^2_{H}}\nonumber}\\
&&=\sqrt{\bb{E}\left\Vert\intx r(x)\,dP_n(x)\right\Vert^2_{H}+\left\Vert \intx r(x)\,dP(x)\right\Vert^2_{H}-2\bb{E}\left\langle\intx r(x)\,dP_n(x),\intx r(x)\,dP(x)\right\rangle_H}\nonumber\\
&&=\sqrt{\bb{E}\left\Vert\intx r(x)\,dP_n(x)\right\Vert^2_{H}+\left\Vert \intx r(x)\,dP(x)\right\Vert^2_{H}-\frac{2}{n}\sum^n_{i=1}\bb{E}\left\langle r(X_i),\intx r(x)\,dP(x)\right\rangle_H}.\label{Eq:aba}
\end{eqnarray}
To simplify the r.h.s.~of \eqref{Eq:aba}, we make the following observation. Note that for any $g\in H$, $T_g:H\rightarrow \bb{R}$, $f\mapsto\langle g,f\rangle_H$ is a bounded linear functional on $H$. Choose $f=\intx r(y)\,dP(y)$. It follows from \citep[Theorem 6, p.47]{Diestel-77}
that 
\begin{equation}
\left\langle g,\intx r(y)\,dP(y)\right\rangle_H=T_g\left(\intx r(y)\,dP(y)\right)=\intx T_g(r(y))\,dP(y)
=\intx \langle g,r(y)\rangle_H\,dP(y).
\label{Eq:interchange}
\end{equation}
Applying \eqref{Eq:interchange} to the third term in the r.h.s.~of \eqref{Eq:aba} with $g=\intx r(x)\,dP(x)$, we obtain
$$\bb{E}\left\langle r(X_i),\intx r(x)\,dP(x)\right\rangle_H=\intx \left\langle r(x_i),g\right\rangle_H\,dP(x_i)=\left\langle \intx r(x_i)\,dP(x_i),g\right\rangle_H=\left\Vert g\right\Vert^2_H$$
and so (\ref{Eq:aba}) reduces to
\begin{equation}
\bb{E}\left\Vert\intx r(x)\,dP_n(x) - \intx r(x)\,dP(x)\right\Vert_{H}\le\sqrt{\bb{E}\left\Vert\intx r(x)\,dP_n(x)\right\Vert^2_{H}-\left\Vert \intx r(x)\,dP(x)\right\Vert^2_{H}}.
\label{Eq:aba-1}
\end{equation}
Consider 
\begin{eqnarray}
\bb{E}\left\Vert\intx r(x)\,dP_n(x)\right\Vert^2_{H}&{}={}&\bb{E}\left\Vert\frac{1}{n}\sum^n_{i=1}r(X_i)\right\Vert^2_H=\frac{1}{n^2}\sum^n_{i,j=1}\bb{E}\langle r(X_i),r(X_j)\rangle_H \nonumber\\
&{}={}&\frac{1}{n^2}\sum_{i=j}\bb{E}\langle r(X_i),r(X_j)\rangle_H+\frac{1}{n^2}\sum_{i\ne j}\bb{E}\langle r(X_i),r(X_j)\rangle_H\nonumber\\
&{}={}&\frac{1}{n}\bb{E}_{X\sim P}\Vert r(X)\Vert^2_H+\frac{n-1}{n}\bb{E}_{X\sim P,Y\sim P}\langle r(X),r(Y)\rangle_H.\label{Eq:firstterm}
\end{eqnarray} 
Using \eqref{Eq:interchange}, the second term in \eqref{Eq:firstterm} can be equivalently written as
\begin{eqnarray}
\bb{E}_{X\sim P,Y\sim P}\langle r(X),r(Y)\rangle_H&{}={}&\intx \left(\intx \langle r(x),r(y)\rangle_H\,dP(y)\right)\,dP(x)\nonumber\\
&{}\stackrel{(\star)}{=}{}&\intx \left\langle r(x),\intx r(y)\,dP(y)\right\rangle_H\,dP(x)\nonumber\\
&{}\stackrel{(\star)}{=}{}& \left\langle \intx r(x)\,dP(x),\intx r(y)\,dP(y)\right\rangle_Hx
=\left\Vert \intx r(x)\,dP(x)\right\Vert^2_H,
\nonumber
\end{eqnarray}
where we invoked (\ref{Eq:interchange}) in $(\star)$. Combining the above with \eqref{Eq:firstterm} and using the result in \eqref{Eq:aba-1} yields
$$\bb{E}\left\Vert\intx r(x)\,dP_n(x)-\intx r(x)\,dP(x)\right\Vert_{H}\le \sqrt{\frac{\bb{E}_{X\sim P}\Vert r(X)\Vert^2_{H}-\Vert \intx r(x)\,dP(x)\Vert^2_{H}}{n}}\le\sqrt{\frac{C_k}{n}}$$ and the result follows.
\end{proof}
\vspace{-7mm}
\begin{appxrem}\label{rem:rkhs-bound}
Suppose $H$ is an RKHS with a reproducing kernel $k$ that is continuous and satisfies $\sup_{x\in\Cal{X}}k(x,x)<\infty$. Choosing $r(x)=k(\cdot,x),\,x\in\Cal{X}$ in Proposition~\ref{thm:general-upper-bound} 
yields a concentration inequality for $\Vert \mu_k(P_n)-\mu_k(P)\Vert_H$ with $C_k:=\sup_{x\in\Cal{X}}k(x,x)$, thereby establishing a convergence rate of $n^{-1/2}$ for $\Vert\mu_k(P_n)-\mu_k(P)\Vert_{H}$.
While such a result has already appeared in \citet[Theorem 2]{Smola07Hilbert}, 
\citet{Gretton12:KTT} and \citet{LMBT15}, the result derived from
Proposition~\ref{thm:general-upper-bound} improves upon them by providing better constants. While all these works including Proposition~\ref{thm:general-upper-bound} are based on 
McDiarmid's inequality (see \eqref{Eq:Mcdi}), the latter obtains better constants by carefully bounding the expectation term in \eqref{Eq:Mcdi}. It is easy to verify that $C_k=1$ for Gaussian and $C_k=C_M$ for mixture of Gaussian kernels, $C_k=c^{-2\gamma}$ for inverse multiquadrics, and $C_k=1$ for Mat\'{e}rn kernels.\vspace{-2mm}
\end{appxrem}
\begin{appxrem}\label{rem:l2-condition}
Assuming $\Cal{X}=\bb{R}^d$, $H=L^2(\R^d)$ and $r(x)=k(\cdot,x),\,x\in\R^d$, where $k$ is a continuous positive definite kernel on $\bb{R}^d$, Proposition~\ref{thm:general-upper-bound}
establishes a convergence rate of $n^{-1/2}$ for $\Vert\mu_k(P_n)-\mu_k(P)\Vert_{L^2(\R^d)}$ under the condition that $\sup_{x\in\R^d} \|k(x,\cdot)\|^2_{L^2(\bb{R}^d)}<\infty$. If $k$ is translation invariant on $\bb{R}^d$, i.e., $k(x,y)=\psi(x-y),\,x,y\in\bb{R}^d$
where $\psi\in C(\bb{R}^d)$ is positive definite, then $\psi\in L^2(\bb{R}^d)$ ensures that $\sup_{x\in\R^d} \|k(x,\cdot)\|^2_{L^2(\bb{R}^d)}=
\sup_{x\in\R^d} \|\psi(x - \cdot)\|_{L^2(\bb{R}^d)}^2=\|\psi\|_{L^2(\bb{R}^d)}^2$ and therefore Propositions~\ref{thm:general-upper-bound} holds with $C_k:=\|\psi\|_{L^2(\bb{R}^d)}^2$. 
On the other hand, for radial kernels on $\bb{R}^d$, i.e., kernels of the form in~(\ref{eq:radial-equivalent}), the condition in (\ref{Eq:uniform-general}) is ensured if 
\begin{eqnarray}
\int_0^{\infty}
t^{-d/2}
d\nu(t)<\infty\label{Eq:radial-cond}
\end{eqnarray}
since 
\begin{eqnarray*}
\sup_{x\in\R^d} \|k(x,\cdot)\|_{L^2(\bb{R}^d)}^2
&{}={}&
\sup_{x\in\R^d}\int_{\R^d}
\left(
\int_0^\infty e^{-t\|x-y\|^2}d\nu(t)
\right)^2
dy\\
&{}\stackrel{(\dagger)}{\leq}{}&
\nu([0,\infty))
\sup_{x\in\R^d}
\int_{\R^d}
\int_0^{\infty} e^{-2t \|x - y\|^2} d\nu(t)
dy\\
&{}\stackrel{(\ddagger)}{=}{}&
\nu([0,\infty))
\sup_{x\in\R^d}
\int_0^{\infty}
\int_{\R^d}
e^{-2t \|x - y\|^2}
dy
\,
d\nu(t)=\frac{\nu([0,\infty))}{(2/\pi)^{d/2}}\int^\infty_0\frac{d\nu(t)}{t^{d/2}},
\end{eqnarray*}
where we used Jensen's inequality in $(\dagger)$ and Fubini's theorem in $(\ddagger)$. Therefore the bound in Proposition~\ref{thm:general-upper-bound} holds with $C_k:=\nu([0,\infty))
\int_0^{\infty}
\left(\frac{\pi}{2t}\right)^{d/2}
\,
d\nu(t)$. (\ref{Eq:radial-cond}) is satisfied by Gaussian, mixture of Gaussian and Mat\'{e}rn
(see \citealp[Equation 6.17]{S15}) kernels. For inverse multiquadrics, while (\ref{Eq:radial-cond}) holds for $\gamma>d/2$ since $\nu=c^{-2\gamma}\text{Gamma}(\gamma,c^2)$ (see \citealp[Theorem 7.15]{W05}), in fact the condition
in (\ref{Eq:uniform-general}) holds for $\gamma>d/4$ (see Lemma~\ref{lemma:im-l2-norm}).
\end{appxrem}
\section{Minimax Lower Bounds and Le Cam's Method}\label{app:sec:lecam}
Let $\Theta$ be a set of parameters (or functions) containing the element $\theta$ which we want to estimate.
Assume there is a class $\mathcal{P} = \{P_{\theta}: \theta\in \Theta\}$ of probability measures on $\R^d$ indexed by $\Theta$. Suppose 
$d\colon \Theta \times \Theta \to [0,\infty)$ is a metric on $\Theta$.
Le Cam's method provides a lower bound on the minimax probability,
$\inf_{\hat{\theta}_n}\sup_{\theta\in\Theta} P^n_\theta(d(\hat{\theta}_n,\theta)\ge s)$ for $s>0$, where the infimum is taken over all possible estimators $\hat{\theta}_n\colon \R^d\to \Theta$ that are constructed from an i.i.d.\:sample $(X_i)^n_{i=1}$ drawn from $P_\theta$.
The following two results which we used throughout this work are based on Le Cam's method and they provide a lower bound on the minimax probability. 
The first one follows from Theorem 2.2 and Equation (2.9) of \cite{T08}.
It requires a construction of two sufficiently distant elements of the set $\Theta$ corresponding to the probability distributions similar in the Kullback-Leibler (KL) divergence sense, where the 
KL divergence
between two distributions $P$ and $Q$ with $P$ absolutely continuous w.r.t.~$Q$ is defined as $\mathrm{KL}(P\|Q)=\int \log\frac{dP}{dQ}\,dP$.
\begin{appxthm}[Lower bound based on two hypotheses]
\label{thm:Tsybakov-two}
Assume $\Theta$ contains $\theta_0$ and $\theta_1$ such that 
$d(\theta_0, \theta_1)\geq 2s$ and $\mathrm{KL}(P_{\theta_0}^n\|P_{\theta_1}^n) \leq \alpha$ for some $s>0$ and $0 < \alpha < \infty$.
Then
\[
\inf_{\hat{\theta}_n}
\sup_{\theta \in \Theta}
P^n_{\theta}\left\{
d(\hat{\theta}_n,\theta)
\geq s
\right\}
\geq
\max\left(
\frac{1}{4}e^{-\alpha},
\frac{1 - \sqrt{\alpha / 2}}{2}
\right).
\]
\end{appxthm}
Note that the second condition of the theorem bounds the distance between the $n$-fold product distributions by a constant independent of $n$.
Recalling the chain rule of the KL-divergence, which states that $\mathrm{KL}(P_{\theta_0}^n\|P_{\theta_1}^n) = n\cdot \mathrm{KL}(P_{\theta_0}\|P_{\theta_1})$, 
we can see that this condition is rather restrictive and requires the marginal distributions to satisfy $\mathrm{KL}(P_{\theta_0}\|P_{\theta_1}) = O(n^{-1})$.
This condition is slightly relaxed in the following result, which follows from Theorem 2.5  of \cite{T08}.
\begin{appxthm}[Lower bound based on many hypotheses]
\label{thm:Tsybakov}
Assume $M\geq 2$ and suppose that there exist $\theta_0,\dots,\theta_M\in\Theta$ such that\vspace{1mm}
(i) $d(\theta_i,\theta_j)\geq 2s>0$,\: $\forall\, 0\leq i < j \leq M$;\vspace{1mm}
(ii) $P_{\theta_j}$ is absolutely continuous w.\,r.\,t.\:$P_{\theta_0}$ for all $j=1,\dots,M$, and
$
\frac{1}{M}\sum_{i=1}^M \mathrm{KL}(P_{\theta_j}^n\| P_{\theta_0}^n) \leq \alpha \log M
$
with $0<\alpha<1/8$.
Then 
\[
\inf_{\hat{\theta}_n}
\sup_{\theta \in \Theta}
P^n_{\theta}\left\{
d (\hat{\theta}_n , \theta)
\geq s
\right\}
\geq
\frac{\sqrt{M}}{1 + \sqrt{M}}
\left(
1 - 2\alpha - \sqrt{\frac{2\alpha}{\log M}}
\right)>0.
\]
\end{appxthm}

The above result is commonly used with $M$ tending to infinity as $n\to \infty$.
In this case the second condition on the KL-divergence indeed becomes less restrictive than the one of Theorem~\ref{thm:Tsybakov-two}, since the upper bound $\alpha \log M$ may now grow with the sample size~$n$.
At the same time, Theorem \ref{thm:Tsybakov} still provides a lower bound on the minimax probability independent of $n$, since ${\sqrt{M}}/({1 + \sqrt{M}})$ and $\log M$ can be 
lower bounded by $1/2$ and $\log 2$ respectively.
%

\section{Technical Lemmas}
The following technical results are used to prove the main results of Sections~\ref{section:rkhs} and \ref{Sec:estimation-l2}.
\begin{appxlem}[Theorem 5.18, \citealp{W05}]
For any $\mu\in\R^d$ and $\sigma^2>0$ the following holds:
\label{lemma:Fourier-Gaussian}
\[
\left[
\frac{1}{(2\pi\sigma^2)^{d/2}}e^{-\frac{\|x - \mu\|^2_2}{2\sigma^2}} \right]^\wedge(w)
= 
\frac{1}{(2\pi)^{d/2}}
\exp\left(-i \langle \mu, w\rangle - \frac{\sigma^2\|w\|^2_2}{2}\right),
\quad w\in\R^d.
\]
\end{appxlem}

\begin{appxlem}
\label{lemma:psi-symmetric}
Let $\psi\colon \R^d \to \R$ be a symmetric and positive definite function.
Let $\Lambda_{\psi}$ be the corresponding finite non-negative Borel measure from \eqref{Eq:Bochner}.  
Then $\Lambda_{\psi}$ is symmetric, i.e., $\Lambda_{\psi}(A) = \Lambda_{\psi}(-A)$ for all $A\subset\R^d$.
\end{appxlem}
\begin{proof}
From the definition of $\Lambda_{\psi}$ we know that it is finite, non-negative, and
\[
\psi(x) =\int_{\R^d} e^{-i\langle w,x\rangle} \Lambda_{\psi}(dw) =
\int_{\R^d} \cos(\langle w,x\rangle) \Lambda_{\psi}(dw)
-
i\cdot\int_{\R^d} \sin(\langle w,x\rangle) \Lambda_{\psi}(dw).
\]
Since $\psi(-x) = \psi(x)$ for all $x\in \R^d$, we get $
\int_{\R^d} \sin(\langle w,x\rangle) \Lambda_{\psi}(dw) = 0.$
Note that $\psi(-x)$ is by definition a characteristic function of measure $\Lambda_{\psi}$, and we have just proved that it is real-valued.
It is known \cite[Corollary 3.8.7]{B07} that in this case the measure $\Lambda_{\psi}$ is invariant under the mapping $x\to -x$.
\end{proof}
\vspace{-6mm}
\begin{appxlem}
\label{lemma:continuous-compact-inf}
Assume $X,Y \subseteq \R^d$.
If $f\colon X \times Y \to \R$ is a continuous function and $Y$ is a compact set then $g(x) := \inf_{y\in Y}f(x,y)$ is continuous.\vspace{-2mm}
\end{appxlem}
\begin{proof}
First, the map $g\colon X \to \R$ is well defined since $f_x(y) := f(x,y)$ is a continuous function for any $x\in X$ and  thus $f_x$ achieves its infimum since $Y$ is a compact set.
We will prove that the map $g\colon X \to \R$ is continuous by showing that $g^{-1}(-\infty,a)$ and $g^{-1}(a,\infty)$ are open sets for all $a\in \R$ \cite[Corollary 2.2.7 (a)]{D02}.

Now we will show that $g^{-1}(-\infty,a)$ is open for any $a\in\R$.
It suffices to show that for any $x\in g^{-1}(-\infty,a)$ there is an open neighborhood $U_x$ of $x$ which also belongs to $g^{-1}(-\infty,a)$.
The set $g^{-1}(-\infty,a)$ consists of elements $x \in X$ for which $g(x) < a$.
In other words, it consists of such elements $x \in X$ for which there is corresponding $y_x \in Y$ satisfying ${f(x, y_x) < a}$.
Take any $x \in g^{-1}(-\infty,a)$.
Since $f$ is continuous, $f^{-1}(-\infty, a)$ is open and contains $(x,y_x)$.
Moreover $f^{-1}(-\infty, a)$ contains $U_x \times V_y$, where $U_x$ and $V_y$ are open sets with $x\in U_x$ and $y_x \in V_y$.
Now suppose $x' \in U_x$. Then for any $y \in V_y$ we have $f(x', y) < a$. In particular, $f(x', y_x) < a$, which means that $g(x') < a$ and $x' \in g^{-1}(-\infty,a)$. This shows that $g^{-1}(-\infty,a)$ is open.

Next we will show that $g^{-1}(a, \infty)$ is also an open set for any $a\in\R$.
Assume this is not the case.
Then there is $x\in g^{-1}(a, \infty)$ such that for any neighborhood $U_x$ of $x$ there is a point $x'\in U_x$ such that $ x' \not\in g^{-1}(a, \infty)$.
This means that for any such $x'$ there is $y_{x'}$ satisfying $f(x', y_{x'}) \leq a$.
Using this we can construct a sequence $\{x_n, y_n\}$ from $X\times Y$, such that $x_n\notin g^{-1}(a,\infty)$ for every $n$, $\lim_{n\to \infty} x_n = x$ and for any $n$ it holds that $f(x_n, y_n) \leq a$.
Since $Y$ is compact we conclude that $\{y_n\}$ has a converging subsequence $\{y_{n(k)}\}$ \cite[Theorem 2.3.1]{D02} with limit $y^* \in Y$.
We just showed that there is a sequence $\{x_{n(k)}, y_{n(k)}\}$ in $X\times Y$, which converges to $(x, y^*)$, such that $\lim_{k\to \infty} f(x_{n(k)}, y_{n(k)}) \leq a$.
Since $f$ is continuous, this also means that $\lim_{k\to \infty} f(x_{n(k)}, y_{n(k)}) = f(x, y^*) \leq a$. This means that $\inf_{y\in Y}f(x,y) \leq f(x,y^*) \leq a$. In other words, 
this shows that $x\not \in g^{-1}(a,\infty)$ leading to a contradiction and therefore $g^{-1}(a, \infty)$ is open.
\end{proof}
\vspace{-4mm}

\begin{appxlem}[$L^2$ norm of inverse multiquadrics kernels]
\label{lemma:im-l2-norm}
For any $c>0$ and $\gamma > \frac{d}{4}$,
\[
\int_{\R^d}(c^2 + \|x\|^2_2)^{-2\gamma}\,dx = c^{d-4\gamma}\pi^{d/2}\frac{\Gamma(2\gamma - \frac{d}{2})}{\Gamma(2\gamma)}.
\]
\end{appxlem}
\begin{proof}
\begin{align*}
\int_{\R^d}(c^2 + \|x\|^2_2)^{-2\gamma}dx 
&=
c^{-4\gamma}\int_{\R^d}\left(1 + \left\|\frac{x}{c}\right\|^2_2\right)^{-2\gamma}dx
=
c^{d-4\gamma}\int_{\R^d}\left(1 + \|x\|^2_2\right)^{-2\gamma}dx\\
&=
c^{d-4\gamma}\frac{2\pi^{d/2}}{\Gamma(d/2)}\int_{0}^\infty\left(1 + r^2\right)^{-2\gamma}r^{d-1}dr\\
&=
c^{d-4\gamma}\frac{\pi^{d/2}}{\Gamma(d/2)}\int_{0}^\infty\left(1 + x\right)^{-2\gamma}x^{d/2-1}dx\\
&=
c^{d-4\gamma}\frac{\pi^{d/2}}{\Gamma(d/2)}\frac{\Gamma(d/2)\Gamma(2\gamma - d/2)}{\Gamma(2\gamma)},
\end{align*}
where last identity can be found in \cite[3.194.3]{GR00}.
\end{proof}
\vspace{-4mm}
\section{Bounds on Constants for Various Radial Kernels}
\label{appendix:bounds-constants}
In this appendix, we present bounds on the constants that appear in Corollaries~\ref{thm:radial-lower-bound-discrete}, \ref{thm:radial-l2-lower-bound-discrete} and 
Theorems~\ref{thm:radial-lower-bound}, 
\ref{thm:radial-l2-lower-bound}.
for various radial kernels.
\subsection{$\alpha$ in Corollary~\ref{thm:radial-lower-bound-discrete}}
\label{alpha:cor-radial1}
In Corollary~\ref{thm:radial-lower-bound-discrete}, we assumed that there exist $0<t_1<\infty$ and $\alpha>0$ such that $\nu([t_1,\infty))\ge\alpha$. In the following, we present the values of $t_1$
and $\alpha$ for various radial kernels.\vspace{3mm}\\
\emph{(i) Gaussian kernel}: $\nu=\delta_{\frac{1}{2\eta^2}}$ and so for any $t_1<\frac{1}{2\eta^2}$, we obtain $\alpha=1$.\vspace{1mm}\\
\emph{(ii) Mixture of Gaussians}: $\nu=\sum^M_{i=1}\beta_i\delta_{\frac{1}{2\eta^2_i}}$ and so $\alpha=C_M$ for any $t_1<\frac{1}{2\eta^2_1}$.\vspace{2mm}\\
\emph{(iii) Inverse multiquadric kernel}: It follows from \cite[Theorem 7.15]{W05} that \begin{equation}
\label{eq:im-nu-form}
k(x,y) = \int_{0}^\infty e^{-t \|x - y\|^2_2} \frac{t^{\gamma - 1}e^{-c^2 t}}{\Gamma(\gamma)} dt,
\end{equation}
and so 
 \begin{equation}\nu=c^{-2\gamma}\text{Gamma}(\gamma,c^2)\label{Eq:gam}\end{equation} where the density of a Gamma distribution with parameters
 $a,b>0$ is defined as $$\text{Gamma}(t;a,b)=\frac{b^a}{\Gamma(a)}t^{a-1}e^{-tb},\,\,t\ge 0.$$ Therefore
 choosing $t_1$ to be the median of $\nu$, we obtain $\alpha=\frac{c^{-2\gamma}}{2}$.\vspace{2mm}\\
\emph{(iv) Mat\'{e}rn kernel}: We know from \cite[Theorem 6.13]{W05} that Mat\'ern kernel is related to the Fourier transform of the inverse multiquadric kernel as
\[
\frac{1}{(2\pi)^{d/2}}
\int_{\R^d}
e^{-i\langle v, w\rangle}
\bigl(c^2 + \|v\|^2_2\bigr)^{-\tau}
dv
=
\frac{2^{1-\tau}}{\Gamma(\tau)}
K_{d/2 - \tau}(c\|w\|_2)
\left(
\frac{\|w\|_2}{c}
\right)^{\tau - d/2},
\]
where $c>0$ and $\tau>d/2$. Using this together with the representation \eqref{eq:im-nu-form} of an inverse multiquadrics kernel we obtain the following identity, 
which already appeared in \citep[Equation (72)]{S15}:
\begin{align}
\notag
k(x,y)
&=
\frac{\Gamma(\tau)c^{2\tau - d}2^{ d/2}}{\Gamma(\tau - d/2)}
\frac{1}{(2\pi)^{d/2}}\int_{\R^d}
e^{-i\langle v, x - y\rangle}
\left[\int_{0}^\infty e^{-t \|v\|^2_2} \frac{t^{\tau - 1}	e^{-c^2 t}}{\Gamma(\tau)} dt\right]
dv\\
\notag
&\stackrel{(\star)}{=}
\frac{c^{2\tau - d}2^{ d/2}}{\Gamma(\tau - d/2)}
\int_{0}^\infty
{t^{\tau - 1}	e^{-c^2 t}}
\frac{1}{(2\pi)^{d/2}}\int_{\R^d}
e^{-i\langle v, x - y\rangle}
 e^{-t \|v\|^2_2}
dv\,
dt\\
\notag
&=
\frac{c^{2\tau - d}2^{ d/2}}{\Gamma(\tau - d/2)}
\int_{0}^\infty
{t^{\tau - 1}	e^{-c^2 t}}
\frac{1}{(2t)^{d/2}}
e^{-\frac{\|x - y\|^2_2}{4t}}
dt\\
\notag
&=
\frac{c^{2\tau - d}}{\Gamma(\tau - d/2)}
\int_{0}^\infty
{t^{\tau - d/2 - 1}	e^{-c^2 t}}
e^{-\frac{\|x - y\|^2_2}{4t}}
dt,
\end{align}
where we invoked Tonelli-Fubini theorem \cite[Theorem 4.4.5]{D02} in $(\star)$ since $\bigl(c^2 + \|\cdot\|^2_2\bigr)^{-\tau}\in L_1(\R^d)$ for $\tau > d/2$.
After change of variables we finally obtain
\begin{equation}
k(x,y)
=
\frac{1}{\Gamma\left(\tau - \frac{d}{2}\right)}
\left(\frac{c^2}{4}\right)^{\tau - \frac{d}{2}}
\int_{0}^\infty
e^{-t{\|x - y\|^2}}
{t^{d/2 - \tau - 1}	e^{-\frac{c^2}{ 4t}}}
dt,\nonumber
\end{equation}
which shows that Mat\'ern kernel is a particular instance of radial kernels with 
\begin{equation}\nu= \mathrm{InvGamma}\left(\tau - \frac{d}{2}, \frac{c^2}{4}\right),\nonumber
\end{equation}
where the density of an inverse-Gamma distribution with parameters $a,b>0$ has the form
\[
\mathrm{InvGamma}(t; a,b) = \frac{b^a}{\Gamma(a)}t^{-a-1}e^{-{b}/{t}}.
\]
Therefore $\alpha=\frac{1}{2}$ for the choice of $t_1$ to be the median of $\nu$.

\subsection{$\frac{\beta t_0}{t_1}$ in Theorem~\ref{thm:radial-lower-bound}}
\label{beta:thm-radial1}
In Theorem~\ref{thm:radial-lower-bound}, we assumed that there exist $0<t_0\le t_1<\infty$ and $0<\beta<\infty$ such that $\nu([t_0,t_1])\ge\beta$. Define $B_k:=\frac{\beta t_0}{t_1}$. 
In the following, we present the values of $B_k$ for various radial kernels. 
\vspace{3mm}\\
\emph{(i) Gaussian kernel:} Choose $t_0=t_1=\frac{1}{2\eta^2}$ so that $\beta=1$ and $B_k=1$.\vspace{1mm}\\
\emph{(ii) Mixture of Gaussians:} Set $t_0 = \frac{1}{2\eta^2_1}$, $t_1 = \frac{1}{2\eta^2_M}$ so that $\beta= C_M$ implying $B_k=\frac{C_M\eta^2_M}{\eta^2_1}$. 
\vspace{2mm}\\
\emph{(iii) Inverse multiquadric kernel:} 
%
%
From \eqref{Eq:gam}, we have $\nu=c^{-2\gamma}\text{Gamma}(\gamma,c^2)$. Therefore 
\begin{align*}
\nu\left(\left[
\frac{\gamma}{2c^2},
\frac{\gamma}{c^2}\right]
\right)
&=
\frac{1}{\Gamma(\gamma)}
\int_{\gamma/(2c^2)}^{\gamma/c^2}
t^{\gamma - 1}
e^{-tc^2}
dt\\
&\geq\begin{cases}
\frac{1}{\Gamma(\gamma)}
\left(\frac{\gamma}{2c^2}\right)^{\gamma - 1}
\exp\left(
- \frac{\gamma}{c^2} c^2
\right)
\frac{\gamma}{2c^2}
,&\text{ for } \gamma \geq 1;\\
\frac{1}{\Gamma(\gamma)}
\left(\frac{\gamma}{c^2}\right)^{\gamma - 1}
\exp\left(
- \frac{\gamma}{c^2} c^2
\right)
\frac{\gamma}{2c^2},&\text{ for } \gamma\in(0,1).
\end{cases}
\end{align*}
Therefore with $t_0=\frac{\gamma}{2c^2}$, $t_1=\frac{\gamma}{c^2}$ and 
$\beta=\begin{cases}\frac{c^{-2\gamma}}{\Gamma(\gamma)}\left(\frac{\gamma}{2e}\right)^{\gamma},&\text{ for } \gamma \geq 1;\\ 
\frac{c^{-2\gamma}}{2\Gamma(\gamma)}\left(\frac{\gamma}{e}\right)^{\gamma},&\text{ for } \gamma\in(0,1)
\end{cases},$
we obtain  $$B_k=\begin{cases}\frac{c^{-2\gamma}}{2\Gamma(\gamma)}\left(\frac{\gamma}{2e}\right)^{\gamma},&\text{ for } \gamma \geq 1;\\ 
\frac{c^{-2\gamma}}{4\Gamma(\gamma)}\left(\frac{\gamma}{e}\right)^{\gamma},&\text{ for } \gamma\in(0,1)
\end{cases}.$$
\vspace{2mm}\\
\emph{(iv) Mat\'{e}rn kernel:} 
It is easy to check that if $X\sim \mathrm{Gamma}(a,b)$ and $Y\sim \mathrm{InvGamma}(a,b)$ for $a,b>0$ then for any $0<x \leq y < \infty$ the following holds:
\[
\mathbb{P}\{ x \leq X \leq y\}
=
\mathbb{P}\{ 1/y \leq Y \leq 1/x\}.
\]
This means, the above calculations for inverse multiquadrics can be used to obtain the following for the Mat\'{e}rn kernel: 
$$B_k=\begin{cases}\frac{1}{2\Gamma(\tau-\frac{d}{2})}\left(\frac{2\tau-d}{4e}\right)^{\tau-\frac{d}{2}},&\text{ for } \tau-\frac{d}{2} \geq 1;\\ 
\frac{1}{4\Gamma(\tau-\frac{d}{2})}\left(\frac{2\tau-d}{2e}\right)^{\tau-\frac{d}{2}},&\text{ for } \tau-\frac{d}{2}\in(0,1)
\end{cases}.$$
\subsection{$\beta^2\delta^{-d/2}_1$ in Corollary~\ref{thm:radial-l2-lower-bound-discrete}}
\label{Ak:cor-radial}
In Corollary~\ref{thm:radial-l2-lower-bound-discrete}, we assumed that there exist $0<\delta_0\le \delta_1<\infty$ and $0<\beta<\infty$ such that $\nu([\delta_0,\delta_1])\ge\beta$. Define $A_k:=\beta^2\delta^{-d/2}_1$. 
Based on the analysis carried out in Appendix~\ref{beta:thm-radial1}, in the following, we present the values of $A_k$ for various radial kernels.\vspace{3mm}\\
\emph{(i) Gaussian kernel:} Choose $\delta_0=\delta_1=\frac{1}{2\eta^2}$ so that $\beta=1$ and $A_k=(2\eta^2)^{d/2}$.\vspace{1mm}\\
\emph{(ii) Mixture of Gaussians:} Set $\delta_0 = \frac{1}{2\eta^2_1}$, $\delta_1 = \frac{1}{2\eta^2_M}$ so that $\beta= C_M$ implying $A_k=C_M^2(2\eta^2_M)^{d/2}$.\vspace{1mm}\\
\emph{(iii) Inverse multiquadric kernels:} Choosing $\delta_0=t_0$ and $\delta_1=t_1$ as in Appendix~\ref{beta:thm-radial1}, we obtain
$$A_k=\begin{cases}\frac{c^{d-4\gamma}}{\Gamma^2(\gamma)}\frac{\gamma^{2\gamma-\frac{d}{2}}}{(2e)^{2\gamma}},&\text{ for } \gamma \geq 1;\\ 
\frac{c^{d-4\gamma}}{4\Gamma^2(\gamma)}\frac{\gamma^{2\gamma-\frac{d}{2}}}{e^{2\gamma}},&\text{ for } \gamma\in(0,1)
\end{cases}.$$
\emph{(iv) Mat\'{e}rn kernel:} Define $\tilde{\gamma}:=\tau-\frac{d}{2}$ and $\tilde{c}:=\frac{c}{2}$. Choosing $\delta_0=\frac{\tilde{c}^2}{\tilde{\gamma}}$ and 
$\delta_1=\frac{2\tilde{c}^2}{\tilde{\gamma}}$, we obtain 
$$\beta=\begin{cases}\frac{1}{\Gamma(\tilde{\gamma})}\left(\frac{\tilde{\gamma}}{2e}\right)^{\tilde{\gamma}},&\text{ for } \tilde{\gamma} \geq 1;\\ 
\frac{1}{2\Gamma(\tilde{\gamma})}\left(\frac{\tilde{\gamma}}{e}\right)^{\tilde{\gamma}},&\text{ for } \tilde{\gamma}\in(0,1)
\end{cases},$$
using the analysis in Appendix~\ref{beta:thm-radial1}. Therefore,
$$A_k=\begin{cases}\frac{c^{-d}e^{-2\tilde{\gamma}}}{\Gamma^2(\tilde{\gamma})}\left(\frac{\tilde{\gamma}}{2}\right)^{2\tilde{\gamma}+\frac{d}{2}},&\text{ for } \tilde{\gamma} \geq 1;\\ 
\frac{c^{-d}e^{-2\tilde{\gamma}}}{\Gamma^2(\tilde{\gamma})}\frac{\tilde{\gamma}^{2\tilde{\gamma}+\frac{d}{2}}}{2^{2+\frac{d}{2}}},&\text{ for } \tilde{\gamma}\in(0,1)
\end{cases}.$$

\subsection{$\beta^2\delta_0\delta^{-\frac{d+2}{2}}_1$ in Theorem~\ref{thm:radial-l2-lower-bound}}
\label{Bk:thm-radial2}
In Theorem~\ref{thm:radial-l2-lower-bound}, we assumed that there exist $0<\delta_0\le \delta_1<\infty$ and $0<\beta<\infty$ such that $\nu([\delta_0,\delta_1])\ge\beta$. Define 
$B_k:=\beta^2\delta_0\delta^{-\frac{d+2}{2}}_1$. 
Based on the analysis carried out in Appendix~\ref{beta:thm-radial1}, in the following, we present the values of $B_k$ for various radial kernels.\vspace{3mm}\\
\emph{(i) Gaussian kernel:} Choose $\delta_0=\delta_1=\frac{1}{2\eta^2}$ so that $\beta=1$ and $B_k=(2\eta^2)^{d/2}$.\vspace{1mm}\\
\emph{(ii) Mixture of Gaussians:} Set $\delta_0 = \frac{1}{2\eta^2_1}$, $\delta_1 = \frac{1}{2\eta^2_M}$ so that $\beta= C_M$ implying $B_k=\frac{C_M^22^{d/2}\eta^{d+2}_M}{\eta^2_1}$.\vspace{1mm}\\
\emph{(iii) Inverse multiquadric kernels:} Choosing $\delta_0=t_0$ and $\delta_1=t_1$ as in Appendix~\ref{beta:thm-radial1}, we obtain
$$B_k=\begin{cases}\frac{c^{d-4\gamma}}{2\Gamma^2(\gamma)}\frac{\gamma^{2\gamma-\frac{d}{2}}}{(2e)^{2\gamma}},&\text{ for } \gamma \geq 1;\\ 
\frac{c^{d-4\gamma}}{8\Gamma^2(\gamma)}\frac{\gamma^{2\gamma-\frac{d}{2}}}{e^{2\gamma}},&\text{ for } \gamma\in(0,1)
\end{cases}.$$
\emph{(iv) Mat\'{e}rn kernel:} With the choice of $\delta_0$ and $\delta_1$ as in Appendix~\ref{Ak:cor-radial}, we obtain 
$$B_k=\begin{cases}\frac{c^{-d}e^{-2\tilde{\gamma}}}{2\Gamma^2(\tilde{\gamma})}\left(\frac{\tilde{\gamma}}{2}\right)^{2\tilde{\gamma}+\frac{d}{2}},&\text{ for } \tilde{\gamma} \geq 1;\\ 
\frac{c^{-d}e^{-2\tilde{\gamma}}}{\Gamma^2(\tilde{\gamma})}\frac{\tilde{\gamma}^{2\tilde{\gamma}+\frac{d}{2}}}{2^{3+\frac{d}{2}}},&\text{ for } \tilde{\gamma}\in(0,1)
\end{cases}.$$

\section{Alternate Proof of Theorem \ref{thm:radial-lower-bound}}
\label{appendix:alternative}
In Theorem \ref{thm:radial-lower-bound} we presented a minimax lower bound for radial kernels based on an appropriate construction of $d$-dimensional Gaussian distributions.
By a clever choice of the variance $\sigma^2$, which decays to zero as $d\to\infty$, we obtained a lower bound of the order $\Omega(n^{-1/2})$ independent of $d$.
This result was based on the direct analysis and special properties of radial kernels.
In this appendix we will show that we can recover almost the same result using only Proposition~\ref{thm:rkhs-distance-lower-bound}, which holds for any translation invariant kernel.
As we will see, this leads to slightly worse constant factors and an additional lower bound on the sample size $n$ in terms of the properties of distribution $\nu$, which specifies the kernel.
Essentially we will repeat the main steps of the proof of Theorem \ref{thm:radial-lower-bound}. 
However, we will use Proposition~\ref{thm:rkhs-distance-lower-bound} instead of direct computations (based on the form of radial kernels) to lower bound 
the RKHS distance between embeddings of Gaussian distributions with the Euclidean distance between their mean vectors.

\begin{appxthm}
\label{thm:radial-lower-bound-indirect}
Let $\mathcal{P}$ be the set of distributions over $\R^d$ whose densities are continuously infinitely differentiable and $k$ be radial on $\bb{R}^d$, i.e.,
\[
k(x,y) = 
\int_0^\infty e^{-t\|x - y\|_2^2}\, d\nu(t),
\] 
where $\nu\in M^b_+([0,\infty))$ 
such that $\mathrm{supp}(\nu) \neq \{0\}$. 
Assume that there exist $0 < t_0 \leq t_1 < \infty$ and $0<\beta < \infty$ such that $\nu([t_0, t_1]) \geq \beta$. Suppose $n \geq 24\frac{ t_1 Z_{\nu}}{\beta t_0}$ where
$Z_{\nu}:= \nu([0,\infty))$.
Then
\[
\inf_{\hat{\theta}_n}
\sup_{P \in \mathcal{P}}
P^n
\left\{
\| \hat{\theta}_n - \mu_k(P) \|_{\Hyp_{k}} \geq 
\frac{1}{50}\sqrt{\frac{1}{2n}\cdot \frac{\beta t_0}{t_1e}\left(1 - \frac{2}{2 + d}\right)}
\right\}
\geq
\frac{1}{5}.
\]
\end{appxthm}

\begin{proof}
We apply Proposition~\ref{thm:rkhs-distance-lower-bound} to the radial kernel $k$.
In order to do so, we need to lower bound the quantity appearing in r.h.s.~of Condition \eqref{eq:psi-condition}, which we do as follows.
We already saw in the proof of Theorem \ref{thm:radial-lower-bound} that in our case $\Lambda_{\psi}$ is absolutely continuous with respect to the Lebesgue measure on $\R^d$ and has the following density: 
\[
\lambda_{\psi}(w)
=
\int_0^{\infty}
\frac{1}{(2 t)^{d/2}}
e^{-\frac{\|w\|^2_2}{4t}}
d\nu(t),
\quad w\in\R^d.
\]
Therefore the r.h.s.~of \eqref{eq:psi-condition} reduces to
\begin{align}
\notag
&
\frac{2}{(2\pi)^{d/2}}
\int_{\R^d}
e^{ - \sigma^2\|w\|_2^2}
\langle e_z, w\rangle^2
\cos\left(
\langle a, w\rangle
\right)
d\Lambda_{\psi}(w)\\
\notag
&=
\frac{2}{(2\pi)^{d/2}}
\int_{\R^d}
e^{ - \sigma^2\|w\|_2^2}
\langle e_z, w\rangle^2
\cos\left(
\langle a, w\rangle
\right)
\left(\int_0^{\infty}
\frac{1}{(2 t)^{d/2}}
e^{-\frac{\|w\|^2_2}{4t}}
d\nu(t)
\right)
dw
\end{align}
\begin{align}
\label{eq:radial-proof-1}
&=
\frac{2}{(2\pi)^{d/2}}
\int_0^{\infty}
\frac{1}{(2 t)^{d/2}}
\underbrace{\int_{\R^d}
e^{ -\frac{1}{2} \bigl(2 \sigma^2 + \frac{1}{2t}\bigr)\|w\|_2^2}
\langle e_z, w\rangle^2
e^{-i
\langle a, w\rangle
}
dw}_{\clubsuit}
d\nu(t),
\end{align}
where we used Euler's formula and Tonelli-Fubini theorem in the last equality. Denoting $\delta := 2 \sigma^2 + \frac{1}{2t}$, we have
%
\begin{align*}
&\clubsuit
=\int_{\R^d}
\exp\left( -\frac{\delta}{2} \sum_{\ell=1}^d w_\ell^2
\right)
\left(
\sum_{j=1}^d (e_z)_j^2 w_j^2
+
\sum_{j \neq \ell } (e_z)_j (e_z)_\ell w_j w_\ell
\right)
\exp\left(
-i
\sum_{\ell = 1}^d a_\ell w_\ell
\right)
dw\\
&=\bigstar+\spadesuit,
\end{align*}
where $$\bigstar:=\sum^d_{j=1}\int_{\R^d}
e^{ -\frac{1}{2} \delta \|w\|_2^2}
(e_z)_j^2w_j^2
e^{-i
\langle a, w\rangle
}
dw$$ and $$\spadesuit:=\sum^d_{j\ne l}\int_{\R^d}
e^{ -\frac{1}{2} \delta \|w\|_2^2}
(e_z)_j(e_z)_lw_jw_l
e^{-i
\langle a, w\rangle
}
dw.$$ 
Note that
\begin{align*}
(e_z)_j^2\int_{\R^d}
e^{ -\frac{1}{2} \delta \|w\|_2^2}
w_j^2
e^{-i
\langle a, w\rangle
}
dw
&=
(e_z)_j^2\left(\prod_{\ell \neq j}
\int_{-\infty}^{\infty}
e^{ -\frac{\delta}{2} w_\ell^2}
e^{
-i
a_\ell w_\ell
}
dw_\ell\right)\\
&\qquad\qquad\qquad\times
\left(\int_{-\infty}^{\infty}
e^{ -\frac{\delta}{2} w_j^2}
w_j^2
e^{
-i
a_j w_j
}
dw_j
\right)\\
&=
(e_z)_j^2\left(
\prod_{\ell \neq j}
\sqrt{\frac{2\pi}{\delta}}
e^{ -\frac{a_\ell^2}{2\delta}}
\right)
\cdot \left(\int_{-\infty}^{\infty}
e^{ -\frac{\delta}{2} w_j^2}
w_j^2
e^{
-i
a_j w_j
}
dw_j
\right),
\end{align*}
where we used Lemma \ref{lemma:Fourier-Gaussian}. It follows from \cite[Theorem 8.22(d)]{F99} that if $g=x^2f\in L^1(\R)$, then $f^\wedge$ is twice differentiable and 
\[
g^\wedge(y) = -\frac{\partial ^2 f^\wedge(y)}{\partial^2 y},
\]
which together with Lemma \ref{lemma:Fourier-Gaussian} shows that
\[
\int_{-\infty}^{\infty}
e^{ -\frac{\delta}{2} w_j^2}
w_j^2
e^{
-i
a_j w_j
}
dw_j
=
\frac{1}{\delta}\sqrt{\frac{2\pi}{\delta}}e^{-\frac{a_j^2}{2\delta}}\left(1 - \frac{a_j^2}{\delta}\right).
\]
Therefore,
we get
\begin{align*}
(e_z)_j^2\int_{\R^d}
e^{ -\frac{1}{2} \delta \|w\|_2^2}
w_j^2
e^{-i
\langle a, w\rangle
}
dw
&=
(e_z)_j^2\left(
\prod_{\ell \neq j}
\sqrt{\frac{2\pi}{\delta}}
e^{ -\frac{a_\ell^2}{2\delta}}
\right)
\sqrt{\frac{2\pi}{\delta}}
\frac{1}{\delta}
e^{ -\frac{a_j^2}{2\delta} }
\left(
1-
\frac{a_j^2}{\delta}
\right)\\
&=
\frac{(e_z)_j^2}{\delta}
\left(\frac{2\pi}{\delta}\right)^{d/2}
\!\!\!\!e^{ -\frac{\|a\|^2_2}{2\delta}}
\left(
1-
\frac{a_j^2}{\delta}
\right).
\end{align*}
Summing over $j=1,\dots,d$ we get
\begin{align*}
\bigstar=\sum_{j=1}^d(e_z)_j^2\int_{\R^d}
e^{ -\frac{1}{2} \delta \|w\|_2^2}
w_j^2
e^{-i
\langle a, w\rangle
}
dw
&=
\frac{1}{\delta}
\left(\frac{2\pi}{\delta}\right)^{d/2}
\!\!\!\!e^{ -\frac{\|a\|^2_2}{2\delta}}
-
\sum_{j=1}^d
\frac{(e_z)_j^2a_j^2}{\delta^2}
\left(\frac{2\pi}{\delta}\right)^{d/2}
\!\!\!\!e^{ -\frac{\|a\|^2_2}{2\delta}}.
\end{align*}
Next, for any $ j\neq \ell$ we compute
\begin{align*}
&\int_{\R^d}
\exp\left( -\frac{\delta}{2} \sum_{\ell=1}^d w_\ell^2
\right)
(e_z)_j (e_z)_\ell w_j w_\ell
\exp\left(
-i
\sum_{\ell = 1}^d a_\ell w_\ell
\right)
dw\\
&=
(e_z)_j (e_z)_\ell\left(
\prod_{q \not\in \{j,\ell\}}
\int_{-\infty}^{\infty}
e^{ -\frac{\delta}{2} w_q^2}
e^{
-i
a_q w_q
}
dw_q\right)
\left(
\prod_{q \in \{j,\ell\}}
\int_{-\infty}^{\infty}
e^{ -\frac{\delta}{2} w_q^2}
w_q
e^{
-i
a_q w_q
}
dw_q
\right)\\
&=
(e_z)_j (e_z)_\ell
\left(
\prod_{q \not\in \{j,\ell\}}
\sqrt{\frac{2\pi}{\delta}}
e^{ -\frac{a_q^2}{2\delta}}
\right)
\left(
\prod_{q \in \{j,\ell\}}
\sqrt{\frac{2\pi}{\delta}}
\frac{i a_q}{\delta}e^{-\frac{a_q^2}{2\delta}}
\right)\\
&=
(e_z)_j (e_z)_\ell
\left(\frac{2\pi}{\delta}\right)^{d/2}
e^{ -\frac{\|a\|^2_2}{2\delta}}
\left(-\frac{a_j a_\ell}{\delta^2}\right).
\end{align*}
Summing over $ j \neq \ell$ we get
\begin{align*}
\spadesuit=-\frac{\langle e_z, a\rangle^2}{\delta^2}
\left(\frac{2\pi}{\delta}\right)^{d/2}
e^{ -\frac{\|a\|^2_2}{2\delta}}
+
\sum_{j=1}^d\frac{ (e_z)_j^2 a_j^2}{\delta^2}
\left(\frac{2\pi}{\delta}\right)^{d/2}
e^{ -\frac{\|a\|^2_2}{2\delta}}.
\end{align*}
Returning to \eqref{eq:radial-proof-1}, we get
\begin{align*}
&\frac{2}{(2\pi)^{d/2}}
\int_{\R^d}
e^{ - \sigma^2\|w\|_2^2}
\langle e_z, w\rangle^2
\cos\left(
\langle a, w\rangle
\right)
d\Lambda_{\psi}(w)\\
&=
\frac{2}{(2\pi)^{d/2}}
\int_0^{\infty}
\frac{1}{(2 t)^{d/2}}
e^{ -\frac{\|a\|^2_2}{2\delta}}
\left(\frac{2\pi}{\delta}\right)^{d/2}
\frac{1}{\delta}
\left(
1
-
\frac{\langle e_z, a\rangle^2}{\delta}
\right)
d\nu(t)\\
&=
4
\int_0^{\infty}
\exp\left(
 -\frac{1}{2} \frac{2t\|a\|^2_2}{4\sigma^2 t + 1}
 \right)
\frac{t}{(4\sigma^2 t + 1)^{1+d/2}}
\left(
1
-
\frac{2t\langle e_z, a\rangle^2}{4\sigma^2 t + 1}
\right)
d\nu(t).
\end{align*}
In order to apply Proposition~\ref{thm:rkhs-distance-lower-bound} we need to lower bound the following value, appearing in Condition \eqref{eq:psi-condition}:
\begin{align}
\notag
\Delta(a) &:=\min_{z\in \R^d\setminus \{0\}}
\frac{2}{(2\pi)^{d/2}}
\int_{\R^d}
e^{ - \sigma^2\|w\|_2^2}
\langle e_z, w\rangle^2
\cos\left(
\langle a, w\rangle
\right)
d\Lambda_{\psi}(w)\\
\notag
&=
4
\int_0^{\infty}
\exp\left(
 -\frac{1}{2} \frac{2t\|a\|^2_2}{4\sigma^2 t + 1}
 \right)
\frac{t}{(4\sigma^2 t + 1)^{1+d/2}}
\left(
1
-
\frac{2t\|a\|^2_2}{4\sigma^2 t + 1}
\right)
d\nu(t).
\end{align}

Next we will separately treat two different cases.

{\bf Case 1:} $d > 2$.
Note that the function $\rho(t) = {t}{(4\sigma^2 t + 1)^{-(d+2)/2}}$ is positive and bounded on $[0,\infty)$ for any $d>0$.
Thus, we can define a non-negative and finite measure $\tilde{\tau}$, absolutely continuous with respect to $\nu$ with density $\rho(t)$.
If we denote $Z_{\tau} := \int_{0}^{\infty} 1 \,d\tilde{\tau}(t)$ and write $\tau$ for the normalized version of $\tilde{\tau}$, then we can rewrite
\begin{align*}
\Delta(a)
&=
4
\int_0^{\infty}
\exp\left(
 -\frac{1}{2} \frac{2t\|a\|^2_2}{4\sigma^2 t + 1}
 \right)
\left(
1
-
\frac{2t\|a\|^2_2}{4\sigma^2 t + 1}
\right)
d\tilde{\tau}(t)\\
&=
4 \,Z_{\tau}
\E_{t\sim \tau}\left[
\exp\left(
 -\frac{1}{2} \frac{2t\|a\|^2_2}{4\sigma^2 t + 1}
 \right)
\left(
1
-
\frac{2t\|a\|^2_2}{4\sigma^2 t + 1}
\right)
\right]\\
&=
4 \,Z_{\tau}
\E_{t\sim \tau}\left[
\exp\left(
 -\frac{1}{2} \frac{2t\|a\|^2_2}{4\sigma^2 t + 1}
 \right)
\right]
-
4 \,Z_{\tau}
\E_{t\sim \tau}\left[
\exp\left(
 -\frac{1}{2} \frac{2t\|a\|^2_2}{4\sigma^2 t + 1}
 \right)
\frac{2t\|a\|^2_2}{4\sigma^2 t + 1}
\right].
\end{align*}
Note that for $d > 2$, $\E_{t\sim \tau}[|t|]$ is finite, since in this case
$
t \mapsto \frac{t^2}{(4\sigma^2 t + 1)^{(d+2)/2}}
$
is bounded and $\nu$ is a finite measure.
Denote $\mu_{\tau} := \E_{t\sim \tau}[t]$ and note that $
t \mapsto \exp\left(
 -\frac{1}{2} \frac{2t\|a\|^2_2}{4\sigma^2 t + 1}
 \right)$
is a convex function on $[0,\infty)$. 
Thus, for $d>2$ we can use Jensen's inequality to get
\[
\E_{t\sim \tau}\left[
\exp\left(
 -\frac{1}{2} \frac{2t\|a\|^2_2}{4\sigma^2 t + 1}
 \right)
\right]
\geq
\exp\left(
 -\frac{1}{2} \frac{2\mu_\tau\|a\|^2_2}{4\sigma^2 \mu_\tau + 1}
 \right).
\]
Also note that
\[
-
4 \,Z_{\tau}
\E_{t\sim \tau}\left[
\exp\left(
 -\frac{1}{2} \frac{2t\|a\|^2_2}{4\sigma^2 t + 1}
 \right)
\frac{2t\|a\|^2_2}{4\sigma^2 t + 1}
\right]
\geq
4 \,Z_{\tau}
\E_{t\sim \tau}\left[
-\frac{2t\|a\|^2_2}{4\sigma^2 t + 1}
\right]
\geq
-4 \,Z_{\tau}
\frac{2\mu_\tau\|a\|^2_2}{4\sigma^2 \mu_\tau + 1},
\]
where we used inequality $e^{-x}\leq 1$, which holds for $x\geq 0$, together with Jensen's inequality and the fact that
$
t \mapsto -\frac{2t\|a\|^2_2}{4\sigma^2 t + 1}
$
is concave on $[0,\infty)$.
Summarizing, we have
\begin{align*}
\Delta(a)
&\geq
4 Z_{\tau}\left( \exp\left(
 -\frac{1}{2} \frac{2\mu_\tau\|a\|^2_2}{4\sigma^2 \mu_\tau + 1}
 \right)
 - 
 \frac{2\mu_\tau\|a\|^2_2}{4\sigma^2 \mu_\tau + 1}\right)\\
& \geq
4 Z_{\tau}\left( 1
 -\frac{1}{2} \frac{2\mu_\tau\|a\|^2_2}{4\sigma^2 \mu_\tau + 1}
 - 
 \frac{2\mu_\tau\|a\|^2_2}{4\sigma^2 \mu_\tau + 1}\right) \\
 &=
2 Z_{\tau}\left( 2
 -3 \frac{2\mu_\tau\|a\|^2_2}{4\sigma^2 \mu_\tau + 1}\right),
\end{align*}
where we used a simple inequality $e^x \geq 1 + x$.
If the following condition is satisfied:
\begin{equation}
\label{eq:condition-mu}
\frac{2\mu_\tau\|a\|^2_2}{4\sigma^2 \mu_\tau + 1}
\leq \frac{1}{3},
\end{equation}
then we get
\begin{equation}
\label{eq:appendix-proof-1}
\Delta(a) \geq 2Z_{\tau} = 2 \int_{0}^{\infty} \frac{t}{(4\sigma^2 t + 1)^{1+d/2}} d\nu(t).
\end{equation}
Together with Proposition~\ref{thm:rkhs-distance-lower-bound} this leads to the following lower bound, which holds for any $\mu_0, \mu_1\in\R^d$ and $\sigma^2 > 0$ satisfying \eqref{eq:condition-mu} with $a:= \mu_0 - \mu_1$:
\[
\| \theta_0 - \theta_1\|^2_{\Hyp_{k}}
\geq
\int_{0}^{\infty} \frac{t\|\mu_0 - \mu_1\|^2_2}{(4\sigma^2 t + 1)^{1+d/2}} d\nu(t),
\]
where $\theta_0$ and $\theta_1$ are KME's of Gaussian measures $G(\mu_0,\sigma^2 I)$ and $G(\mu_1,\sigma^2 I)$ respectively.
Note that this lower bound is identical to the one in \eqref{eq:rkhs-distance-lower}, which we obtained using direct analysis for the radial kernels.
However, condition \eqref{eq:proof-assumption-2} is now replaced with the stronger one in \eqref{eq:condition-mu}.
We can now repeat the proof of Theorem~\ref{thm:radial-lower-bound} starting from inequality \eqref{eq:rkhs-distance-lower} and making sure that condition \eqref{eq:condition-mu} 
is satisfied when we choose constants appearing in definitions of $\mu_0, \mu_1$ and $\sigma^2$.

In order to check condition \eqref{eq:condition-mu} we need to upper bound the expectation $\mu_\tau$.
It is easily seen that for $d>2$, $
t \mapsto \frac{t^2}{(4\sigma^2 t + 1)^{(d+2)/2}}
$
achieves its maximum on $[0,\infty)$ for $t^* = \frac{1}{\sigma^2(d-2)}$.
Using this fact, denoting $Z_{\nu} = \int_{0}^{\infty} 1 \, d\nu(t)$, and setting $\sigma^2 = \frac{1}{2t_1d}$ we get
\begin{align*}
\mu_\tau &= 
\frac{1}{Z_{\tau}} \int_{0}^{\infty} \frac{t^2}{(4\sigma^2 t + 1)^{1+d/2}} d\nu(t)
\leq
\frac{Z_{\nu}}{Z_{\tau}} \frac{(t^*)^2}{(4\sigma^2 t^* + 1)^{1+d/2}}\\ 
&=
\frac{4t_1^2Z_{\nu}}{Z_{\tau}} \frac{ d^2}{(d-2)^2}\frac{1}{(\frac{4}{d-2} + 1)^{1+d/2}} \\
&=
\frac{4t_1^2Z_{\nu}}{Z_{\tau}} \left(1 + \frac{2}{d-2} \right)^2 \left(\left(1 - \frac{4}{d+2} \right)^{(d+2)/4}\right)^2 \leq
\frac{4t_1^2Z_{\nu}}{Z_{\tau}e^2} \left(1 + \frac{2}{d-2} \right)^2.
\end{align*}
We may finally use \eqref{eq:integral-d-lower-bound-final} to get
\[
Z_{\tau}
\geq
\frac{\beta t_0}{e}\left(1 - \frac{2}{d+2}\right),
\]
which leads to the following upper bound on the expectation $\mu_\tau$:
\[
\mu_\tau
\leq
\frac{4t_1^2Z_{\nu}}{\beta t_0 e}\frac{d(d+2)}{(d-2)^2}. 
\]
This upper bound shows that the condition \eqref{eq:condition-mu} is satisfied if the following holds:
\begin{align}
\label{eq:condition-worse}
t_1\|a\|^2_2
\leq
\frac{2}{3}t_1\sigma^2
+
\frac{\beta t_0 e}{24 t_1 Z_{\nu}}\frac{(d-2)^2}{d(d+2)}.
\end{align}
We conclude the proof by repeating the remaining steps of the proof of Theorem~\ref{thm:radial-lower-bound} and replacing condition \eqref{eq:proof-assumption-2} on the value 
$\|\mu_0 - \mu_1\|^2_2$ with \eqref{eq:condition-worse} specified to $a =\mu_0 - \mu_1$.

\medskip
{\bf Case 2:} $d \leq 2$.
We can use a simple inequality $e^{-x/2}(1 - x) \geq 1 - 3x/2$ which holds for any $x$ and get the following lower bound:
\[
\Delta(a) \geq
4
\int_0^{\infty}
\frac{t}{(4\sigma^2 t + 1)^{1+d/2}}
\left(
1
-
\frac{3t\|a\|^2_2}{4\sigma^2 t + 1}
\right)
d\nu(t).
\]
Assuming $\|a\|^2_2 \leq \sigma^2$ we further get
\[
\left(
1
-
\frac{3t\|a\|^2_2}{4\sigma^2 t + 1}
\right)
\geq
\left(
1
-
\frac{3t\sigma^2}{4\sigma^2 t + 1}
\right)
\geq
\left(
1
-
\frac{3t\sigma^2}{4\sigma^2 t}
\right)
=
\frac{1}{4}
\]
and as a consequence, we also get
\[
\Delta(a)
\geq
\int_0^{\infty}
\frac{t}{(4\sigma^2 t + 1)^{1+d/2}}
d\nu(t),
\]
which coincides with \eqref{eq:appendix-proof-1} up to an additional factor of $2$.
We can now repeat all the steps for the previous case, and it is also easy to check that in this case $\|\mu_0 - \mu_1\|^2_2 \leq \sigma^2$ will be indeed satisfied.
This concludes the proof.
\end{proof}
\vspace{-7mm}
\begin{appxrem}
This result should be compared to Theorem \ref{thm:radial-lower-bound}, which was based on the direct analysis for radial kernels.
We see that apart from an extra factor $2$ appearing under the square root in the lower bound, Theorem \ref{thm:radial-lower-bound-indirect} also requires a superfluous condition on the minimal sample size $n$, which depends on properties of $\nu$. For instance, for Gaussian kernel with $\nu$ concentrated on a single point $\frac{1}{2\eta^2}$ for some $\eta^2 > 0$, the result holds as long as $n \geq 24$, because in this case we can take $t_0=t_1=\frac{1}{2\eta^2}$ and $\beta = 1$.
However, other choices of $\nu$ may lead to quite restrictive lower bounds on $n$.\vspace{-2mm}
\end{appxrem}
\begin{appxrem}
Conceptually, the main difference between the proofs of Theorems \ref{thm:radial-lower-bound} and \ref{thm:radial-lower-bound-indirect}
lies in the way we lower bound the RKHS distance between embeddings of Gaussian measures with the Euclidean distance between their mean vectors.
In Theorem~\ref{thm:radial-lower-bound} we derived a closed-form expression for the RKHS distance in \eqref{eq:rkhs-distance-radial-closed} and then lower bounded it directly using the 
properties specific to its form.
On the other hand, in Theorem~\ref{thm:radial-lower-bound-indirect} we resorted to the lower bound of Lemma~\ref{thm:rkhs-distance-lower-bound}, which holds for any translation invariant kernel and hence is less tight.
\end{appxrem}

\bibliography{17-032}

\end{document}